\documentclass{amsart}
 
\usepackage{xy}
\input xy
\xyoption{all}
\newcommand{\scr}[1]{\mathscr{#1}}
\usepackage{Definitions}
\usepackage{hyperref}
\hypersetup{
 hidelinks,colorlinks=false,breaklinks=true,
 pdfkeywords={equivariant homotopy theory, K-theory, tensor triangulated categories, Quillen's F-isomorphism theorem, group cohomology, descent},
 pdfauthor={Akhil Mathew, Niko Naumann, Justin Noel},
}
\usepackage{Environments2}
\usepackage{PageSetup}
\usepackage{verbatim}
\usepackage{amssymb}
\usepackage{datetime}
\usepackage{enumitem} %Used for individualized enumerate labellings 
\usepackage[left=1.3in, right=1.3in, top=1.2in, bottom=1.2in]{geometry}
\usepackage[toc]{appendix}

\newcommand{\Sp}{\mathrm{Sp}}
\renewcommand{\sp}{\mathrm{Sp}}
\newcommand{\md}{\mathrm{Mod}}
\newcommand{\CoInd}{\mathrm{Coind}}
\newcommand{\Coind}{\mathrm{Coind}}
\newcommand{\Res}{\mathrm{Res}}
\newcommand{\fun}{\mathrm{Fun}}
\usepackage{url}

\renewcommand{\hom}{\mathrm{Hom}}
\theoremstyle{definition}
\newtheorem{cons}[thm]{Construction}
\newtheorem*{thmstar}{Theorem}
\newtheorem*{defstar}{Definition}

\newcommand{\res}{\mathrm{Res}}

\newcommand{\gsp}[1]{\mathrm{Sp}_{#1}}

\newcommand{\cato}{\mathrm{Cat}_\infty^{\otimes}}

\newcommand{\bor}[1]{\underline{#1}}
\newcommand{\sFNil}{\sF^{\textrm{Nil}}}

\newcommand{\GSpec}{\mathrm{Sp}_G}
\newcommand{\GS}{\mathcal{S}_G}
\newcommand{\GSpt}{\mathcal{S}_{G*}}

\newcommand{\HSpec}{\mathrm{Sp}_H}

\newcommand{\e}[1]{\mathbb{E}_{#1}}

\renewcommand{\theenumi}{\arabic{enumi}}

\newcommand{\triplearrows}{\begin{smallmatrix} \to \\ \to \\ 
\to \end{smallmatrix} }

\newcommand{\orsg}{\mathrm{OrthSpec}_{G}}
\newcommand{\topgp}[1]{\mathrm{Top}_{\ast, #1}}
\newcommand{\topgpu}[1]{\mathrm{Top}_{ #1}}
\newcommand{\CAlg}{\mathrm{CAlg}}
\newcommand{\clg}{\CAlg}
 \newcommand{\Ator}{\mathcal{C}_{A\mathrm{-tors}}}
\newcommand{\Acp}{\mathcal{C}_{A\mathrm{-cpl}}}
\newcommand{\cbaug}{\mathrm{CB}^\bullet_{\mathrm{aug}}}
\newcommand{\cb}{\mathrm{CB}^\bullet}

\newcommand{\Gbs}{(\GSpec)_{\mathrm{Borel}}}
\newtheorem{hypotheses}[thm]{Hypotheses}

\makeatletter
\renewcommand{\p@enumii}{\theenumi.}
\makeatother
\begin{document}
	\title{Nilpotence and descent in equivariant stable homotopy theory}	\author{Akhil Mathew}
	\address{Harvard University, Cambridge, MA 02138}
	\email{amathew@math.harvard.edu}
	\urladdr{http://math.harvard.edu/~amathew/}

	\author{Niko Naumann}
	\address{University of Regensburg\\
	NWF I - Mathematik; Regensburg, Germany}
	\email{Niko.Naumann@mathematik.uni-regensburg.de}
	\urladdr{http://homepages.uni-regensburg.de/~nan25776/}

	\author{Justin Noel}
	\address{University of Regensburg\\
	NWF I - Mathematik; Regensburg, Germany}
	\email{justin.noel@mathematik.uni-regensburg.de}
	\urladdr{http://nullplug.org}
	\thanks{The first author was supported by the NSF Graduate Research Fellowship under grant DGE-110640. The second author was partially supported by the SFB 1085 - Higher Invariants, Regensburg. The third author was partially supported by the DFG grants: NO 1175/1-1 and SFB 1085 - Higher Invariants, Regensburg.}

\date{\today}

\begin{abstract}
 Let $G$ be a finite group and let $\sF$ be a family of subgroups of $G$. We introduce
a class of $G$-equivariant spectra that we call \emph{$\sF$-nilpotent}. 
This definition fits into the general theory of torsion, complete, and
nilpotent objects in a symmetric monoidal stable $\infty$-category, with which we
begin.
We then develop some
of the basic properties of $\sF$-nilpotent $G$-spectra, which are explored further in the sequel to this paper. 

In the rest of the paper,  we prove several general structure theorems for $\infty$-categories of module spectra
over objects such as equivariant real and complex $K$-theory and Borel-equivariant
$MU$.
Using these structure theorems and a technique with the flag variety dating
back to Quillen, 
we
then show that large classes of equivariant cohomology theories for which a
type of complex-orientability holds are nilpotent for the family of
abelian subgroups.  
In particular, we prove that equivariant real and complex $K$-theory, as well
as the Borel-equivariant versions of complex-oriented theories, have this
property.
\end{abstract}

\maketitle

\tableofcontents
\section{Introduction}

\subsection{Quillen's theorem}
This is the first in a series of papers whose goal is to investigate
certain phenomena in equivariant stable homotopy theory revolving around the
categorical notion of \emph{nilpotence}. 
Our starting point is the classical theorem of Quillen \cite{Qui71b} on the
cohomology of a finite group $G$, which describes $H^*(BG; k)$  for $k$ a
field of characteristic $p > 0$  up to a relation called $\mathcal{F}$-isomorphism.

This result is as follows. Given a cohomology
class $x \in H^r(BG; k)$, it determines for each elementary abelian
$p$-subgroup $A \leqslant G$, a cohomology class $x_A \in H^r(BA; k)$ via restriction. These classes $\left\{x_A\right\}_{A \leq G}$ are not arbitrary; they satisfy the following two basic relations:
\begin{enumerate}
	\item If $A, A' \leqslant G$ are a pair of elementary abelian $p$-subgroups
	which are conjugate by an element $g \in G$, then $x_A$ maps to $x_{A'}$
	under the isomorphism $H^*(BA; k) \cong H^*(BA'; k)$ induced by conjugation by $g$.

	\item If $A \leqslant A'$ is an inclusion of elementary abelian subgroups of
	$G$, then $x_{A'}$ maps to $x_A$ under the restriction map $H^*(BA'; k) \to
	H^*(BA; k)$.
\end{enumerate}

Let $E_p(G)$ denote the family of elementary abelian $p$-subgroups of $G$ and consider the subring 
\[R \subseteq \prod_{A \in E_p(G)} H^*(BA; k)\] 
of all tuples $\left\{x_A \in H^*(BA; k)\right\}_{A \in E_p(G)}$
which satisfy the two conditions above.  The product of the restriction maps
lifts to a ring homomorphism
\begin{equation} \label{eq:qmap} 
	H^*(BG; k) \stackrel{\psi}{\to} R, \quad x \mapsto \left\{x_A\right\}_{A \in
	E_p(G)}.  
\end{equation}

Quillen's $\cF$-isomorphism theorem states roughly that \eqref{eq:qmap} is an isomorphism modulo nilpotence.
More precisely:
\begin{thm}[{\cite[Theorem~7.1]{Qui71b}}] \label{thm:quillst}
\label{quillst}
	The map $\psi$ is a uniform $\cF_p$-isomorphism: in other words, there exist integers $m$ and $n$ such that 
	\begin{enumerate}
		\item For every $x\in \ker \psi$, $x^{m}=0$. 
		\item For every $x \in R$, $x^{p^n}$ belongs to the image of $\psi$.
	\end{enumerate}
\end{thm} 

\Cref{thm:quillst} establishes the fundamental role of \emph{elementary abelian}
groups in the cohomology of finite groups, and is extremely useful in
calculations, especially since there are large known classes of groups for
which \eqref{eq:qmap} is an injection (or at least an injection when one uses
the larger class of all abelian subgroups); see for instance
\cite[Prop. 3.4, Cor. 3.5]{Qui71e}. 
Since the cohomology of elementary abelian groups is known, \Cref{thm:quillst}
enables one to, for example, determine the \emph{Krull dimension} of $H^*(BG;
k)$ \cite[Theorem~7.7]{Qui71b}.

\subsection{Descent up to nilpotence}
\Cref{thm:quillst} by itself is a computational result about cohomology. However, as the authors
learned from \cite{carlson, balmer}, it can be interpreted as a consequence of
a more precise \emph{homotopical} statement. In the homotopy theory $\fun(BG,
\md(k))$ of $k$-module spectra equipped with a $G$-action (equivalently, the
derived category of $k[G]$-modules), the commutative algebra objects $\{k^{G/A}\}_{A
\in E_p(G)}$ satisfy a type of \emph{descent up to nilpotence:} more precisely,
the thick $\otimes$-ideal they generate is all of $\fun(BG, \md(k))$. 
From this, using a descent type spectral sequence, it is not too difficult to
extract \Cref{thm:quillst} (compare \cite[\S 4.2]{galois}). 
However, the descent-up-to-nilpotence statement is much more precise and has
additional applications. 

The purpose of these two papers is, first, to formulate a general 
categorical definition that encompasses the Carlson-Balmer interpretation of
\Cref{thm:quillst}. Our categorical definition lives in the world of 
\emph{genuine} equivariant stable homotopy theory, and, for a finite
group $G$, isolates a class of
$G$-equivariant spectra for which results such as \Cref{thm:quillst}
hold with respect to a given family of subgroups. 
The use of genuine $G$-equivariant theories allows for additional applications.
For instance, our application to equivariant complex $K$-theory gives a
homotopical lifting of Artin's theorem and gives a categorical explanation of
results of Bojanowska \cite{Boj83, Boj91} and
Bojanowska-Jackowski \cite{BojJac} on equivariant $K$-theory of finite
groups. This application, which relies on an analysis of the descent spectral
sequence, will appear in the second paper \cite{MNN15i}. 
In addition, the methods of $\sF$-nilpotence can be applied to equivariant
versions of algebraic $K$-theory, which leads to Thomason-style descent theorems
in the algebraic $K$-theory of ring spectra. We will return to this in a 
third paper \cite{CMNN15iii}. 
 
We will give numerous examples of equivariant cohomology theories that
fulfill this criterion.
The specialization to Borel-equivariant mod $p$ cohomology 
will recover results such as \Cref{thm:quillst}, as well as versions
of \Cref{thm:quillst} where $k$ is replaced by any complex-oriented theory.
%In fact, our work on this project began with an observation that for a Morava
%$E$-theory at height $n$,
%one can prove an analog of the first part of \Cref{thm:quillst} for the family
%of \emph{abelian} subgroups of rank at most $n$ using the main result
%of \cite{MNN}. 
Indeed, the second purpose of these papers is to prove $\mathcal{F}$-isomorphism
theorems generalizing \Cref{thm:quillst}, using a careful analysis of the
relevant descent spectral sequences.

\subsection{$\sF$-nilpotence}

We now summarize the contents of this paper. 
The present paper is almost exclusively theoretical, and the computational
results (i.e., analogs of \Cref{thm:quillst}) will be the focus of the
sequel \cite{MNN15i}, so we refer to the introduction of the sequel for further discussion.

Let $\mathcal{C}$ be a presentable stable $\infty$-category with a compatible symmetric monoidal
structure, i.e., such that $\otimes$ preserves colimits in each variable. 
Given an algebra object $A$ of $\mathcal{C}$, one says, following Bousfield,
that an object of $\mathcal{C}$ is
\emph{$A$-nilpotent} if it belongs to the thick $\otimes$-ideal generated by
$A$. 

The following is 
the main definition of this series of papers. 

\begin{defstar}[See \Cref{Fnildef} below]
Let $G$ be a finite group, and let  $\GSpec$
denote the $\infty$-category of $G$-spectra (see \Cref{def:G_spectra}).
Let $\sF$ be a family of subgroups of $G$. 
We say that $M \in \GSpec$ is \emph{$\sF$-nilpotent}
if it is nilpotent with respect to the algebra object $\prod_{H \in \sF}
F(G/H_+, S^0_G) \in \mathrm{CAlg}( \GSpec)$.
\end{defstar}

We will especially be interested in this definition for a ring $G$-spectrum
$R$ (up to homotopy, not necessarily structured).
In this case, we will see that $R$ is $\sF$-nilpotent if and only if the
geometric
fixed points $\Phi^H R$ are contractible for any subgroup $H \leq G$ which does not
belong to $\sF$ (\Cref{geofixedpointscontractible}). 
In the sequel to this paper, we will show that if $R \in \GSpec$ is a ring
$G$-spectrum which is $\sF$-nilpotent, then the $R$-cohomology of any $G$-space
satisfies an analog of \Cref{thm:quillst} (with the elementary abelian
subgroups replaced by those subgroups in $\sF$).

The first goal of this paper is to develop the theory of nilpotence in an
appropriately general context. We have also taken the opportunity to discuss certain
general features of symmetric monoidal stable $\infty$-categories, such as a
general version of Dwyer-Greenlees theory \cite{DwG02}, due to
Hovey-Palmieri-Strickland \cite{HPS97}, yielding an
equivalence between complete and torsion objects. 
Similar ideas have also been
explored in recent work of Barthel-Heard-Valenzuela \cite{BHV}.
This material is largely expository, but certain aspects (in particular,
decompositions such as \Cref{Acplimit} and \Cref{Afracsquare} below) rely on the theory of
$\infty$-categories and have not always been documented in the classical
literature on triangulated categories. 
Our presentation is intended to make it clear that the notion of $\sF$-nilpotence is a natural
generalization  of a \emph{bounded torsion} condition.

Let $R \in \mathrm{Alg}(\GSpec)$ be an associative algebra, and suppose that
$R$ is $\sF$-nilpotent. A major consequence of $\sF$-nilpotence is an
associated decomposition (\Cref{decompFnilpmodules}) of
the $\infty$-category of $R$-module $G$-spectra. 
\begin{thmstar} 
\label{Decompos}
Suppose $R \in \mathrm{Alg}(\GSpec)$ is $\sF$-nilpotent. 
Let $\mathcal{O}_{\sF}(G)$ be the category of $G$-sets of the form
$G/H$, $H \in \sF$.
Then 
there is an equivalence of $\infty$-categories
\[  \md_{\GSpec}(R) \simeq \varprojlim_{G/H \in \mathcal{O}_{\sF}(G)^{op}}
\md_{\HSpec}( \Res^G_H R).  \]
If $R$ is an $\e{\infty}$-algebra object, then the above equivalence is
(canonically) symmetric monoidal too. 
\end{thmstar} 
We expect the decomposition given above to have future applications, as in
most practical situations where $\sF$-nilpotence arises, it is easier to study
modules in $\HSpec$ over $\Res^G_H R$ (for $H \in \sF$) than to study modules over $R$ itself. 

\subsection{Equivariant module spectra}

In the rest of this paper, we take a somewhat different direction, albeit
with a view towards proving $\sF$-nilpotence results. 
We analyze the structure of \emph{modules}
over certain equivariant ring spectra.  These results  generalize work of Greenlees-Shipley \cite{GS11}
in the rational setting. 

Our first results concern the structure of the $\infty$-category $\fun(BG,
\md(R))$ where $R$ is a complex-oriented $\e{\infty}$-ring and $G$ is a
connected compact
Lie group. The application of these results to $\sF$-nilpotence statements 
will come from embedding a finite group in a unitary group.  In case $G$ is a product of copies of tori or unitary groups (see
\Cref{basicunipex}  for precise conditions), we describe  
$\fun(BG,
\md(R))$
as an $\infty$-category of \emph{complete} modules over a (non-equivariant) ring
spectrum. 
For instance, we prove the following result.

\begin{thm} 
\label{kduality}
Let $R$ be an even periodic $\e{\infty}$-ring. Then we have an equivalence
of symmetric monoidal $\infty$-categories,
\[ \fun(BU(n), \md(R)) \simeq \md( F(BU(n)_+, R))_{\mathrm{cpl}},  \]
where on the right we consider modules complete with respect to the
augmentation ideal in $\pi_0 (F(BU(n)_+, R)) \simeq \pi_0(R)[[c_1, \dots, c_n]]$.
\end{thm} 

One can think of \Cref{kduality} as a homotopy-theoretic
(complex-orientable) version of the
\emph{Koszul duality} between DG modules over an exterior 
algebra (which is replaced by the group algebra $R \wedge U(n)_+$) and DG modules over a polynomial 
algebra (which is replaced by $F(BU(n)_+, R)$). 
Rationally, these results are due to Greenlees-Shipley \cite{GS11}. 
\Cref{kduality} is useful because it is generally much
easier to work with modules over the non-equivariant ring spectrum $F(BU(n)_+,
R)$
than to analyze $U(n)$-actions directly.

The unitary group is especially well-behaved because its cohomology is
torsion-free.  A more general result (\Cref{polyEMSS} below) runs as follows: 
\begin{thm} \label{polyEMSSintro}
Let $R$ be an $\e{\infty}$-ring  and let $G$ be  a compact, connected
Lie group. Suppose 
$H^*(BG; \pi_0 R)\simeq H^*(BG; \mathbb{Z})\otimes_{\mathbb{Z}}\pi_0 R$ and
that this is a polynomial ring over $\pi_0 R$; suppose furthermore that the
cohomological $R$-based Atiyah-Hirzebruch spectral sequence (AHSS) for $BG$ degenerates (e.g., $\pi_*(R)$ is
torsion-free).
Then there is an equivalence of symmetric monoidal $\infty$-categories between  $\mathrm{Fun}(BG, \md(R))$
and the symmetric monoidal $\infty$-category of $R$-complete $F(BG_+, R)$-modules.
\end{thm}

Let $(\mathcal{C}, \otimes, \mathbf{1})$ be a presentable, symmetric monoidal $\infty$-category where
the tensor product commutes with colimits in each variable. Then there is an
adjunction of symmetric monoidal $\infty$-categories between  $\md(
\mathrm{End}( \mathbf{1}))$ and $\mathcal{C}$, and we will discuss general criteria for this
adjunction to be a {localization.} We call such $\mathcal{C}$ \emph{unipotent.}  When applied to $\infty$-categories of
the form $\fun(BG, \md(R))$, these general criteria will recover results such as
\Cref{polyEMSSintro}. 

We will 
then  explain that results such as \Cref{kduality} lead to very quick  and
explicit proofs
(via the flag variety) 
of results including the following: 

\begin{thm} 
Let $R$ be a complex-orientable $\e{\infty}$-ring. Then, if $G$ is a finite
group, the Borel-equivariant $G$-spectrum associated to $R$ is $\sF$-nilpotent for
$\sF$ the family of abelian subgroups. 
\end{thm}

We will prove more precise results for particular complex-orientable theories
in the sequel to this paper. 
The main observation is that the (very nontrivial) action of the unitary group
$U(n)$ on the flag variety $F = U(n)/T$ becomes trivialized after
smashing with a complex-oriented theory; the trivialization is a consequence
of \Cref{kduality} (though can also be proved by hand). The use of the flag variety in this setting is of course 
classical, and the argument is essentially due to Quillen (albeit
stated in a slightly different form).

Finally, we shall treat the cases of equivariant real and complex $K$-theory.
Here again, we make a study of their module categories in the case of
compact, connected Lie groups. 
Our main result is that, once again, 
under certain conditions the symmetric monoidal $\infty$-category of modules (in equivariant spectra) over equivariant real and
complex $K$-theory can be identified with 
the symmetric monoidal $\infty$-category of modules over a \emph{non-equivariant}
$\e{\infty}$-ring spectrum.

\begin{thm} 
Let $G$ be a compact, connected Lie group 
with $\pi_1(G)$ torsion-free.  Then the respective symmetric monoidal
$\infty$-categories $\md_{\GSpec}(KU_G)$ and $ \md_{\GSpec}(KO_G)$
are equivalent to the symmetric monoidal $\infty$-categories of modules (in the $\infty$-category of spectra) over the
categorical fixed points of $KU_G$ and $
KO_G$ respectively.
\end{thm} 

For equivariant complex $K$-theory, these results use (and give a modern
perspective on) the theory of
K\"unneth spectral sequences in equivariant $K$-theory developed by Hodgkin 
\cite{Hod75}, Snaith \cite{Sna72}, and McLeod \cite{Mc78}. 
By embedding a finite group in a unitary group, one obtains a quick proof that
equivariant $K$-theory  is nilpotent for the family of abelian subgroups; in the sequel
we shall see that it is actually nilpotent for the family of cyclic subgroups. 
The condition that $\pi_1(G)$ should be torsion-free does not rule out
torsion in $H^*(G; \mathbb{Z})$ (e.g., $G = \mathrm{Spin}(n)$) and that the
conclusion holds in these cases is a special feature of $K$-theory. 

To obtain the result for equivariant \emph{real} $K$-theory, we prove a version of the 
theorem of Wood $KO \wedge \Sigma^{-2} \mathbb{CP}^2 \simeq KU$ in the equivariant
setting  (\Cref{equivwood}) below. 

\begin{thm} 
Let $G$ be any compact Lie group. Then there is an equivalence of
$KO_G$-modules $KO_G \wedge \Sigma^{-2} \mathbb{CP}^2 \simeq KU_G$. 
\end{thm} 
Here
$\mathbb{CP}^2$ is considered  as a pointed space with trivial $G$-action.
We then develop an analogous $\mathbb{Z}/2$-Galois descent picture from
equivariant complex to real $K$-theory (which is due to Rognes \cite{Rog08}
for $G = 1$). In particular, we show that (for any compact Lie group $G$) the
map $KO_G \to KU_G$ of $\e{\infty}$-algebras in $\GSpec$ is  a faithful
$\mathbb{Z}/2$-Galois extension. The Galois descent or homotopy fixed point
spectral sequence is carefully analyzed; here the trichotomy of irreducible
representations into real, complex, and quaternionic plays a fundamental role.

\subsection*{Notation}
We will freely use the theory of $\infty$-categories (quasi-categories)
as treated in \cite{Lur09} and the theory of symmetric monoidal
$\infty$-categories, as well as that of rings and modules in them,  developed in \cite{Lur14}. 
Note that we will identify the $\e{1}$ and associative $\infty$-operads. 
In a symmetric monoidal $\infty$-category, we will let $\mathbb{D}A$ denote the
dual of a dualizable object $A$. We refer to \cite[\S 4.6.1]{Lur14} for a
treatment of duality and dualizable objects. 
Homotopy limits and colimits in an $\infty$-category will be written as $\varprojlim$ and $\varinjlim$.
We will abuse notation and often identify an ordinary category $\mathcal{C}$
with the associated quasi-category $N(\mathcal{C})$.
In addition, we will frequently identify abelian groups (resp.~commutative rings) with their
associated Eilenberg-MacLane spectra when confusion is unlikely to arise. 

Throughout, we will write $\mathcal{S}$ for the $\infty$-category of spaces and
$\Sp$ for the $\infty$-category of spectra. 
For $G$ a compact Lie group, the $G$-equivariant analogs will be denoted
$\mathcal{S}_G$ and $\GSpec$.
We will also write $BG$ for both the classifying space of $G$ and its
associated $\infty$-category ($\infty$-groupoid), so that, for an
$\infty$-category $\mathcal{C}$, $\fun(BG, \mathcal{C})$ denotes the
$\infty$-category of objects in $\mathcal{C}$ equipped with a $G$-action.
\subsection*{Acknowledgments}
We would like to thank Clark Barwick, John Greenlees, Mike Hopkins, Srikanth
Iyengar, Jacob Lurie, Peter May,
and David Treumann for helpful discussions.
We would especially like to thank Paul Balmer, whose preprint \cite{balmer} made us
aware of the connections between Quillen stratification and thick subcategory
methods. We thank the referee for a very careful reading and catching numerous typos.

\part{Generalities on symmetric monoidal $\infty$-categories}

\section{Complete objects}
\label{sec:axiomatic}

Consider the $\infty$-category $\md(\mathbb{Z})$ of modules over the
Eilenberg-MacLane spectrum $H\mathbb{Z}$, or equivalently the (unbounded) derived
$\infty$-category \cite[\S 1.3.5]{Lur14} of the category of abelian groups. 
Fix a prime number $p$. 
Then there are four stable subcategories of $\md(\mathbb{Z})$ that one can define. 
\begin{enumerate}
\item The subcategory $(\md(\mathbb{Z}))_{p-\mathrm{tors}}$ of $p$-torsion
$\mathbb{Z}$-modules: that is, the smallest localizing\footnote{Recall that a subcategory
of a presentable stable $\infty$-category is said to be \emph{localizing} if it
is a stable subcategory closed under colimits.} subcategory of $\md(\mathbb{Z})$
containing $\mathbb{Z}/p$. An object of $\md(\mathbb{Z})$ belongs to 
$(\md(\mathbb{Z}))_{p-\mathrm{tors}}$ if and only if all of its homotopy
groups are $p$-power torsion.
\item 
The subcategory $\md_{\mathbb{Z}[p^{-1}]}$ of $\mathbb{Z}[p^{-1}]$-modules:
that is, those objects $M  \in \md(\mathbb{Z})$ such that $M \otimes N $ is
contractible for every $N \in (\md(\mathbb{Z}))_{p-\mathrm{tors}}$.
This subcategory is closed under both arbitrary limits and colimits. 
\item The subcategory 
$(\md(\mathbb{Z}))_{p-\mathrm{cpl}}$ of $p$-complete $\mathbb{Z}$-modules:
that is, those $M$ such that for any $N \in \md_{\mathbb{Z}[p^{-1}]}$, the
space of maps $\hom_{\md(\mathbb{Z})}(N, M)$ is contractible.
This subcategory is closed under arbitrary limits and $\aleph_1$-filtered
colimits  (but not all colimits).
\item The subcategory 
$(\md(\mathbb{Z}))_{p-\mathrm{nil}}$ consisting of those $M \in
\md(\mathbb{Z})$ such that some power of $p$ annihilates $M$: that is, such
that $1_M \in \pi_0 \hom_{\md(\mathbb{Z})}(M, M)$ is $p$-power torsion.
This subcategory is only closed under \emph{finite} limits and colimits, as
well as retracts.
\end{enumerate}

The first three subcategories satisfy a number of well-known relationships. 
For
instance:
\begin{enumerate}
\item There is a
\emph{completion} (or $\mathbb{Z}/p$-localization) functor 
$\md(\mathbb{Z}) \to
(\md(\mathbb{Z}))_{p-\mathrm{cpl}}$.
\item  
There is an  
\emph{acyclization} or \emph{colocalization} functor 
$\md(\mathbb{Z}) \to
(\md(\mathbb{Z}))_{p-\mathrm{tors}}$. 
\item 
There is a \emph{localization} functor $L\colon \md(\mathbb{Z}) \to
\md_{\mathbb{Z}[p^{-1}]}$
\end{enumerate}

Dwyer-Greenlees theory \cite{DwG02} implies that $p$-adic completion
induces an equivalence 
\[   (\md(\mathbb{Z}))_{p-\mathrm{tors}} \simeq 
(\md(\mathbb{Z}))_{p-\mathrm{cpl}}.    \]
Moreover, there is an \emph{arithmetic
square} for building any object $X$ of $\md(\mathbb{Z})$ from $X[p^{-1}]$, the
$p$-adic completion $\widehat{X}_p$, and a compatibility map. 
Namely, any $X \in \md(\mathbb{Z})$
fits into a pullback square
\begin{equation} \label{arithmeticsquare} \xymatrix{
X \ar[d] \ar[r] & \widehat{X}_p  \ar[d] \\
X[p^{-1}]   \ar[r] &  (\widehat{X}_p)[p^{-1}].
}\end{equation}

This picture and its generalizations (for instance, its version in chromatic
homotopy theory) are often extremely useful in understanding how 
to build objects. The fourth subcategory 
$(\md(\mathbb{Z}))_{p-\mathrm{nil}}$ does not fit into such a functorial
picture, but every object here is both $p$-torsion and $p$-complete, and the
$p$-torsion is \emph{bounded}.

Let $(\mathcal{C}, \otimes, \mathbf{1})$ be a presentable, symmetric monoidal,
stable $\infty$-category where the tensor product commutes with colimits in
each variable. 
Let $A$ be an associative algebra object in $\mathcal{C}$. 
In the next two sections, we briefly review an axiomatic version of the above picture in $\mathcal{C}$ with
respect to $A$.
The main focus of this paper is the fourth subcategory of nilpotent objects in
equivariant stable homotopy theory. 
We emphasize that these ideas  are by no means new, and have been developed
by several authors, including \cite{Bou79,  Mil92, HPS97, DwG02,
Tateaxiomatic, DAGXII, BHV}.

\subsection{The Adams tower and the cobar construction}
As usual, let $(\mathcal{C}, \otimes, \mathbf{1})$ be a presentable, symmetric 
monoidal stable $\infty$-category where the tensor product commutes with
colimits. Let $A  \in \mathrm{Alg}(\mathcal{C})$ be an associative algebra object of $\mathcal{C}$. 
We begin with a basic construction. 

\begin{cons}[The Adams tower]
\label{adamstowerdef}
Let $M \in \mathcal{C}$. Then we can form 
a tower in $\mathcal{C}$
\begin{equation} \label{eq:atower} \dots \to  T_2(A, M) \to T_1(A, M) \to T_0(A, M) \simeq M  \end{equation}
as follows: 
\begin{enumerate}
\item $T_1(A, M) $ is the fiber of the map $M \to A \otimes M$
induced by the unit $\mathbf{1} \to A$, so that $T_1(A, M)$ maps naturally to $M$.  
\item More generally, $T_i(A, M) :=T_1(A, T_{i-1}(A, M))$ with its
natural map to $T_{i-1}(A, M)$.
\end{enumerate}
Inductively, this defines the functors $T_i$ and the desired tower. 
We will call this the \emph{$A$-Adams tower} of $M$.
Observe that the $A$-Adams tower of $M$ is simply the tensor product of $M$
with the $A$-Adams tower of $\mathbf{1}$.
\end{cons}

We can write the construction of the Adams tower in 
another way. 
Let $I = \mathrm{fib}( \mathbf{1}  \to A)$, so that $I$ is a \emph{nonunital}
associative algebra in $\mathcal{C}$ equipped with a map $I \to \mathbf{1}$. We have a tower
\[ \dots \to I^{\otimes n} \to I^{\otimes (n-1)} \to \dots \to I^{\otimes 2}
\to I \to \mathbf{1},  \]
and this is precisely the $A$-Adams tower $\left\{T_i(A, \mathbf{1})\right\}_{i \geq
0}$. The $A$-Adams tower for $M$ is obtained by tensoring this with $M$.

\begin{example} 
\label{adamstowerZ}
Take $\mathcal{C} = \md(\mathbb{Z})$ and $A = \mathbb{Z}/p$. Then the Adams
tower $\left\{T_i(\mathbb{Z}/p, M)\right\}$ of an object $M \in \md(\mathbb{Z} )$ is given by 
\[ \dots \to M \stackrel{p}{\to}M \stackrel{p}{\to}
M.  \]
\end{example}

The $A$-Adams tower has two basic properties: 
\begin{prop}
\label{Adamstowprop}
\begin{enumerate}
\item For each $i$, the cofiber of $T_{i} (A, M) \to T_{i-1}(A, M)$ admits
the structure of an $A$-module (internal to $\mathcal{C}$).
\item Each map $T_i(A, M) \to T_{i-1}(A, M)$ becomes nullhomotopic after
tensoring with $A$.
\end{enumerate}
\end{prop}
\begin{proof}
Suppose $i = 1$. In this case,  the cofiber of $T_1(A, M) \to M$ is precisely
$A \otimes M$ by construction. 
We have a cofiber sequence
\[ T_1(A, M) \to M \to A \otimes M,  \]
and the last map admits a section after tensoring with $A$. Therefore, the map $T_1(A, M)
\to M$ must become nullhomotopic after tensoring with $A$. 
Since $T_i(A, M) = T_1(A, T_{i-1}(A, M))$, the general case follows. 
\end{proof}

\begin{corollary} 
\label{adamstowernullmodule}
Suppose $M \in \mathcal{C}$ is an $A$-module up to
homotopy. Then the
successive maps $T_i(A, M) \to T_{i-1}(A, M)$ in the Adams tower are 
nullhomotopic.
\end{corollary} 
\begin{proof} 
If $M = A \otimes N$, then we just saw that each of the maps in the Adams tower
is nullhomotopic. If $M$ is an $A$-module up to homotopy, then $M$ is a
\emph{retract} (in $\mathcal{C}$) of $A \otimes M$, so each of the maps in the
Adams tower $\left\{T_i(A, M)\right\}$ must be nullhomotopic as well. 
\end{proof}

The construction of the Adams tower can be carried out even if
$A$ is only an algebra object in the \emph{homotopy category} of $\mathcal{C}$:
that is, one does not need the full strength of the associative algebra structure in
$\mathcal{C}$. 
However, we will also need the following construction that does use this extra (homotopy
coherent) structure.

\begin{cons}
\label{cobarconst}
Given $A \in \mathrm{Alg}(\mathcal{C})$, we can form 
a cosimplicial object in $\mathcal{C}$,
\[ \cb(A) = \left\{A \rightrightarrows A \otimes A \triplearrows \dots
\right\} \in \mathrm{Fun} ( \Delta, \mathcal{C}) , \]
called the \emph{cobar construction} on $A$. The cobar construction extends to
an augmented cosimplicial object $$\cbaug(A)\colon N( \Delta_+) \to
\mathcal{C},$$
(where $\Delta_+$ is the augmented simplex category of finite ordered sets), where the augmentation is from
the unit object $\mathbf{1}$. The augmented
cosimplicial object $\cbaug(A)$ admits a \emph{splitting} 
\cite[\S 4.7.3]{Lur14}
after tensoring
with $A$: that is, the augmented cosimplicial object $\cbaug(A) \otimes A$ is
split. 
\end{cons}

Although the cobar construction in the 1-categorical context is classical,
for precision
we 
spell out the details of how one may extract the cobar construction using the
formalism of \cite[\S 2.1-2.2]{Lur14}. 
By definition, since $\mathcal{C}$ is a symmetric monoidal $\infty$-category,
one has an $\infty$-category $\mathcal{C}^{\otimes}$ together with a
cocartesian fibration $\mathcal{C}^{\otimes }
\to N( \mathrm{Fin}_*)$ where $\mathrm{Fin}_*$ is the category of pointed finite sets. 
The underlying $\infty$-category $\mathcal{C}$ is obtained as the fiber over the pointed finite set
$\left\{0\right\} \cup \left\{\ast\right\}$.
The associative operad has an operadic nerve $N^{\otimes}( \mathbf{E}_1)$
 which maps to $N(   \mathrm{Fin}_*)$, and the algebra object $A$ defines a
 morphism
$ \phi_A\colon N^{\otimes}( \mathbf{E}_1) \to \mathcal{C}^{\otimes}$ over $N(
\mathrm{Fin}_*)$. The crucial point is that we have a functor
$N(\Delta_+) \to N^{\otimes}( \mathbf{E}_1)$ whose definition we will now recall. 

To understand this functor, recall that the 
operadic nerve
$N^{\otimes}( \mathbf{E}_1)$ (recall that we identify $\mathbf{E}_1$ and the
associative operad)
comes from the (ordinary) category described as follows: 
\begin{enumerate}
\item The objects are finite pointed sets $S \in \mathrm{Fin}_*$. 
\item Given $S, T$, to give a morphism $S \to T$ in $N^{\otimes}( \mathbf{E}_1)$ amounts to giving a morphism
$\rho\colon S \to T$ in $\mathrm{Fin}_*$ 
and an \emph{ordering} on each of the sets $\rho^{-1}(t) $
for $t \in T \setminus \left\{\ast\right\}$.
\end{enumerate}
We now obtain a functor $N(\Delta_+) \to N^{\otimes}( \mathbf{E}_1)$ which sends a finite
ordered set $S$ to $S \sqcup \left\{\ast\right\}$; a morphism $S \to
T$ in $\Delta_+$ clearly induces a morphism $N^{\otimes}(\mathbf{E}_1)$ (using the induced
ordering on the preimages). 
Composing, we obtain a functor
\[ \psi_A\colon N(\Delta_+ )\to N^{\otimes}( \mathbf{E}_1) \stackrel{\phi_A}{\to } \mathcal{C}^{\otimes}.  \]

Now, let $\mathrm{Fin}_*^{\mathrm{ac}} \subset \mathrm{Fin}_*$ be the (non-full)
subcategory with the same objects, but such that morphisms of pointed sets
$\rho\colon S \to T$ are
required to be \emph{active}, i.e., such that $\rho^{-1}(\ast) = \ast$. 
Observe that $\psi_A$ factors (canonically) over $\mathcal{C}^{\otimes}
\times_{N( \mathrm{Fin}_*)} N( \mathrm{Fin}_*^{\mathrm{ac}})$.

Finally, since $\mathcal{C}$ is a symmetric monoidal $\infty$-category, we have
a functor $$\bigotimes\colon
\mathcal{C}^{\otimes}
\times_{N( \mathrm{Fin}_*)} N( \mathrm{Fin}_*^{\mathrm{ac}})
\to \mathcal{C}$$ that, informally, tensors
together a tuple of objects.  
To obtain this, observe that for any object $S \in
\mathrm{Fin}_*^{\mathrm{ac}}$, there is a
\emph{natural} map $f_S\colon S \to \left\{0\right\} \cup \left\{\ast\right\}$ such that 
$f_S^{-1}(0) = S \setminus \left\{\ast\right\}$. 
(The naturality holds on $\mathrm{Fin}_*^{\mathrm{ac}}$, not on the larger 
category  $\mathrm{Fin}_*$.)
The 
functor $\bigotimes$ is the coCartesian lift of this natural transformation. 
Now, the (augmented) cobar construction 
is the composition
\[ \cbaug(A) \colon N(\Delta_+) \stackrel{\psi_A}{\to}  \mathcal{C}^{\otimes}
\times_{N( \mathrm{Fin}_*)} N( \mathrm{Fin}_*^{\mathrm{ac}})
\stackrel{\bigotimes}{\to}\mathcal{C}.\]

Our first goal is to demonstrate the connection between the Adams tower and the
cobar construction. 
Given a functor $X^\bullet\colon \Delta \to \mathcal{C}$, we recall that
$\mathrm{Tot}_n(X^\bullet)$ is defined to be the homotopy limit of
$X^\bullet|_{\Delta^{\leq n}}$ for $\Delta^{\leq n} \subset \Delta$ the full
subcategory spanned by $\left\{[0], [1], \dots, [n]\right\}$.
We will need the following important result. The notion of stability is
self-dual, so we have dualized the statement in the cited reference.  

\begin{thm}[{Lurie~\cite[Th.~1.2.4.1]{Lur14}}; $\infty$-categorical Dold-Kan
correspondence] 
Let $\mathcal{C}$ be a stable $\infty$-category. Then the functor
\[ \mathrm{Fun}( \Delta , \mathcal{C}) \to \mathrm{Fun}(\mathbb{Z}_{\geq
0}^{op}, \mathcal{C}), \quad X^\bullet \mapsto
\left\{\mathrm{Tot}_n(X^\bullet)\right\}_{n \geq 0}\]
establishes an equivalence between cosimplicial objects in $\mathcal{C}$ and
towers in $\mathcal{C}$.
\end{thm}

In particular, we will show (\Cref{adamscobar}) that under the
$\infty$-categorical Dold-Kan correspondence, the cobar construction and the
Adams tower correspond to one another. 
This result is  certainly not new, but we have
included it for lack of a convenient reference. 

\begin{definition} 
Given a finite nonempty set $S$, we will let $\mathcal{P}(S)$ denote the partially ordered
set of nonempty subsets of $S$ ordered by inclusion. 
We will let $\mathcal{P}^+(S)$ denote the partially ordered set of \emph{all}
subsets of $S$ ordered by inclusion. 
\end{definition} 
\begin{cons}\label{tensorfunctor}
Suppose given morphisms $f_s \colon X_s \to Y_s \in \mathcal{C}$ for each $s \in S$. 
Then we obtain a functor 
\[ F^+(\left\{f_s\right\}) \colon \mathcal{P}^+(S) \to \mathcal{C}  \]
whose value on a subset $S' \subset S$ is given by $$F^+(
\left\{f_s\right\})(S')  = \bigotimes_{s_1
\notin S'} X_{s_1}
\otimes \bigotimes_{s_2 \in S'} Y_{s_2}.$$
We will let $F( \left\{f_s\right\}) \colon \mathcal{P}(S) \to \mathcal{C}$
denote the restriction of 
$F^+( \left\{f_s\right\}) $.
\end{cons}

Our first goal is to give a formula 
for the inverse limit of these functors $F( \left\{f_s\right\})$.
This will be important in determining the partial totalizations of the Adams
tower (\Cref{adamscobar} below).

\begin{prop} 
\label{tensorlimit}
Let $S$ be a finite nonempty set and suppose given morphisms $f_s\colon X_s \to Y_s$
in $\mathcal{C}$ for each $s \in S$.
Form a functor $F( \left\{f_s\right\})\colon \mathcal{P}(S) \to \mathcal{C}$ as in 
\Cref{tensorfunctor}. Then there is an identification, functorial in $\prod_{s
\in S} \fun(\Delta^1, \mathcal{C})$,
\[ \varprojlim_{\mathcal{P}(S)}  F ( \left\{f_s\right\})\simeq \mathrm{cofib}
\left( \bigotimes_{s \in S} \mathrm{fib}(X_s\to Y_s) \to \bigotimes_{s \in S}
X_s \right)
.\]
\end{prop} 
\begin{proof} 
We first explain the map. 
Let $\mathrm{Fun}(\Delta^1, \mathcal{C})$ denote the $\infty$-category of
arrows in $\mathcal{C}$; it is itself a stable $\infty$-category. 
Observe that any object $X \to Y$ of 
$\mathrm{Fun}(\Delta^1, \mathcal{C})$ fits into a cofiber sequence
\begin{equation}  \label{arrowcof} ( \mathrm{fib}(X \to Y) \to 0) \to  (X \to Y) \to (Y \to
Y)  . \end{equation}
Given an $S$-indexed family of objects $\left\{f_s\colon X_s \to Y_s\right\}$ of
$\mathrm{Fun}(\Delta^1, \mathcal{C})$, we have associated an object $F^+(
\left\{f_s\right\}) \in \mathrm{Fun}(\mathcal{P}^+(S), \mathcal{C})$. 
We obtain a functor
\[ \prod_{s \in S}  \mathrm{Fun}(\Delta^1, \mathcal{C}) \to 
\mathrm{Fun}(\mathcal{P}^+(S), \mathcal{C})
\]
which is exact in each variable. 
Therefore, using \eqref{arrowcof}, we obtain a natural morphism
\[  F^+( \left\{\mathrm{fib}(X_s \to Y_s ) \to 0 \right\})  \to 
 F^+( \left\{f_s\right\} ),
\]
in $\mathrm{Fun}( \mathcal{P}^+(S), \mathcal{C})$.
Taking the cofiber of this morphism 
yields an object of 
$\mathrm{Fun}( \mathcal{P}^+(S), \mathcal{C})$ where the initial vertex is mapped
precisely to 
$ \mathrm{cofib}
\left( \bigotimes_{s \in S} \mathrm{fib}(X_s\to Y_s) \to \bigotimes_{s \in S}
X_s \right)$ 
and whose restriction to $\mathcal{P}(S)$ is identified with $F(
\left\{f_s\right\})$.

By the universal property of the homotopy limit (since $\mathcal{P}^+(S)$ is
the \emph{cone} on $\mathcal{P}(S)$), this gives a natural morphism 
\begin{equation} \label{colimmap}\mathrm{cofib}
\left( \bigotimes_{s \in S}\mathrm{fib}(X_s\to Y_s) \to \bigotimes_{s \in S}
X_s \right)
\to \varprojlim_{\mathcal{P}(S)} F( \left\{f_s\right\}) \in \mathcal{C}.\end{equation}

We need to argue that this morphism \eqref{colimmap} is an equivalence. We first claim that 
if one of the morphisms $f_s\colon X_s \to Y_s$ is an equivalence, 
then \eqref{colimmap} is an equivalence, i.e., that
$$F^+(\left\{f_s\right\}) \colon \mathcal{P}^+(S) \to \mathcal{C}$$ is a limit diagram. However, this follows from the dual of 
\cite[Lem.~1.2.4.15]{Lur14} applied to $K = \mathcal{P}(S \setminus
\left\{s\right\})$ as $K^{\lhd} = \mathcal{P}^+(S\setminus\left\{s\right\})$ and $ K^{\lhd} \times \Delta^1
= \mathcal{P}^+(S)$; the fiber of the natural map of diagrams $K^{\lhd}
\to \mathcal{C}$ thus obtained is contractible. Here $K^{\lhd}$ is the
\emph{left cone} over $K$ \cite[Notation 1.2.8.4]{Lur09}. 

Now, to show that \eqref{colimmap} is an equivalence, we 
observe that 
both sides are exact functors in each $\mathrm{Fun}(\Delta^1,
\mathcal{C})$ variable.  We use induction on the number of $f_s\colon X_s \to Y_s$
with $Y_s$ noncontractible. If all the $Y_s  = 0$, both sides of \eqref{colimmap}
are contractible. Now suppose $n$ of the $Y_s$'s are not zero, and choose $s_1
\in S$ with $Y_{s_1} \neq 0$. In this case, we use the cofiber sequence
\eqref{arrowcof}. In order to show that \eqref{colimmap}
is an equivalence, it suffices to show that \eqref{colimmap}
becomes an equivalence after we replace $f_{s_1}$ either by $Y_{s_1}
\stackrel{\mathrm{id}}{\to} Y_{s_1}$ or $ \mathrm{fib}(X_{s_1} \to Y_{s_1}) \to
0$. We have treated the first case in the previous paragraph, and the second
case follows by the inductive hypothesis. 
\end{proof}

%The $A$-Adams tower $\left\{T_i(A, M)\right\}$ maps to the constant tower at
%$M$. 
\begin{prop} 
\label{adamscobar}
The tower associated (via the Dold-Kan correspondence) to the 
cosimplicial object $\cb(A)$ is precisely the tower $$\{  \mathrm{cofib}(
T_{n+1}(A,
\mathbf{1}) \to \mathbf{1})\}.$$ In other words, we have equivalences $\mathrm{Tot}_n( \cb(A)) \simeq
\mathrm{cofib}( I^{\otimes (n+1) } \to \mathbf{1})$.
\end{prop} 
\begin{proof} 
We compute $\mathrm{Tot}_n(\cb(A))$. 
For this, 
we let $\mathcal{P}([n])$ denote the partially ordered set of nonempty 
subsets of $[n]$. There is  a natural functor
\[ \mathcal{P}([n]) \to \Delta^{\leq n}  \]
which is right cofinal by \cite[Lem.~1.2.4.17]{Lur14}. 
We can describe the composite functor $$\mathcal{P}([n]) \to \Delta^{\leq n}
\stackrel{\cb(A)}{\longrightarrow} 
\mathcal{C}$$ as follows: it is obtained 
by
considering the unit maps $f_s\colon \mathbf{1} \to A$
for each $s \in [n]$ and forming $F( \left\{f_s\right\})$ as in 
\Cref{tensorfunctor}. Now, the homotopy limit is thus computed by 
\Cref{tensorlimit} and it is as desired. The maps in the tower, too, are seen to be the
natural ones. 
\end{proof}

\subsection{Complete objects}
We review rudiments of
the theory of Bousfield localization \cite{Bou79} in our setting.
As before, $(\mathcal{C}, \otimes, \mathbf{1})$ is a presentable,
symmetric monoidal stable 
$\infty$-category where $\otimes$ commutes with colimits in each variable, and $A \in
\mathrm{Alg}(\mathcal{C})$.

\begin{definition} 
We say that an object $X \in \mathcal{C}$ is \emph{$A$-complete} or
\emph{$A$-local}
if, for any $Y \in \mathcal{C}$ with $Y \otimes A \simeq 0$, the space of maps
$\hom_{\mathcal{C}}(Y, X)$ is contractible. 
The $A$-complete objects of $\mathcal{C}$ span a full  subcategory
$\Acp \subset \mathcal{C}$. 
\end{definition} 

\begin{example} 
A motivating example to keep in mind throughout is $\mathcal{C} =
\md(\mathbb{Z})$ and $A = \mathbb{Z}/p$. Here, the $A$-complete objects of
$\md(\mathbb{Z})$ are referred to as
\emph{$p$-adically complete.}
\end{example} 

\begin{example} 
\label{everythingAcomplete}
Suppose $A$ has the property
that tensoring with $A$ is conservative. For instance, a duality argument
shows that this holds if $A$ is dualizable (cf. \cite[\S 4.6.1]{Lur14}) and $\mathbb{D} A $ generates $\mathcal{C}$ 
as a localizing subcategory. Then every object of $\mathcal{C}$ is $A$-complete.
\end{example} 

\begin{example} 
\label{ex:modcomplete}
Suppose $M \in \mathcal{C}$ admits the structure of an $A$-module. Then
$M$ is $A$-complete. In fact, suppose $X \in \mathcal{C}$ is such that 
$A \otimes X $ is contractible. 
Then 
\[ \hom_{\mathcal{C}}(X, M) \simeq \hom_{\md_{\mathcal{C}}(A)}(A \otimes X, M)
\simeq
\hom_{\md_{\mathcal{C}}(A)}(0, M) \simeq 0.
\]
\end{example} 

It follows formally from the definitions that $\Acp$ is closed under all limits
in $\mathcal{C}$. 
The subcategory  $\Acp$ can equivalently be described as consisting of those objects $X \in \mathcal{C}$ such
that if $Y \to Y'$ is a map that becomes an equivalence after tensoring with
$A$, then $\hom_{\mathcal{C}}(Y', X) \to \hom_{\mathcal{C}}(Y, X)$ is an
equivalence.

We invoke here the theory of Bousfield localization in the $\infty$-categorical 
context \cite[\S 5.5.4]{Lur09}. 
In particular, we 
let $S$ be the collection of morphisms $Y \to Y'$ in $\mathcal{C}$ which become
an equivalence after tensoring with $A$. By \cite[Prop.~5.5.4.16]{Lur09}, this
class $S$, as a strongly saturated class (\cite[Def.~5.5.4.5]{Lur09}) is of small
generation. 
We now invoke the basic existence result
\cite[Prop.~5.5.4.15]{Lur09}, which implies
that $\Acp$ is a presentable $\infty$-category, and that 
the inclusion $\Acp \subset \mathcal{C}$
has a left adjoint. 

\begin{definition} 
We will let $L_A\colon \mathcal{C} \to \Acp$ denote the left adjoint to the
inclusion $\Acp \subset \mathcal{C}$ and refer to $L_A$ as \emph{$A$-completion.} 
We will also abuse notation and use $L_A$ to denote the composition $\mathcal{C}
\stackrel{L_A}{\to} \Acp \subset \mathcal{C}$ when confusion is unlikely to
arise. 
\end{definition}

When regarded as a functor $L_A\colon \mathcal{C} \to \mathcal{C}$ (as $\Acp \subset
\mathcal{C}$ is a full subcategory), we have a natural transformation $X \to
L_A X$ for any $X$, with the properties:
\begin{enumerate}
\item The map $X \to L_A X $ becomes an equivalence after tensoring with $A$.  
\item The object $L_A X$ is $A$-complete.
\end{enumerate}

\begin{remark} 
\label{AcptensorAcocont}
We recall that colimits in $\Acp$ are computed by first computing
the colimit in $\mathcal{C}$ and then applying localization $L_A$ again. In
particular, while the inclusion $\Acp \subset \mathcal{C}$ need not preserve
colimits, the composition $\Acp \subset \mathcal{C} \stackrel{\otimes A }{\to}
\mathcal{C}$ does.
\end{remark} 

Suppose $\phi\colon X \to Y$ is a map in $\mathcal{C}$ such that $\phi
\otimes {1_A}\colon X \otimes A \to Y
\otimes A$ is an equivalence. Then for any $Z$, the map $\phi \otimes 1_Z$ has
the same property. 
In view of \cite[Prop.~2.2.1.9]{Lur14}, $\Acp$ inherits the structure of a
symmetric monoidal $\infty$-category such that the functor $L_A\colon \mathcal{C}
\to \Acp$ is symmetric monoidal.

In this subsection, we will review several characterizations of 
$A$-complete objects, and describe the subcategory of complete objects as a
homotopy limit of presentable $\infty$-categories. 
 Throughout, the assumption that $A$ is dualizable will be critical
as it implies that tensoring with $A$ commutes with
homotopy limits. 
The first basic result is as follows.

\begin{prop} \label{Acompl:tot} Suppose $A$ is dualizable.
For any object $M \in \mathcal{C}$, 
the map $M \to \mathrm{Tot}( M \otimes \cb(A))$ exhibits the target as the $A$-completion of $M$.
\end{prop} 
\begin{proof} 
In fact, the map 
$M \to \mathrm{Tot}( M \otimes \cb(A))$ becomes an equivalence after tensoring
with $A$. 
This follows because $M \otimes \cbaug(A)$ becomes a split augmented
cosimplicial object after tensoring with $A$. In addition, we use 
the fact that tensoring with $A$ commutes with arbitrary homotopy
limits (as $A$ is dualizable).
Moreover, 
$\mathrm{Tot}( M \otimes \cb(A))$ is $A$-complete as it is the homotopy limit
of a diagram of objects, each of which is an $A$-module and therefore
$A$-complete (\Cref{ex:modcomplete}).
\end{proof} 

In view of \Cref{adamscobar}, we find (with $I = \mathrm{fib}( \mathbf{1} \to
A)$) an equivalence
\begin{equation} 
L_A M  \simeq \varprojlim_n \left[ \mathrm{cofib}(I^{\otimes n+1} \to \mathbf{1})
\otimes M\right] .
\end{equation} 
This recovers the familiar formula for $p$-adic completion in
$\md(\mathbb{Z})$, for example. 

We now obtain the following criteria for $A$-completeness.
\begin{prop} 
The following are equivalent for an object $M \in \mathcal{C}$ and for $A \in
\mathrm{Alg}(\mathcal{C})$, assumed dualizable in $\mathcal{C}$. 
\begin{enumerate}
\item The object $M$ is $A$-complete.  
\item The homotopy limit of the Adams tower $\left\{T_i(A, M)\right\}_{i \geq
0}$ is contractible. 
\item 
The augmented cosimplicial object $\cbaug(A) \otimes M$ is a limit diagram. 
\end{enumerate}
\end{prop} 
\begin{proof}
$(1) \Leftrightarrow (3)$. This follows 
from \Cref{Acompl:tot}. 

To see that  $(1) \Leftrightarrow(2)$, one 
can use the comparison between the Adams tower
and the cobar construction (\Cref{adamscobar}) and conclude.
One can also argue directly; we leave this to the reader. 

\begin{comment}We give a direct, elementary argument as well.

$(1) \implies (2)$. If $M$ is $A$-complete, then we saw in
\Cref{Acompl:tot} that $M$ is an inverse limit of a diagram consisting of
$A$-modules. For an $A$-module,  we saw in \Cref{adamstowernullmodule} that each of
the successive maps in the Adams tower is nullhomotopic, and therefore the
homotopy limit of the Adams tower is contractible. The formation of the Adams
tower $M \mapsto T_i(A, M)$ commutes with homotopy limits in $M$ (as $A$ is
dualizable). 
Therefore, the homotopy limit of the Adams tower for $M$ itself must be
contractible. 

$(2) \implies (1)$. Suppose the homotopy limit of the Adams tower is
contractible. Consider the cofiber of the natural map from the Adams tower to
the constant tower $\left\{M\right\}$, which is a new tower
$\left\{C_i\right\}_{i \geq 0}$. By assumption, the homotopy limit of the tower
$\left\{C_i\right\}$ is equivalent to $M$. It is easy to see by induction that
each $C_i$ is $A$-complete. For instance, $C_1$ is $A \otimes M$. Moreover, the
cofiber of $C_i \to C_{i-1}$ admits the structure of an $A$-module and so is
$A$-complete. Inducting upward, we see that each $C_i$ is $A$-complete, and therefore the
homotopy limit ($M$ itself) is $A$-complete. 
\end{comment}
\end{proof} 

\begin{corollary} 
\label{dualtensorcompl}
If $X \in \mathcal{C}$ is $A$-complete and $Y \in \mathcal{C}$ is dualizable,
then $X \otimes Y$ is $A$-complete.
\end{corollary}

\begin{cons}
\label{dualisamodule}
The object $\mathbb{D}A \in \mathcal{C}$ admits the structure
of an $A$-module. 
In fact, we have a map $A \otimes \mathbb{D}A \to \mathbb{D}A$ which is
(doubly) adjoint
to the multiplication map $A \otimes A \to A$ which makes $A$ into an
$A$-module. 
Alternatively, the module structure on $\mathbb{D} A$ comes from applying the
\emph{right} adjoint $\hom_{\mathcal{C}}(A, \cdot )$ of the forgetful functor $\md_{\mathcal{C}}(A) \to
\mathcal{C} $ to $\mathbf{1} \in \mathcal{C}$. 
\end{cons}

 It will now be convenient to make the following further
hypotheses on $\mathcal{C}$ and $A$, which will be in effect until the end of the
subsection.

\begin{hypotheses}
\label{hypothesis}
$(\mathcal{C}, \otimes, \mathbf{1})$ is a presentable, symmetric monoidal
stable $\infty$-category where $\otimes$ commutes with colimits in each
variable. We assume furthermore
that:
\begin{enumerate}
\item The unit $\mathbf{1}$ is compact.
\item The object $A$ is dualizable (as already assumed).
\item The $\infty$-category $\mathcal{C}$ is generated as a localizing subcategory by dualizable objects.
\end{enumerate}
\end{hypotheses}

Recall that in this setting, compactness of the unit implies compactness of all dualizable
objects.

\begin{prop} 
\label{cptgencompl}
Let $\mathcal{D}$ be a family of of dualizable generators for $\mathcal{C}$.
Then the objects $\{\mathbb{D}A \otimes X\}_{X \in \mathcal{D}}$ form a system of
compact generators for $\Acp$.
\end{prop} 
\begin{proof} 
Fix $X \in \mathcal{D}$.
By \Cref{dualisamodule}, $\mathbb{D}A \otimes X$ belongs to $\Acp$ as it is
an $A$-module. We show that $\mathbb{D}A \otimes X$ is compact in $\Acp$. Indeed, 
\[ \hom_{\mathcal{C}}( \mathbb{D}A \otimes X, Y) \simeq \hom_{\mathcal{C}}(
X, A \otimes Y),  \]
and we observe that the functor $\Acp \to \mathcal{C}$, $Y \mapsto A \otimes
Y$, commutes with colimits
(\Cref{AcptensorAcocont}). Since $X$ is compact in $\mathcal{C}$, we can now
conclude that $\mathbb{D}A \otimes X$ is compact in $\Acp$. 

To show that the $\{\mathbb{D}A \otimes X\}_{X \in \mathcal{D}}$ generate
$\Acp$, it suffices (\Cref{generateifcons}) to show
that if $Y \in \Acp$ is arbitrary and 
$\hom_{\mathcal{C}}( \mathbb{D}A \otimes X, Y)$ is contractible for all $X
\in \mathcal{D}$, then $Y$ is
contractible. But this means  
that 
$\hom_{\mathcal{C}}(
X, A \otimes Y)$ is contractible for all $X \in \mathcal{D}$. Thus $A  \otimes Y$ is
contractible, so $Y$ in turn is contractible by $A$-completeness.
\end{proof}

Next, we include a result that describes complete objects for a tensor product
of algebras. This result (and its variants for torsion and nilpotent objects)
will be useful in the sequel. 

\begin{prop} 
\label{complAB}
Suppose $A, B \in \mathrm{Alg}(\mathcal{C})$ are dualizable in $\mathcal{C}$.
Then an object $X \in \mathcal{C}$ is $(A \otimes B)$-complete if and only if
$X$ is both $A$-complete and $B$-complete.
\end{prop} 
\begin{proof} 
Suppose $X$ is $(A \otimes B)$-complete. Then 
\( X \simeq \mathrm{Tot}( X \otimes \cb(A \otimes B)).  \)
Each term in this totalization is an $A$-module, and therefore
$A$-complete. Thus, the homotopy limit $X$ is $A$-complete too. Similarly, $X$
is $B$-complete.

Conversely, suppose $X$ is both $A$-complete and $B$-complete. 
Then consider the bicosimplicial diagram 
\begin{equation} \label{bicosimp} X \otimes \cb(A) \otimes \cb(B) \colon \Delta \times
\Delta \to \mathcal{C}.  \end{equation}
Since $A^{\otimes k} \otimes X$ is $B$-complete for any $k$
(\Cref{dualtensorcompl}), and since $X$ is $A$-complete, one sees that the
homotopy limit of the bicosimplicial diagram  
\eqref{bicosimp} is $X$ itself: indeed, one computes the bitotalization one
factor at a time. However, every term in the bicosimplicial
diagram \eqref{bicosimp} 
is an $(A \otimes B)$-module and thus $(A \otimes B)$-complete. Thus, $X$ is
$(A \otimes B)$-complete itself. 
\end{proof}

The final goal of this section is to describe the $\infty$-category $\Acp$ as a
homotopy limit, via descent theory, when $A$ is actually a \emph{commutative}
algebra object, so that the cobar construction takes values in commutative
algebra objects.\footnote{The cobar construction on an associative algebra
object does not live in the $\infty$-category of algebra objects.} This is the one part of the present
section where the language of $\infty$-categories is necessary, and  the result will be useful to us in the sequel. 

Consider the augmented cobar construction $\cbaug(A) \colon \Delta^+ \to
\mathrm{CAlg}(\mathcal{C})$. 
Taking module $\infty$-categories everywhere, we obtain a cosimplicial diagram
of symmetric monoidal stable $\infty$-categories $$ \md_{\mathcal{C}}(A) 
\rightrightarrows \md_{\mathcal{C}}(A \otimes A) \triplearrows \dots $$
receiving an augmentation  from $\mathcal{C}$.

\begin{thm} 
\label{Acplimit}
If $A \in \mathrm{CAlg}(\mathcal{C})$ is dualizable in $\mathcal{C}$, then 
$\Acp $ can be recovered as the homotopy
limit
\[ \Acp \simeq \mathrm{Tot}\left( \md_{\mathcal{C}}(A) 
\rightrightarrows \md_{\mathcal{C}}(A \otimes A) \triplearrows \dots \right)
,\]
in the $\infty$-category of symmetric monoidal $\infty$-categories. 
\end{thm} 
\begin{proof} 
We have an adjunction
\[ (F, G)\colon  \mathcal{C} \rightleftarrows  \md_{\mathcal{C}}(A)  \]
where $F(X) = A \otimes X$ and $G$ forgets the $A$-module structure. This
adjunction descends to a similar adjunction 
$$(F', G') \colon \Acp \rightleftarrows \md_{\mathcal{C}}(A),$$
with the same formulas. 
As a result, the coaugmentation from $\mathcal{C}$ of the cosimplicial 
symmetric monoidal $\infty$-category
$\md_{\mathcal{C}}(
\cb(A))$ descends to a coaugmentation from $\Acp$, leading to the natural
functor
\begin{equation} \label{functorAcptot} \Acp \to\mathrm{Tot}\left( \md_{\mathcal{C}}(A) 
\rightrightarrows \md_{\mathcal{C}}(A \otimes A) \triplearrows \dots \right)
.\end{equation}
We want to see that this functor is an equivalence of symmetric monoidal
$\infty$-categories.
This is a descent argument using the $\infty$-categorical monadicity theorem 
and an identification of the above homotopy limit, which appears in the proof
of \cite[Prop.~6.18]{DAGVII}.

For this, we consider the map \eqref{functorAcptot}
and replace all $\infty$-categories with their opposites to obtain 
a new map 
\begin{equation} \label{functorAcptot2} (\Acp)^{op} \to\mathrm{Tot}\left(
\md_{\mathcal{C}}(A)^{op} 
\rightrightarrows \md_{\mathcal{C}}(A \otimes A)^{op} \triplearrows \dots \right)
.\end{equation}
It suffices to 
show that \eqref{functorAcptot2} is an equivalence.
For this, we invoke \cite[Cor.
4.7.6.3]{Lur14}.
The necessary condition on left adjointability is satisfied in view of
\cite[Lem.~6.15]{DAGVII}. 
In order to apply 
\cite[Cor.
4.7.6.3]{Lur14}, 
it therefore suffices to show that tensoring with $A$, as a functor $\Acp \to
\md_{\mathcal{C}}(A)$, preserves
$A$-split totalizations and is conservative. 
However, since $A$ is dualizable, tensoring with $A$ preserves \emph{all}
limits, and it is conservative on $\Acp$ (since any object $X \in \Acp$ with
$X \otimes A \simeq 0$ must be contractible itself). Therefore, we can apply the
comonadicity theorem and complete the proof.
\end{proof} 

We emphasize that the above argument is standard \cite[\S 6]{DAGVII} in
$\infty$-categorical descent theory. The main use of it here is to identify an
$\infty$-category of \emph{complete} objects with respect to a dualizable
algebra object.

\begin{example} 
Suppose $A$ has the property that tensoring with $A$ is \emph{conservative} on
$\mathcal{C}$. In this case, one sees easily that $\Acp = \mathcal{C}$ and the
above result, \Cref{Acplimit}, is a descent theorem for $\mathcal{C}$ itself as
a homotopy limit of modules over the tensor powers $\{A^{\otimes (n+1)}\}_{n
\geq 0}$. In fact, by \cite[Th.~3.36]{galois}, 
the commutative algebra object $A$ is \emph{descendable} in $\mathcal{C}$,
i.e., the thick $\otimes$-ideal it generates is all of $\mathcal{C}$. In
particular, this descent theorem is \cite[Prop.~3.21]{galois}. 
While the decomposition of $\mathcal{C}$ as a homotopy limit does not
require compactness of the unit, the additional conclusion of descendability of
$A$ does. 
\end{example}

\section{$A$-torsion objects and $A^{-1}$-local objects}
In this section, we describe the theory of \emph{torsion} objects
with respect to the algebra object $A \in \mathrm{Alg}(\mathcal{C})$, and the
dual theory of $A^{-1}$-local objects.
The main results are a general version (\Cref{ourdwg}) of the Dwyer-Greenlees \cite{DwG02} equivalence 
between complete and torsion objects, which is due to Hovey-Palmieri-Strickland \cite[Th.
3.3.5]{HPS97} and a version of the arithmetic square
(\Cref{Afracsquare}).
We continue to 
work under \Cref{hypothesis}.

\subsection{Torsion objects}
\begin{definition} 
\label{def:torsionobj}
The subcategory $\Ator$ of \emph{$A$-torsion} objects in $\mathcal{C}$ is the
smallest localizing subcategory of $\mathcal{C}$ containing $A \otimes
X$, for $X \in \mathcal{C}$ dualizable. \end{definition} 

As with $\Acp$, our first goal is to make $\Ator$ explicit. 
\newcommand{\AC}{\mathrm{AC}}

\begin{cons}
By \cite[Cor.~1.4.4.2]{Lur14}, $\Ator$ is a presentable $\infty$-category. In particular, by the adjoint
functor theorem \cite[Cor.~5.5.2.9]{Lur09}, the fully faithful inclusion $\Ator \subset \mathcal{C} $ is a left adjoint and admits a
right adjoint 
\[ \AC_A\colon \mathcal{C} \to \Ator,  \]
which is called the \emph{$A$-acyclization functor.}
For any object $X \in \mathcal{C}$, there is a natural (counit) map $\AC_A(X)
\to X$ in $\mathcal{C}$.
\end{cons}

Our first goal is to get a handle on $\AC_A$. 
We begin by showing that $\Ator$ is a $\otimes$-ideal (\Cref{thickdef}).

\begin{prop} \label{Atorideal}
If $Y \in \Ator$ and $X \in \mathcal{C}$, then $X \otimes Y \in \Ator$.
\end{prop} 
\begin{proof} 
Consider the collection of $X \in \mathcal{C}$ such that $X \otimes Y \in
\Ator$. By definition, this collection is localizing,
so to show that it is all of $\mathcal{C}$, it suffices to show that it
contains all $X$ \emph{dualizable.} So, we may assume that $X$ is  dualizable.

Fix a dualizable $X$. Consider the collection of all $Y' \in \Ator$ such that $X
\otimes Y' \in \Ator$. The collection of such $Y'$ is localizing, so
to show that it is all of $\Ator$, it suffices to consider the case where $Y' =
A \otimes Y''$ for $Y'' \in \mathcal{C}$ dualizable. But in this case $X \otimes Y' = X
\otimes A \otimes Y''$ clearly belongs
to $\Ator$: in fact, it is one of the generating objects. 
\end{proof} 

As a result, we find that $\Ator$ is also the localizing $\otimes$-ideal
generated by $A$.
We will now write down an explicit formula for $\AC_A(X)$.

\begin{cons}
\label{U_iseq}
Recall the Adams tower
$$ \dots \to T_2(A, \mathbf{1}) \to T_1(A, \mathbf{1}) \to T_0(A, \mathbf{1})
\simeq \mathbf{1} $$ 
of the unit object $\mathbf{1}$. As $A$ is dualizable, each of the objects in
this tower is dualizable, so we can
form the dual tower
\[ \mathbf{1} \to U_1 \to U_2 \to \dots,  \]
where $U_i := \mathbb{D}( T_i(A, \mathbf{1}))$.
We define $U_A = \varinjlim U_i$ and let $V_A$ be the fiber of $\mathbf{1}\
\to U_A$.

Equivalently, let $\cb(A) \colon \Delta \to \mathcal{C}$ denote the cobar
construction on $A$ and form the pointwise dual 
\[ \mathbb{D}(\cb(A))  \colon \Delta^{op} \to \mathcal{C},\]
which maps via an augmentation to $\mathbb{D}(\mathbf{1})\simeq\mathbf{1}$. Then $V_A = |
\mathbb{D}(\cb(A))|$.
\end{cons}

\begin{prop} 
\label{VAacyclization}
For any $X \in \mathcal{C}$, we have a natural equivalence $\AC_A(X) \simeq V_A \otimes X $.
Therefore,  $X \in \Ator$ if and only if the natural map
$V_A \otimes X \to X$ is an equivalence. 
\end{prop} 
\begin{proof}
We have a natural map $V_A \otimes X \to X$. In order to show that $V_A \otimes
X$ is identified with $\AC_A(X)$, we need to show two things: 
\begin{enumerate}
\item $V_A \otimes X$ belongs to $\Ator$. 
\item For any $Y \in \Ator$, we have that $\hom_{\mathcal{C}}(Y, U_A \otimes X)
$ is contractible.
\end{enumerate}
The latter condition comes from the natural cofiber sequence $V_A \otimes X \to
X \to U_A \otimes X$.
We now prove these claims.
\begin{enumerate}
\item It suffices to show that $V_A \in \Ator$, by \Cref{Atorideal}. 
For this, it suffices to show that the cofiber of each map $\mathbf{1} \to U_i$
belongs to $\Ator$. By induction and the octahedral axiom, it suffices to show that for each $i \geq 0$,
the cofiber of $U_i \to U_{i+1}$ belongs to $\Ator$. 
But this map is the dual to the map $T_{i+1}(A, \mathbf{1}) \to T_{i}(A,
\mathbf{1})$, and the fiber of this map is of the form $A \otimes M$ for a
dualizable object $M$. Now any object of the form $\mathbb{D}(A \otimes
M) \simeq \mathbb{D}A \otimes
\mathbb{D} M$ belongs to $\Ator$ as $\mathbb{D}A$ is an $A$-module
(\Cref{dualisamodule}) and thus a  retract of $A \otimes
\mathbb{D}A$. Thus, the cofiber of $U_i \to U_{i+1}$ belongs to $\Ator$.
\item  Fix $X$ arbitrary. The collection of $Y$ for which $\hom_{\mathcal{C}}
(Y, U_A \otimes X)$
is contractible is localizing, so it suffices to prove
the claim for $Y = A \otimes Y'$ with $Y'$ dualizable. In this case, we have
\[  \hom_{\mathcal{C}}
(Y, U_A \otimes X)
\simeq \hom_{\mathcal{C}}(Y', \mathbb{D}A \otimes U_A \otimes X).
\]
Now the tower $\{ A \otimes T_i(A, \mathbf{1})\}$ has the property that every
map is null (\Cref{Adamstowprop}), so by duality, every map $U_i \otimes \mathbb{D}A \to U_{i+1} \otimes
\mathbb{D}A$ is nullhomotopic. In particular, $\mathbb{D}A \otimes U_A$ is
contractible, which proves the claim. 
\end{enumerate}
\end{proof}

\begin{example} 
We consider $\mathcal{C} = \md(\mathbb{Z})$, $A = \mathbb{Z}/p$. 
In this case, the sequence $1 \to U_1 \to U_2 \to \dots $ becomes the sequence
\[ \mathbb{Z} \stackrel{p}{\to}  \mathbb{Z} \stackrel{p}{\to} \dots,  \]
so that $U_{\mathbb{Z}/p} = \mathbb{Z}[p^{-1}]$ and $V_{\mathbb{Z}/p} =
\Sigma^{-1} ( \mathbb{Z}[p^{-1}]/\mathbb{Z})$.
\end{example}

\begin{prop} 
\label{cptgentor}
Let $\mathcal{D}$ be a collection of dualizable generators for $\mathcal{C}$.
Then the objects $\{\mathbb{D}A \otimes X\}_{X \in \mathcal{D}}$ form a system of compact
generators for $\Ator$.
\end{prop} 
\begin{proof} 
These objects are $A$-modules by \Cref{dualisamodule}, so they belong to $\Ator$; since they
are compact in $\mathcal{C}$, they are compact in $\Ator$. It suffices to show
that they are generators. Here the argument proceeds as in \Cref{cptgencompl}: 
it reduces to showing that if $Y \in \Ator$ and $A \otimes Y$ is contractible, then $Y$ is
contractible. 
But in this case, $Q \otimes Y$ is contractible for any $A$-module $Q$. 
It follows that $Q' \otimes Y$ is contractible for any $Q' \in \Ator$, in
particular, for $Q' = V_A$, so that $V_A \otimes Y$ is contractible. But $V_A
\otimes Y \simeq Y$ as $Y \in \Ator$. Therefore, $Y$ is contractible. 
\end{proof} 

Next, we include an analog of 
\Cref{complAB} for torsion objects. 

\begin{prop} 
\label{torsAB}
Let $A, B \in \mathrm{Alg}(\mathcal{C})$ be dualizable in $\mathcal{C}$. Then
an object $X \in \mathcal{C} $ is $(A \otimes B)$-torsion if and only if it is
both $A$-torsion and $B$-torsion.
\end{prop} 
\begin{proof} 
If $X$ is $(A \otimes B)$-torsion, then we know that $X$ belongs to the
localizing $\otimes$-ideal (see \Cref{thickdef}) generated by 
$A \otimes B$. Therefore, $X$ belongs to the localizing $\otimes$-ideal
generated by $A$ and is consequently $A$-torsion. Similarly, $X$ must be
$B$-torsion. 

Suppose now that $X$ is both $A$-torsion and $B$-torsion. 
Then $V_A \otimes X \simeq X$ and $V_B \otimes X \simeq X$, so $V_A \otimes V_B
\otimes X \simeq X$. It suffices to show, as a result, that $V_A \otimes V_B$
is $(A \otimes B)$-torsion. 
For this, we construct sequences $$\mathbf{1} \to U_1^{(A)} \to U_2^{(A)} \to
\dots
 \quad \text{and} \quad \mathbf{1} \to U_1^{(B)} \to U_2^{(B)} \to \dots $$
as in \Cref{U_iseq}, such that $V_A \simeq \mathrm{fib}\left( \mathbf{1} \to
\varinjlim U_i^{(A)} \right)$ and $V_B \simeq \mathrm{fib}\left( \mathbf{1} \to
\varinjlim U_i^{(B)} \right)$. 
To show that $V_A \otimes V_B$ is $(A \otimes B)$-torsion, we first observe that, for each $i,j$, $\mathrm{cofib}\left(U_i^{(A)}
\to U_{i+1}^{(A)} \right) \otimes 
\mathrm{cofib}\left(U_j^{(B)}
\to U_{j+1}^{(B)}\right)$ is an $(A \otimes B)$-module and hence $(A \otimes B)$-torsion. The claim for $V_A\otimes V_B$ now follows by induction.  
\end{proof}

We now state and briefly prove a version of \cite[Th. 2.1]{DwG02}, due to
Hovey-Palmieri-Strickland in our
context.  
\begin{thm}[{Cf. \cite[Th. 3.3.5]{HPS97}}] 
\label{ourdwg}
The functor of $A$-completion establishes an equivalence of $\infty$-categories $L_A\colon \Ator
\simeq \Acp$ (whose inverse is given by $\AC_A$).
\end{thm} 
\begin{proof} 
Let $X \in \mathcal{C}$ be a dualizable object. Then $\mathbb{D}A \otimes X \in
\mathcal{C}$ belongs to both the subcategories $\Ator$ and $\Acp$. Moreover,
the $\mathbb{D}A \otimes X$ form a family of compact generators (as $X$ ranges over the
dualizable objects) for both subcategories $\Ator, \Acp$, in view of
\Cref{cptgencompl} and \Cref{cptgentor}. Since $L_A$ carries
$\mathbb{D}A \otimes X $ to itself (as any $A$-module is $A$-complete), it follows formally that $L_A$ induces an equivalence
as stated. 

In more detail, we let $\mathcal{C}'_A \subset \mathcal{C}$ denote the thick
subcategory (\Cref{thickdef}) generated by the $\left\{\mathbb{D} A \otimes X\right\}$ for the $X
\in \mathcal{C}$ dualizable. Then $\mathcal{C}'_A \subset \Ator \cap \Acp$ and
identifies with a system of compact generators of each. Therefore, we have
equivalences $\Ator \simeq \mathrm{Ind} (\mathcal{C}'_A), 
\Acp \simeq \mathrm{Ind} (\mathcal{C}'_A)$ 
by \cite[Prop.~5.3.5.11]{Lur09}. The functor $L_A\colon \Ator \to \Acp$
preserves colimits, as the composition of the inclusion $\Ator \subset
\mathcal{C}$ and $L_A\colon \mathcal{C} \to \Acp$.  It also takes the
compact generators $\mathbb{D}A \otimes X$ to compact objects of $\Acp$. It is
therefore induced by left Kan extension
of the identity $\mathcal{C}'_A \to \mathcal{C}'_A \subset
\mathrm{Ind}(\mathcal{C}'_A)$ \cite[Lem.~5.3.5.8]{Lur09}
and is therefore an equivalence.
\end{proof}

\subsection{$A^{-1}$-local objects and fracture squares}
\label{sec:A-1local}

We keep the notation of the previous subsection.

\begin{definition} 
We say that an object $X \in \mathcal{C}$ is $A^{-1}$-local if, for any object
$Y \in \Ator$, we have $\hom_{\mathcal{C}}(Y, X) \simeq 0$.
We let $\mathcal{C}[A^{-1}] \subset \mathcal{C}$ denote the full subcategory
spanned by the $A^{-1}$-local objects. 
\end{definition} 

This condition (for a fixed $X$) is preserved under colimits in $Y$. It follows
that: 

\begin{prop} 
An object $X \in \mathcal{C}$ is $A^{-1}$-local if and only if $A \otimes X$ is
contractible.
\end{prop} 
\begin{proof} 
This follows easily from duality. 
In fact, $A \otimes X$ is contractible if and only if 
$\hom_{\mathcal{C}}(Y, A \otimes X)$ is contractible for all dualizable $Y$,
and this holds if and only if $\hom_{\mathcal{C}}(Y \otimes \mathbb{D}A, X)$ is
contractible for such $Y$. This is equivalent to the condition that $X$ be
$A^{-1}$-local.
\end{proof}

In particular, an object is $A^{-1}$-local if and only if it is $S$-local as
$S$ ranges over the collection of maps $\mathbb{D} A \otimes X \to 0$ for $X$
a dualizable object of $\mathcal{C}$. It follows by general theory that one can
construct an $A^{-1}$-localization of any object in $\mathcal{C}$. 
However, we can do so directly: 

\begin{cons}
Recall (\Cref{U_iseq}) the sequence $\mathbf{1} \to U_1 \to U_2 \to \dots $ and
the cofiber sequence 
\[ V_A \to \mathbf{1} \to U_A,  \]
with $U_A = \varinjlim U_i$.
Recall also that the cofiber of each $U_i \to U_{i+1}$ admits the structure of an
$A$-module. 

For any $X \in \mathcal{C}$, we consider the morphism $X \to X[A^{-1}]
:=X \otimes U_A$.
As shown in the proof of \Cref{VAacyclization}, 
the object $X[A^{-1}]$ is indeed $A^{-1}$-local.
Moreover, if $Y$ is any $A^{-1}$-local object, the natural map
$\hom_{\mathcal{C}}(X[A^{-1}], Y) \to \hom_{\mathcal{C}}(X, Y)$ is an
equivalence; this follows because the fiber is $\hom_{\mathcal{C}}(V_A \otimes
X, Y)$ and $V_A \otimes X$ belongs to $\Ator$ (compare \Cref{VAacyclization}). 

It follows from this that $X \to X[A^{-1}]$ is precisely
$A^{-1}$-localization, i.e., the left adjoint to the inclusion
$\mathcal{C}[A^{-1}] \subset \mathcal{C}$.
Note that, unlike $A$-completion, $A^{-1}$-localization is \emph{smashing}: it
is given by tensoring with $\mathbf{1}[A^{-1}] \simeq U_A$.
\end{cons}

\begin{example} 
If $\mathcal{C} = \md(\mathbb{Z})$ and $A = \mathbb{Z}/p$, then
$A^{-1}$-localization is precisely $p^{-1}$-localization, i.e., tensoring with
$\mathbb{Z}[p^{-1}]$.
\end{example} 

\begin{remark} 
These types of localizations are called \emph{finite localizations} in
\cite{Mil92}; here we are localizing away from the compact objects
$\mathbb{D}A  \otimes X$ for $X$ dualizable.
We refer also to \cite[\S 3.3]{HPS97} for a discussion of finite localizations. 
\end{remark} 

Our final goal in this section is to develop the theory of fracture squares,
and to show that $\mathcal{C}$ can be described using a combination of the
$A$-complete and the $A^{-1}$-local categories.
We begin by checking that equivalences can be detected after tensoring with $A$ and after
$A^{-1}$-localization. 

\begin{prop} 
\label{checkequiv}
Let $f\colon X \to Y$ be a morphism in $\mathcal{C}$. Then $f$ is an equivalence if
and only if both $1_A \otimes f\colon A \otimes X \to A \otimes Y$ and $
f[A^{-1}]\colon X[A^{-1}] \to Y[A^{-1}]$ are equivalences.
\end{prop} 
\begin{proof} 
We prove the non-obvious implication. 
Let $f\colon X \to Y$ be a morphism such that $1_A \otimes f$ and $f[A^{-1}]$ are
equivalences. Consider the localizing subcategory $\mathcal{A}$ of all $Z \in \mathcal{C}$ such that $1_Z
\otimes f$ is an equivalence. By hypothesis, $\mathcal{A}$ contains $A$ and
$U_A$. To show that it contains $\mathbf{1}$ (which is what we want), it
suffices to show that $V_A \in \mathcal{A}$. 
But  $V_A$
belongs to the smallest localizing subcategory containing the $A \otimes X$ for $X$
dualizable. The hypotheses imply that $A \otimes X \in \mathcal{A}$ for any $X
\in \mathcal{C}$, so that $V_A \in \mathcal{A}$ as desired. 
\end{proof} 

We are now ready to set up the arithmetic square.
Compare \cite[Prop. 4.13]{DwG02}. 
\begin{cons}
For any $X \in \mathcal{C}$, we have a commutative square
\begin{equation} \label{fracsquare} \xymatrix{
X \ar[d] \ar[r] & L_A X \ar[d] \\
 X[A^{-1}] \ar[r] &  (L_A X)[A^{-1}].
}\end{equation}
We will call this the \emph{$A$-arithmetic fracture square} of $X$.
\end{cons}

\begin{prop}\label{prop:frac_square_cart}
The fracture square
\eqref{fracsquare}
is cartesian.
\end{prop} 
\begin{proof} 
 In fact, one checks that the square is cartesian after
tensoring with $A$ (which annihilates the domain and codomain of the
bottom horizontal arrow, both of which are $A^{-1}$-local), and one checks
that the square is cartesian after applying $A^{-1}$-localization. Thus, by \Cref{checkequiv} we find
that \eqref{fracsquare}
is cartesian.

\end{proof} 

In particular, any object $X \in \mathcal{C}$ can be recovered from the
$A$-localization $L_A X$, the $A^{-1}$-localization $X[A^{-1}]$, and the
morphism $X[A^{-1}] \to (L_A X)[A^{-1}]$. Our next goal is to promote this to
an equivalence of stable $\infty$-categories.

\begin{cons}
Let $\mathrm{FracSquare}_A$ be the stable $\infty$-category defined by 
the homotopy fiber product
$\mathrm{FracSquare}_A = \mathrm{Fun}( \Delta^1, \mathcal{C}[A^{-1}])
\times_{\mathcal{C}[A^{-1}]} \Acp$.
Here:
\begin{enumerate}
\item The functor $\mathrm{Fun}( \Delta^1, \mathcal{C}[A^{-1}])
\to \mathcal{C}[A^{-1}]$ is given by evaluation at the vertex $1$.
\item The functor $\Acp \to \mathcal{C}[A^{-1}]$ is given by applying
$A^{-1}$-localization.
\end{enumerate}
In other words, to give an object in 
$\mathrm{FracSquare}_A$ amounts to giving a map of $A^{-1}$-local objects $X_1
\to X_2$, an $A$-complete object $X_0$, and an equivalence $X_2 \simeq
X_0[A^{-1}]$.

We thus obtain  a functor $\mathcal{C} \to \mathrm{FracSquare}_A$ sending $X$
to the associated fracture square.
\end{cons}

\begin{thm} 
\label{Afracsquare}
The functor $\mathcal{C} \to \mathrm{FracSquare}_A$ that sends $X \in
\mathcal{C}$ to the associated arithmetic square is an equivalence of
$\infty$-categories.
\end{thm} 
\begin{proof} 
We first check full faithfulness.
For ease of notation, we will write $X_{tA} :=(L_A
X)[A^{-1}]$. 
Consider the triple $(X[A^{-1}] \to X_{tA}, L_A X, \mathrm{id}\colon (L_A X)[A^{-1}]\simeq X_{tA}) \in \mathrm{FracSquare}_A$
associated to $X \in \mathcal{C}$.

Fix a triple $(Y_1 \to Y_2, Y_0, \phi\colon Y_0[A^{-1}]\simeq Y_2 )$ in
$\mathrm{FracSquare}_A$ and an object $X \in
\mathcal{C}$.
Then the space of maps between the object of
$\mathrm{FracSquare}_A$ associated to $X$ and this triple is computed as 
the homotopy fiber product
\[ \left(  \hom_{\mathcal{C}}(X[A^{-1}], Y_1)
\times_{\hom_{\mathcal{C}}(X[A^{-1}], Y_2)} \hom_{\mathcal{C}}(X_{tA}, Y_2) \right)
\times_{\hom_{\mathcal{C}}(X_{tA}, Y_2 )}
\hom_{\mathcal{C}}(L_A X, Y_0).
\]
Using the identifications $\hom_{\mathcal{C}}(X[A^{-1}], Y_i) \simeq
\hom_{\mathcal{C}}(X, Y_i)$ for $i=1,2$ and $\hom_{\mathcal{C}}(L_A X, Y_0) \simeq
\hom_{\mathcal{C}}(X, Y_0)$, we identify this fiber product with 
\[ \hom_{\mathcal{C}}( X, Y_1) \times_{\hom_{\mathcal{C}}( X, Y_{2})}
\hom_{\mathcal{C}}(X, Y_0) \simeq \hom_{\mathcal{C}}(X, Y_1 \times_{Y_2} Y_0),  \]

It follows that the functor $\mathcal{C} \to \mathrm{FracSquare}_A$ sending $X
\in \mathcal{C}$ to the associated arithmetic square admits a right adjoint $G$
that sends a 
triple $(Y_1 \to Y_2, Y_0, Y_0[A^{-1}] \simeq Y_2)$ to the pullback $Y_1
\times_{Y_2} Y_0$. As the composition $\mathcal{C} \to  \mathrm{FracSquare}_A \to
\mathcal{C}$ is homotopic to the identity by \Cref{prop:frac_square_cart}, we find that 
the left adjoint $\mathcal{C} \to  \mathrm{FracSquare}_A$ is fully faithful. In
order to show that we have an equivalence of $\infty$-categories, it therefore
suffices to show that the right adjoint is conservative, since we have a
\emph{colocalization}. This is checked as
follows: 
given $(Y_1 \to Y_2, Y_0, Y_0[A^{-1}] \simeq Y_2)$, we find that 
\[ (Y_1 \times_{Y_2} Y_0)[A^{-1}] \simeq Y_1, \quad L_A (Y_1 \times_{Y_2}
Y_0) \simeq Y_0, \quad 
Y_2 \simeq \left( L_A (Y_1 \times_{Y_2}
Y_0)\right)[A^{-1}].
\]
In particular, $Y_0, Y_1, Y_2$ can be recovered from the 
pullback, which implies that $G$ is conservative as desired. 
\end{proof} 

\begin{remark} 
We have tacitly used the following two standard facts about $\infty$-categories. Given a
(homotopy)
fiber product $\mathcal{A} = \mathcal{A}_1 \times_{\mathcal{A}_3} \mathcal{A}_2$, then we can
compute mapping spaces in $\mathcal{A}$ as the homotopy 
fiber product of mapping spaces in $\mathcal{A}_1, \mathcal{A}_2,
\mathcal{A}_3$. Second, in $\fun( \Delta^1, \mathcal{B})$ for an
$\infty$-category $\mathcal{B}$, we can compute
maps between a pair of objects $(X_1 \to X_2), (Y_1 \to Y_2)$ via
the homotopy fiber product $\hom_{\mathcal{B}}(X_2, Y_2)
\times_{\hom_{\mathcal{B}}(X_1, Y_2)} \hom_{\mathcal{B}}(X_1, Y_1).$
\end{remark} 

\section{Nilpotence}

Let $\mathcal{C}$ be a symmetric monoidal, stable $\infty$-category 
whose tensor product functor is exact in each variable.  Let $A \in \mathcal{C}$ be an algebra object. 
In this subsection, we develop the theory of nilpotence: that is, the generalization to
our setting of those objects in $\md(\mathbb{Z})$ annihilated by a power of
the prime number $p$. For the moment, we do \emph{not} assume anything as
strong as \Cref{hypothesis}. 

Recall that: 

\begin{definition} 
\label{thickdef}
A full stable subcategory $\mathcal{C}' \subset \mathcal{C}$ is called
\emph{thick} if $\mathcal{C}'$ is also idempotent-complete, i.e., every
idempotent endomorphism induces a splitting. 
A subcategory $\mathcal{I} \subset \mathcal{C}$ is called a \emph{$\otimes$-ideal} if
whenever $X \in \mathcal{I}$ and $ Y \in \mathcal{C}$, we have $X \otimes Y \in
\mathcal{I}$. One then obtains the notion of a \emph{thick $\otimes$-ideal}, which
is a full subcategory
that is both a thick subcategory and a $\otimes$-ideal.
\end{definition}

We list some common sources of thick $\otimes$-ideals. 
\begin{example} 
If $Z \in \mathcal{C}$, the collection of $X \in \mathcal{C}$ such that $X
\otimes Z$ is contractible is a thick $\otimes$-ideal.
Of course, in this case the collection is actually a localizing
$\otimes$-ideal too.
\end{example} 

\begin{example} 
\label{smashnilpideal}
Let $f\colon B \to C$ be a morphism. Consider the collection of all $X \in
\mathcal{C}$ such that $1_X \otimes f^{\otimes n} \colon X \otimes B^{\otimes n} \to
X \otimes C^{\otimes n}$ is nullhomotopic for $n \gg 0$. This is a thick
$\otimes$-ideal (which is generally not a localizing subcategory).
\end{example}

We can now make the main definition of this section. 
\begin{definition} 
\label{descent2}
Let $\mathcal{C}$ be as above and let $A \in \mathrm{Alg}(\mathcal{C})$. 
\begin{enumerate}
\item 
We will say that an object of $\mathcal{C}$ is \emph{$A$-nilpotent}
(\cite[Def. 3.7]{Bou79}) if it
belongs to the thick $\otimes$-ideal generated by $A$ (i.e., the smallest thick
$\otimes$-ideal containing $A$). 
We will let $\mathrm{Nil}_A \subset \mathcal{C}$ be the full subcategory
spanned by the $A$-nilpotent objects. 
\item 
 We will say that  $A
$ is \emph{descendable} (see \cite{balmer} and
\cite[\S 3]{galois}) if the thick $\otimes$-ideal
generated by $A$ is all of $\mathcal{C}$. 
\end{enumerate}
\end{definition} 

\begin{example} 
$A$ is descendable if and only if the unit object $\mathbf{1}$ is
$A$-nilpotent.
\end{example} 

\begin{example} 
\label{modulenilpotent}
Let $M $ be an $A$-module in $\mathcal{C}$. Then $M$ is $A$-nilpotent.
In fact, $M$ is a retract (in $\mathcal{C}$) of $A \otimes M$. 
\end{example}

We now give an important characterization of $A$-nilpotence 
in terms of the Adams tower. 
\begin{prop} \label{properties:nilpotent}
The following are equivalent for an object $M \in \mathcal{C}$ and $A \in
\mathrm{Alg}(\mathcal{C})$: 
\begin{enumerate}
\item $M$ is  $A$-nilpotent.
\item For all $N \gg
0$, the maps $T_N(A, M) \to M$ in the Adams tower are nullhomotopic.
\item There exists a finite tower in $\mathcal{C}$
\[ T'_N \to \dots \to T'_2 \to T'_1 \to T'_0 \simeq M,  \]
with the properties that: 
\begin{itemize}
\item For each $i$, the cofiber of $T'_i \to T'_{i-1}$ admits the structure of an $A$-module
object in $\mathcal{C}$.
\item The composite map $T'_N \to T'_0$ is nullhomotopic (in $\mathcal{C}$).
\end{itemize}
\end{enumerate}
\end{prop} 
If we write $I = \mathrm{fib}(\mathbf{1} \to A)$ as before, then the conditions state
that $M$ is $A$-nilpotent if and only if for some $N$, the map $M \otimes
I^{\otimes N} \to M$ is nullhomotopic. 
Note that fits into \Cref{smashnilpideal}. 
\begin{proof} 
We will prove the implications cyclically. 

$(1) \implies (2)$. Suppose $M$ is $A$-nilpotent;
then we want to show that the maps $T_N(A, M) \to M$ are nullhomotopic for $N
\gg 0$. Observe that the collection of $M \in \mathcal{C}$ for which this
satisfied is a $\otimes$-ideal, since we have natural isomorphisms $T_i(A, M )
\simeq T_i(A, \mathbf{1}) \otimes M$. It is easy to see that the collection of
such $M$ is in addition thick, since the passage from $M$ to its Adams tower
commutes with cofiber sequences and retracts. 
Therefore, the collection of $M$
with the desired property is a thick $\otimes$-ideal. To show that it contains
every $A$-nilpotent object, it suffices to show that it contains $A$ itself. 
In other words, we need to show that the $A$-based Adams tower has this
property for $A$. But in this case, the map $T_1(A, A) \to T_0(A, A)$ is
already nullhomotopic as it is the fiber of the map $A \to A \otimes A$
(which has a section). This completes the proof
that $(1) \implies (2)$.

$(2) \implies (3)$. We can take the Adams tower, in view of \Cref{Adamstowprop}
and the hypotheses.

$(3) \implies (1)$.
In this case, the cofiber of each map $T'_i \to T'_{i-1}$ is $A$-nilpotent since it
admits the structure of an $A$-module (\Cref{modulenilpotent}).
Using the octahedral axiom and induction, it follows that 
the cofiber of $T'_N \to T'_0$ is $A$-nilpotent, since the class of
$A$-nilpotent objects is closed under cofiber sequences. 
Since this map is nullhomotopic, it follows that $M \simeq T'_0$ is a retract of the
cofiber of $T'_N \to T'_0$ and is thus $A$-nilpotent itself. 
\end{proof}

The above result enables one to quantify the notion of $A$-nilpotence.
\begin{definition} 
Suppose $M$ is $A$-nilpotent. We will write $\exp_A(M)$ for the smallest
integer $N \geq 0$
such that $T_N(A, M) = I^{\otimes N} \otimes M \to M$ is nullhomotopic and
call it the \emph{$A$-exponent of $M$}. \end{definition}

Using the axioms and \Cref{adamscobar}, one sees easily the following result (whose proof we leave to
the reader). 
\begin{prop}\label{prop:basiconexp} 
\begin{enumerate}
\item If $M$ is $A$-nilpotent and $M'$ is a retract of $M$, then $\exp_A(M') \leq \exp_A(M)$. 
\item If $M' \to M \to M''$ is a cofiber sequence of $A$-nilpotent objects,
then $\exp_A(M) \leq \exp_A(M') + \exp_A(M'')$. 
\item The exponent $\exp_A(M)$ is the smallest choice of $N$ that 
one can take in the second (or third) condition of \Cref{properties:nilpotent}.
\item The exponent $\exp_A(M)$ is the smallest $N \geq 0$ such that the map  $M
\to \mathrm{Tot}_{N-1}(M \otimes \cb(A))$ admits a retraction.
\end{enumerate}
\end{prop}

\begin{remark} 
\label{expisolderidea}
The quantification of nilpotence in this way is a special case of older ideas. 
For example, it can be obtained as a special case of the
discussion in \cite[Sec. 2]{ABIM}. 
In fact, we see that if $\mathcal{A}$ is the full subcategory of $\mathcal{C}$
consisting of objects of the form $ A \otimes X, X \in \mathcal{C}$, then 
$\exp_A(M)$ is precisely the \emph{$\mathcal{A}$-level} \cite[Def. 2.3]{ABIM}
of $M$.  

The idea also appears (earlier) in \cite{Chr98} as follows. 
When $A$ is dualizable, one has a \emph{stable projective class} (cf.
\cite[Sec. 2.3]{Chr98}) given by the pair $(\mathcal{A}, \mathcal{I})$ where
$\mathcal{A}$ is in the previous paragraph and $\mathcal{I}$ consists of all maps $X \to Y$ such that $X \otimes A \to Y
\otimes A $ is null. To see this, we observe that for any $X \in
\mathcal{A}$, 
we have that $\mathbb{D} A \otimes X \in \mathcal{A}$ and 
the map $\mathbb{D}A \otimes X
\to X$ splits after tensoring with $A$, as one sees by dualizing the fact that $1 \to
A$ splits after tensoring with $A$.   
Using duality, one sees also that for any $X \in \mathcal{A}$ and $f\colon Y \to Z$
in $\mathcal{I}$, the map $\pi_* \hom_{\mathcal{A}}(X, Y) \to \pi_*
\hom_{\mathcal{A}}(X, Z)$ is zero. From this, it is easy to verify that one
has a projective class, via \cite[Lem. 3.2]{Chr98}. 
The collection of objects of $A$-nilpotence at most $n$ is precisely
$\mathcal{A}_n$ in the sense of \cite[Sec. 3.2]{Chr98}.  
\end{remark} 

\begin{example} 
Consider the usual test example of $\mathcal{C} = \md(\mathbb{Z})$ and $A =
\mathbb{Z}/p$. In this case, an object $X$ is $A$-nilpotent if and only if
multiplication by $p^n$ annihilates it for some $n$: that is, if 
$p^n\colon X \to X$ is nullhomotopic for some $n$. This follows from
\Cref{properties:nilpotent} in view of the explicit description of the Adams
tower in this case (\Cref{adamstowerZ}). 
Moreover, one sees that $\exp_A(M)$ is the smallest $n$ such that $p^n$
annihilates $M$.  
Note that in this case, the Adams spectral sequence is precisely the Bockstein
spectral sequence. 
\end{example}

The theory of {exponents} leads to an exhaustive 
\emph{filtration} of the $\infty$-category $\mathrm{Nil}_A \subset \mathcal{C}$
of $A$-nilpotent objects, 
\[ \mathrm{Nil}_A^{(0)} \subset  \mathrm{Nil}_A^{(1)} \subset
\mathrm{Nil}_A^{(2)} \subset \dots  \subset \bigcup_i
\mathrm{Nil}_A^{(i)} =\mathrm{Nil}_A, \]
where an object $X \in \mathrm{Nil}_A$ belongs to $\mathrm{Nil}_A^{(i)}$
if and only if $\exp_A(X) \leq i$. 
It is easy to see that $\mathrm{Nil}_A^{(i)}\subset \mathcal{C}$ is a stable full subcategory which is closed under retracts, 
arbitrary direct sums (although $\mathrm{Nil}_A$ is only closed under
finite direct sums), and tensoring with arbitrary elements of $\mathcal{C}$. For example, $\mathrm{Nil}_A^{(0)} = \left\{0\right\}$ and
$\mathrm{Nil}_A^{(1)}$ consists of the retracts of $A$-modules. 
Moreover, the cofiber of a map from an object in $\mathrm{Nil}_A^{(i)}$ to an
object in $\mathrm{Nil}_A^{(j)}$ belongs to $\mathrm{Nil}_A^{(i+ j)}$.

\begin{corollary} 
\label{thicksubcatmodulesisnilpotent}
Given $\mathcal{C}$ and $A$ as above, the collection of $A$-nilpotent objects
is the thick subcategory generated by those objects in $\mathcal{C}$ which
admit the structure of an $A$-module.
\end{corollary} 
\begin{proof} 
This follows from the third property in \Cref{properties:nilpotent}. 
\end{proof} 

\begin{corollary} 
\label{Anilplaxmonoidal}
Let $(\mathcal{C}, \otimes, \mathbf{1}) $ and $(\mathcal{D}, \otimes,
\mathbf{1})$ be symmetric monoidal stable $\infty$-categories where the
tensor structure is compatible with colimits. Let $F\colon \mathcal{C} \to
\mathcal{D}$ be a lax symmetric monoidal, exact functor. 
Let $A \in \mathrm{Alg}(\mathcal{C})$ and let $M \in \mathcal{C}$. Then if $M$
is $A$-nilpotent, $F(M)$ is $F(A)$-nilpotent, and in fact, one has
$\exp_{F(A)}(F(M))\leq\exp_A(M)$.
\end{corollary}

\begin{proof} 
The first assertion follows from the third condition of \Cref{properties:nilpotent}, since
that condition is preserved under any exact lax symmetric monoidal functor. The second assertion
follows using \Cref{prop:basiconexp}.3.
\end{proof}

We now prove another characterization of $A$-nilpotence. In classical terms,
this characterization states the $A$-Adams filtration on
$\hom_{\mathcal{C}}(\cdot, M)$
should have a uniform bound (namely, the $A$-exponent of $M$). 
\begin{prop} 
\label{compositenilpotent}
The following are equivalent for $M \in \mathcal{C}$: 
\begin{enumerate}
\item $M$ is $A$-nilpotent.  
\item There exists an integer $N$ such that given any sequence 
in $\mathcal{C}$
\[ M_N \stackrel{\phi_N}{\to} M_{N-1} \stackrel{\phi_{N-1}}{\to} \dots
\stackrel{\phi_2}{\to} M_1  \stackrel{\phi_1}{\to} M,  \]
such that each $\phi_i$ becomes nullhomotopic after tensoring with $A$, the
composition $\phi_1 \circ \dots \circ \phi_N\colon M_N \to M$ is nullhomotopic.
\end{enumerate}
If $M$ is $A$-nilpotent, the smallest $N$ satisfying condition 2.~is $\exp_A(M)$.
\end{prop} 
\begin{proof} 
$(2) \implies (1)$. Take the Adams tower of $M$, $\left\{T_i(A, M)\right\}$.
Each successive map in this tower becomes null after tensoring with $A$. Therefore,
there is a composition $T_N(A, M) \to M$ for $N \gg 0$ which is nullhomotopic,
which by \Cref{properties:nilpotent} implies that $M$ is $A$-nilpotent. 

$(1) \implies (2)$. Consider the map $M_1 \stackrel{\phi_1}{\to} M \to A
\otimes M$. It fits into a commutative square
\[ \xymatrix{
M_1 \ar[d] \ar[r]^{\phi_1} & M \ar[d] \\
A \otimes M_1  \ar[r]^{1_A \otimes \phi_1} &  A \otimes M.
}\]
By assumption, the bottom horizontal map is nullhomotopic. It follows that the
composition $M_1 \to A \otimes M$ is nullhomotopic, and in particular we have a
commutative diagram
\[ \xymatrix{
 & T_1(A, M) \ar[d]  \\
 M_1 \ar[r]^{\phi_1}  \ar[ru]^{\overline{\phi_1}} &  M,
}\]
i.e., $\phi_1$ lifts to $T_1(A, M)$. 

Now, it follows that $\phi_1 \circ \phi_2\colon M_2 \to M$ lifts to $T_1(A, M)$ as
well, via $\overline{\phi_1} \circ \phi_2$. The map $\overline{\phi_1} \circ
\phi_2$ becomes nullhomotopic after tensoring with $A$
since $\phi_2$ does. Therefore, $\phi_1 \circ \phi_2$ lifts to $T_2(A, M)$.
Similarly, $\phi_1 \circ \dots \circ \phi_k$ lifts to $T_k(A, M)$. If $k \gg
0$, the map $T_k(A, M) \to M$ is nullhomotopic, so $\phi_1 \circ \dots \circ
\phi_k$ must be nullhomotopic. 

We leave the identification of the smallest such choice of $N$ and $\exp_A(M)$
to the reader.
\end{proof}

\begin{cor} 
\label{nil:closedundercotensor}
Suppose $\mathcal{C}$ is a closed symmetric monoidal $\infty$-category.
	Suppose $M \in \mathcal{C}$ is $A$-nilpotent and $X \in \mathcal{C}$ is arbitrary. Then
	the internal mapping object $\underline{\mathrm{Hom}}(X, M)$ is $A$-nilpotent. That
	is, $\mathrm{Nil}_A$ is closed under cotensors. 
\end{cor} 
\begin{proof} 
By \Cref{thicksubcatmodulesisnilpotent}, it suffices to consider the case where $M$ is an $A$-module. But in this case, the internal  mapping object
$\underline{\mathrm{Hom}}(X, M)$ also inherits the structure of an $A$-module
and is therefore $A$-nilpotent. (In fact, this argument shows that $\exp_A
\underline{\mathrm{Hom}}(X, M) \leq \exp_A(M)$; when $X=M$ this implies that the $A$-exponent of $M$ is equal to the $A$-exponent of its endomorphism algebra.)
\end{proof} 

Next, we show that whether or not an $A$-nilpotent object belongs to a thick
$\otimes$-ideal can be tested after tensoring with $A$: a sort of descent property.

\begin{prop} 
\label{tensorideal:desc}
Let $M \in \mathcal{C}$ be $A$-nilpotent and let $\mathcal{I} \subset
\mathcal{C}$ be any thick $\otimes$-ideal. Then we have $A \otimes M \in \mathcal{I}$
if and only if $M \in \mathcal{I}$.
\end{prop} 
\begin{proof} 
We will prove the non-trivial implication. The hypotheses imply that $Q\otimes M \in \mathcal{I}$ for any $A$-module object $Q$.
Consider 
the Adams tower $\left\{T_i(A, M)\right\} \simeq \left\{T_i (A,
\mathbf{1}) \otimes M\right\}$. The cofiber of each map $T_i(A, M) \to
T_{i-1}(A, M) $ is of the form $Q \otimes M$ where
$Q$ admits the structure of an $A$-module. In particular, it belongs to
$\mathcal{I}$. Since $\mathcal{I}$ is a thick subcategory, it follows that the cofiber of each $T_k(A, M) \to
M$ belongs to $\mathcal{I}$ and, since these maps are null for $k \gg 0$, it
follows that $M \in \mathcal{I}$ too.
\end{proof}

We now include the analog of \Cref{complAB,torsAB} in the nilpotent
case. 
\begin{prop} 
\label{nilpAB}
Let $A, B \in \mathrm{Alg}(\mathcal{C})$ be algebra objects and let $
X \in \mathcal{C}$. Then $X$ is $(A \otimes B)$-nilpotent if and only if it is
both $A$-nilpotent and $B$-nilpotent.
\end{prop} 
\begin{proof} 
Suppose $X$ is both $A$-nilpotent and $B$-nilpotent. Then we want to show that $
X$ is $(A \otimes B)$-nilpotent.
This is a straightforward application of \Cref{tensorideal:desc} with
$\mathcal{I} = \mathrm{Nil}_{A \otimes B}$. 
Indeed, since $X$ is $B$-nilpotent, we find that $X \in
\mathrm{Nil}_{A \otimes B}$ if and only if $B \otimes X \in \mathrm{Nil}_{A
\otimes B}$. Since $B \otimes X$
is $A$-nilpotent, we find that $B \otimes X \in \mathrm{Nil}_{A
\otimes B}$ if and only if $A
\otimes B \otimes X \in \mathrm{Nil}_{A
\otimes B}$. But $A \otimes B \otimes X$ is an $(A
\otimes B)$-module and clearly belongs to $\mathrm{Nil}_{A
\otimes B}$. This proves that $X$
is $(A \otimes B)$-nilpotent. We leave the other (easier) direction to the
reader. 
\end{proof}

We now prove some (slightly) less formal results about $A$-nilpotence. In 
the rest of this section, the compactness and dualizability hypotheses will become
important. In particular, for the rest
of this section, we assume \Cref{hypothesis}. 

\begin{prop} Let $\mathcal{C}, A$ satisfy \Cref{hypothesis} (so that in
particular, $\mathcal{C} $ is presentable). 
Let $M \in \mathcal{C}$ be compact. Then the following are equivalent: 
\begin{enumerate}
\item $M $ is $A$-nilpotent.   
\item $M$ is $A$-torsion.
\end{enumerate}
\end{prop} 
\begin{proof} 
If $M$ is $A$-torsion, then 
$M$ is a compact object of $\Ator$ (since $M$ is compact in $\mathcal{C}$),
which means that it belongs to the thick subcategory generated by the compact
generators $\mathbb{D}A \otimes X$
for $X \in \mathcal{C}$ dualizable:
in fact, $\Ator$ is equivalent to the $\mathrm{Ind}$-completion of precisely
this thick subcategory.
This implies that $M$ is $A$-nilpotent.
The other direction is evident.
\end{proof} 

\begin{thm} 
\label{ringobjectnilp} Suppose $\mathcal{C}, A$ satisfy \Cref{hypothesis}.
Let $X \in \mathcal{C}$ be such that $X$ admits 
a unital multiplication in $\mathrm{Ho}(\mathcal{C})$. Then the following are equivalent: 
\begin{enumerate}
\item $X$ is $A$-nilpotent. 
\item $X$ is $A$-torsion (equivalently, $X[A^{-1}]$ is contractible).
\end{enumerate}
\end{thm} 
\begin{proof} 
We establish the non-obvious implication.
Since $X$ is $A$-torsion, the sequence $\mathbf{1}
\simeq U_0 \to U_1 \to \dots $ with colimit $U_A$ has the property that $U_A
\otimes X \simeq 0$, so that in particular, the map $\mathbf{1} \to X \to U_A \otimes X$ is nullhomotopic. 
Since $\mathbf{1}$ is compact, this means that the
composition $\phi_N \colon \mathbf{1} \to U_N  \to U_N \otimes X$ is nullhomotopic for $N \gg 0$.
As a result, the map $\psi_N\colon X \to U_N \otimes X$ must be nullhomotopic too, for such
$N$, as one sees using the unitary multiplication on $X$.
So $X$ is a retract of $V_N\otimes X$, which is $A$-nilpotent.

\end{proof}

\begin{remark} 
\Cref{ringobjectnilp} corresponds to the following simple observation: given a
(discrete) unital ring $R$ which is all $p$-power torsion, there exists a
uniform $n$ such that $p^n R = 0$ (namely, we can take $n$ so large that
$p^n . 1 = 0 \in R$). 
\end{remark} 

We now include a result that describes thick $\otimes$-ideals generated by a
single dualizable object in terms of nilpotence.

\begin{prop} 
Let $Y \in \mathcal{C}$ be dualizable and let $X \in \mathcal{C}$ be arbitrary.
Let $R =Y \otimes \mathbb{D} Y$
be the internal ring of endomorphisms of $Y$, so $R \in \mathrm{Alg}
(\mathcal{C})$.
Then $X$ belongs to the thick $\otimes$-ideal generated by $Y$ if and only if $X$
is $R$-nilpotent.
\end{prop} 
\begin{proof} 
Let $\mathcal{I}$ be the thick $\otimes$-ideal generated by $Y$. We want to show
that an object belongs to $\mathcal{I}$ if and only if it is $R$-nilpotent. To
show that $Z \in \mathcal{I}$ implies $Z$ is $R$-nilpotent, 
it suffices to consider the case 
$Z = Y$. In this case, $Y$ is an $R$-module and hence $R$-nilpotent.

Finally, to show that every $R$-nilpotent object belongs to 
$\mathcal{I}$, it suffices to show that $R$ does. But $R \simeq \mathbb{D}Y
\otimes Y$ and thus clearly belongs to $\mathcal{I}$.
\end{proof}

Finally, we treat a special class of examples, where \emph{every} torsion
object is nilpotent. We will encounter this in the sequel to this paper when we
discuss $\sF$-nilpotence results. 
The main result is that such phenomena only arise when one has an idempotent
splitting of the unit, so that the $\infty$-category itself decomposes as a
product.
\begin{prop} 
\label{idempotentsplitnilp}
The following are equivalent for a dualizable algebra object $A \in
\mathrm{Alg}(\mathcal{C})$:
\begin{enumerate}
\item We have an equality $\Ator = \Acp$ of subcategories of $\mathcal{C}$
(i.e., an object is $A$-torsion if and only if it is $A$-complete).
\item The inclusion $\mathrm{Nil}_A \subset \Ator$ is an equality.
\item The inclusion $\mathrm{Nil}_A \subset \Acp$ is an equality.
\item The map $\mathbf{1} \to L_A \mathbf{1} \times \mathbf{1}[A^{-1}]$ is an
equivalence and the symmetric monoidal $\infty$-category $\mathcal{C}$ decomposes as a
product $\mathcal{C} \simeq \Acp \times \mathcal{C}[A^{-1}]$. 
\item The localization $(L_A \mathbf{1})[A^{-1}]$ is contractible.
\end{enumerate}
\end{prop} 
\begin{proof} 
We first show that $(2)$ implies $(1)$. Every torsion object is $A$-complete, since nilpotent
objects are complete. To see that conversely, an $A$-complete object $X$ is $A$-torsion, recall the equivalence
$L_A:\Ator\cong\Acp$ to write $X\cong L_A(L_A^{-1}(X))\cong L_A^{-1}(X)$, where the second equivalence follows
because the torsion object $L_A^{-1}(X)$ is nilpotent, and $L_A$ acts as the identity on nilpotent objects.

Again via the equivalence $L_A$ between $\Ator$ and $\Acp$ (which acts as the identity of
$\mathrm{Nil}_A$), one sees easily that $(2)$ and $(3)$ are equivalent. 
Moreover, $(3)$ implies that $L_A \mathbf{1}$ is $A$-torsion, so that $(5)$
holds; if conversely $(5)$ holds, the algebra $L_A \mathbf{1}$ is $A$-torsion and therefore 
$A$-nilpotent by \Cref{ringobjectnilp}, and any $A$-complete object, as a
module over $L_A \mathbf{1}$, is also $A$-nilpotent, so that $(3)$ holds.

We now show that $(3)$ implies $(4)$.
Consider the natural map
$\phi\colon \mathbf{1} \to L_A \mathbf{1} \times \mathbf{1}[A^{-1}]$. 
Observe that $L_A \mathbf{1}$ is $A$-complete by definition, so it is
$A$-torsion by assumption. Thus, 
$\phi$ becomes an equivalence after $A^{-1}$-localizing since the first factor
on the right-hand side has trivial $A^{-1}$-localization. Moreover, $\phi$ also becomes
an equivalence after tensoring with $A$ (this does not use any of our
assumptions). As a result, $\phi$ is an equivalence by \Cref{checkequiv}. It
follows easily that one gets a decomposition of $\mathcal{C}$ as desired. 
(Note that the above decomposition of $\mathcal{C}$ also follows from
\Cref{Afracsquare}.) 

To see that $(4)$ implies $(1)$, observe that $X\in\mathcal{C}$ is $A$-torsion if and only if
its image in $\mathcal{C}[A^{-1}]$ is contractible, if and only if, by the decomposition in $(4)$,
$X$ is $A$-complete.

Finally, $(1)$ implies $(5)$ because $L_A\mathbf{1}$ is $A$-complete, hence $A$-torsion,
so that $(L_A\mathbf{1})[A^{-1}]\cong *$.

%4.21 : the pf of the decomposition in item 4 could be fleshed out:
%the reference to 3.20 (fracture square) might be misleading because 
%3.20 explicitly is not an equivalence of _symmetric monoidal_ infty-categories 
%i think the sole reason being that the tensor product of A-complete objects is 
%_the completion_ of their tensor product in C ; hence in the special situation of
%4.21, the tensor product, formed in C, of any two A-complete objects is 
%still A-complete because it is in fact A-nilpotent; in this case, the functor Frac_A
%from 3.20 _is_ symmetric monodidal)

\end{proof}

\part{$G$-equivariant spectra and $\sF$-nilpotence}\label{part:equi}
\section{$G$-spectra}

Let $G$ be a compact Lie group. 
In this section, we quickly review the basic facts about the
homotopy theory 
$\GSpec$ of (genuine) $G$-spectra, which we will treat as an $\infty$-category. 
Since a full exposition of $\GSpec$ using $\infty$-categories rather than
model categories has not yet appeared in the literature, we
have included a discussion, beginning with a review of the relationship between
model and $\infty$-categories.
This is by no means intended to be a treatment of the classical theory and we refer to 
sources such as
\cite{May96, LMSM86, Die87a, MaM02, Ada84} for introductions to equivariant stable
homotopy theory.

For our purposes, we will take $\GSpec$ to be the $\infty$-category associated
to the symmetric monoidal model category of \emph{orthogonal $G$-spectra}. 
Although it is possible to construct $\GSpec$ purely $\infty$-categorically via
the theory of \emph{spectral Mackey functors}
(cf.\ \cite{Barwick}), we will need to use the existence of models of certain $\e{\infty}$-algebras in $\GSpec$ (namely,
equivariant real and complex $K$-theory), even though an
$\infty$-categorical treatment (and new construction) is to appear
in Lurie's forthcoming work on elliptic cohomology (see \cite{Lur09b} for a survey).

\subsection{Model categories and $\infty$-categories}

In this subsection, we begin by recalling how one passes from (symmetric monoidal) model categories to
(symmetric monoidal) $\infty$-categories. Suppose that $\mathcal{C}$ is a model category with weak
equivalences $\mathcal{W}$. 
\newcommand{\p}[1]{\underline{#1}}
\begin{cons}
Let $\mathcal{C}^{c} \subset \mathcal{C}$  
be the full subcategory
spanned by the cofibrant objects.  
The model category $\mathcal{C}$ \emph{presents} an $\infty$-category
$\p{\mathcal{C}}$
which, by definition, is the $\infty$-categorical localization
$\p{\mathcal{C}} :=\mathcal{C}^c[\mathcal{W}^{-1}]$ \cite[Def.\ 1.3.4.15]{Lur14}. 
\end{cons}

In case $\mathcal{C}$ is a \emph{simplicial} model category, one knows that the
localization 
$\p{\mathcal{C}} = \mathcal{C}^c[\mathcal{W}^{-1}]$ can also be described as the homotopy
coherent nerve of the fibrant \emph{simplicial} category spanned by the
cofibrant-fibrant objects of $\mathcal{C}$ \cite[Th.\ 1.3.4.20]{Lur14}. 
Given  a Quillen equivalence $(F, G)\colon \mathcal{C} \rightleftarrows \mathcal{D}$
between model categories admitting \emph{functorial} cofibrant and fibrant
replacements, the induced functor $\p{F}\colon \p{\mathcal{C}} \to \p{\mathcal{D}}$
(obtained via universal properties) is an equivalence of $\infty$-categories in
view of \cite[Lem.\ 1.3.4.21]{Lur14}. Note that the cited theorem assumes that $\mathcal{C}$ and
$\mathcal{D}$ are in addition combinatorial, but only uses the functorial
fibrant and cofibrant replacement functors. 
We recall that if $\mathcal{C}, \mathcal{D}$ are cofibrantly generated, then
the existence of functorial factorizations is well-known. In fact, such factorizations are so fundamental to the theory they are sometimes part of the definition of a model category \cite{Hov99}. For a modern
reference, see \cite[\S 12.1]{Rie14}.

\begin{example} 
Let $\mathrm{Top}$ denote the model category of compactly generated  topological
spaces (i.e., weak Hausdorff $k$-spaces) with the Quillen model structure (where weak 
equivalences are weak homotopy equivalences and fibrations are Serre
fibrations). Then $\p{\mathrm{Top}}$ is the $\infty$-category $\mathcal{S}$ of
spaces. 
\end{example}

\begin{comment}
We recall an important example of such a presentation. This example is
fundamental in that it shows that the $\infty$-categorical notion of a diagram
is well-behaved. 
\begin{example}[{\cite[Prop.~4.2.4.4]{Lur09}}]
Let $\mathcal{C}$ be a combinatorial simplicial model category 
and let $\mathcal{D}$ be a
simplicial category corresponding to an $\infty$-category
$\widetilde{\mathcal{D}}$. In this case, the (simplicial) category 
of simplicial functors $\mathrm{Fun}_{\Delta}(\mathcal{D}, \mathcal{C})$ admits
a model structure  where the weak equivalences and fibrations are detected pointwise: 
the \emph{projective} model structure. 
The $\infty$-category 
$\p{\mathrm{Fun}_{\Delta}(\mathcal{D}, \mathcal{C})}$
that 
$\mathrm{Fun}_{\Delta}(\mathcal{D}, \mathcal{C})$ presents is precisely the
$\infty$-categorical functor category $\fun(
\widetilde{\mathcal{D}}, \p{\mathcal{C}})$.
\end{example} 

\end{comment}

We next recall the construction of \emph{symmetric monoidal}
$\infty$-categories from symmetric monoidal model categories. 
We begin with some preliminaries.
\begin{definition}[{\cite{ScS00}}]
Let $(\mathcal{C}, \otimes, \mathbf{1})$ be a model category with a closed symmetric monoidal structure. Suppose the unit
is cofibrant and  that for cofibrations $c \to c', d \to d'$, the natural
pushout-product map
\[ c \otimes d' \sqcup_{c \otimes d} c' \otimes d \to c' \otimes d'\]
is a cofibration, and a weak equivalence if either of the maps $c \to c'$ or
$d \to d'$ is a weak equivalence.
In this case, $(\mathcal{C}, \otimes, \mathbf{1})$ is called a \emph{symmetric
monoidal model category.}
\end{definition}

This definition appears in \cite{ScS00}, which 
replaces cofibrancy of the unit with a slightly weaker condition.  
In the case where $\mathcal{C}$ has a symmetric monoidal model structure, this
can be used to construct symmetric monoidal $\infty$-categories.

\begin{cons}[{\cite[Prop.~4.1.3.4]{Lur14}}]
\label{firstmodeltoinfinity}
Suppose $(\mathcal{C}, \otimes, \mathbf{1})$ is a symmetric monoidal model
category. As before, let $\mathcal{C}^c \subset \mathcal{C}$ be the full subcategory
spanned by the cofibrant objects (so that $\mathcal{C}^c$ is a monoidal
subcategory) and $\mathcal{W}$ the class of weak equivalences in
$\mathcal{C}^{c}$. Then, since the class $\mathcal{W}$ is compatible with the
symmetric monoidal structure on $\mathcal{C}^c$,  the $\infty$-categorical localization
$\underline{\mathcal{C}} = \mathcal{C}^c[\mathcal{W}^{-1}]$ inherits a symmetric monoidal structure such
that $\mathcal{C}^c \to \underline{\mathcal{C}}$ is symmetric
monoidal.\footnote{In practice, $\mathcal{C}$ is frequently a \emph{simplicial}
model category. However, in this construction, one considers $\mathcal{C}^c$ as a discrete
category with weak equivalences, and ignoring the simplicial structure.} 
\end{cons}

In case $\mathcal{C}$ is a \emph{simplicial} symmetric monoidal category, there
is an equivalent version of this construction that is often easier to work with. 

\begin{cons}
\label{secondmodeltoinfinity}
Suppose $(\mathcal{C}, \otimes, \mathbf{1})$ is a simplicial symmetric
monoidal model category. 
Let $\mathcal{C}^{cf} \subset \mathcal{C}$ denote the full subcategory spanned
by the cofibrant-fibrant
objects. Consider the colored operad in simplicial sets 
whose objects are ordered tuples of the objects of $\mathcal{C}^{cf}$ and such that the morphisms
between $\left\{X_1, \dots, X_n\right\} \in \mathcal{C}^{cf}$ and $Y \in
\mathcal{C}^{cf}$ are given by the simplicial mapping object 
$\mathrm{Map}_{\mathcal{C}}(X_1 \otimes \dots \otimes X_n, Y)$. The associated 
$\infty$-operad defines a symmetric monoidal structure on the
homotopy coherent nerve of $\mathcal{C}^{cf}$ which is canonically equivalent to the symmetric monoidal structure on $\underline{\mathcal{C}} = \mathcal{C}^c[\mathcal{W}^{-1}]$ from \Cref{firstmodeltoinfinity} \cite[Prop.~4.1.3.10, Cor.~4.1.3.16]{Lur14}.
\end{cons}

% By \cite[Cor. 4.1.3.16]{Lur14}, the two symmetric monoidal $\infty$-categories
% of \Cref{firstmodeltoinfinity} and \Cref{secondmodeltoinfinity} 
% are canonically equivalent. 

We note also that this construction is functorial 
in symmetric monoidal Quillen adjunctions. 
If $(\mathcal{C}, \otimes, \mathbf{1}_{\mathcal{C}})$ and 
$(\mathcal{D}, \otimes, \mathbf{1}_{\mathcal{D}})$
are symmetric monoidal model categories  and 
if $F\colon \mathcal{C} \to \mathcal{D}$ is a symmetric monoidal left Quillen functor, then $F$
induces a symmetric monoidal functor 
of $\infty$-categories
\[  \underline{F}\colon\underline{\mathcal{C}} \to \underline{\mathcal{D}}, \]
by the universal property of localization.

\newcommand{\topp}{\mathrm{Top}_*}

\begin{example} 
Let $\topp$ denote the model category of \emph{pointed} compactly generated topological spaces with the
usual 
Quillen model structure. 
Then $\topp$ is a symmetric monoidal model category with the smash product.
We have a symmetric monoidal equivalence $\mathcal{S}_*
\simeq \underline{\topp} $. 
\end{example}

\subsection{$G$-spaces and $G$-spectra}

We will now review the $\infty$-categories of $G$-spaces and $G$-spectra, and
some of the basic functoriality in $G$ that they possess.

\begin{cons} 
\label{topG}
Let $G$ be a compact Lie group. 
Let $\topgp{G}$ denote the category of pointed
compactly generated
topological spaces equipped with a $G$-action (fixing the basepoint). 
We regard $\topgp{G}$ as a model category where a morphism $X \to Y$ is a weak
equivalence (resp. fibration) if and only if for each closed subgroup $H \leq G$, $X^H
\to Y^H$ is a weak homotopy equivalence (resp.\ Serre fibration). The generating
cofibrations in $\topgp{G}$ are the morphisms $(G/H \times S^{n-1})_+ \to (G/H
\times D^n)_+$ for $n \geq 0$. 
Via the smash product of pointed $G$-spaces, this is a symmetric monoidal model
category (with unit given by $S^0$). We refer to \cite[III.1]{MaM02} for a
treatment of this model category. 

Similarly, there is a model category $\topgpu{G}$ of (unpointed) compactly
generated topological spaces equipped with a $G$-action, where the weak
equivalences and fibrations are detected on fixed points for closed subgroups. 
Via the cartesian  product of $G$-spaces, this is a symmetric monoidal model category.
\end{cons} 

\begin{definition} 
\label{def:Gspaces}
The $\infty$-category
$\GS$ of \emph{$G$-spaces} is the symmetric monoidal $\infty$-category
associated
to the symmetric monoidal model category $\topgpu{G}$ of 
\Cref{topG}. We define $\GSpt$ similarly from $\topgp{G}$ and call it the
$\infty$-category of \emph{pointed $G$-spaces.}
\end{definition} 

We now discuss the analog for $G$-spectra.

\begin{example} 
\label{orsg}
The category $\orsg$  of orthogonal $G$-spectra \cite{MaM02}, equipped with the stable model
structure and the smash product, is an example of a symmetric monoidal model category.
The pushout-product axiom is \cite[III.7.5]{MaM02}, and the unit is cofibrant (``$q$-cofibrant'') as well. 
\end{example}

\begin{definition}\label{def:G_spectra}
The symmetric monoidal $\infty$-category $\GSpec$ of \emph{$G$-spectra} is the
symmetric
monoidal $\infty$-category associated to 
the symmetric monoidal model category $\orsg$ of \Cref{orsg}. 
As is customary, we will denote the monoidal product by $\wedge$ and the unit by either $S^0$ or
$S^0_G$ (depending on whether the group is clear from the context). 
We will also write $F(X, Y)$ for the internal mapping object for $X, Y \in
\GSpec $ (i.e., the function spectrum). 
\end{definition}

One has a symmetric monoidal left Quillen functor
\[ \Sigma^\infty\colon \topgp{G} \to \orsg,  \]
and as a result one obtains a symmetric monoidal left adjoint functor
\[ \Sigma^\infty\colon \GSpt  \to \GSpec,  \]
with right adjoint $\Omega^\infty$.

\begin{example} 
For $H \leq G$ a closed subgroup, we consider the $G$-space $G/H$ and the
pointed $G$-space $G/H_+$. The suspension spectra $\{\Sigma^\infty_+ G/H \in
\GSpec\}_{H \leq G}$ form a system of compact generators of $\GSpec$ as a
localizing subcategory. This is the
assertion (or definition) that a $G$-spectrum $M$ is weakly contractible if and only if its
$H$-homotopy groups 
$\pi_*^H (M) := \pi_* \hom_{\GSpec}(\Sigma^\infty_+ G/H
, M)$
(for $H \leq G$ an arbitrary closed subgroup)
all vanish, and that these homotopy groups commute with arbitrary wedges. 
For simplicity, we will often write $G/H_+$ for $\Sigma^\infty_+ G/H$. 
\end{example} 

\begin{remark}
	For convenience we remark that the above $\infty$-categories are presentable and hence admit presentations by combinatorial model categories. In the case of $\GS$ and $\GSpt$ this follows from the well known fact that the transitive orbit spaces form a set of compact projective generators for these categories and \cite[Prop.\ 5.5.8.25]{Lur09}. The same argument applies to $\GSpec$, but one can more easily apply \cite[Cor.\ 1.4.4.2]{Lur14}.
\end{remark}

Next, we review the interaction between $\GSpec$ and the $\infty$-categories
of genuine equivariant spectra for subgroups. Let $G$ be a compact Lie group.
Let $H \leq G$ be a closed subgroup. 
There is a symmetric monoidal, colimit-preserving functor
 \begin{equation} \Res^G_H\colon \GSpec \to \HSpec \label{resspectra}
\end{equation} 
given by \emph{restriction}.
This arises from a symmetric monoidal, left adjoint functor of restriction on the category of
equivariant orthogonal spectra.

\newcommand{\HSpt}{\mathcal{S}_{H \ast}}

We will use the following properties of restriction and its adjoints. 
 	\begin{prop}\label{prop:ind-coind}
	Let $H \leq G$ be a closed subgroup.
 		The restriction functors \[\Res_H^G  \colon \GSpec\rightarrow  \HSpec\] admit left adjoints $\Ind_H^G$ and right adjoints $\Coind_H^G$. For a sequence of subgroup inclusions $K\leq H \leq G$ there are natural equivalences $\Ind_H^G\circ\Ind_K^H\cong \Ind_K^G$ and $\Coind_H^G\circ\Coind_K^H\cong \Coind_K^G$.

 		Moreover, for $X\in  \GSpec$, $Y\in  \HSpec$, there are natural equivalences:
 		\begin{align}
 			(\Coind_H^G Y)\wedge X & \simeq \Coind_H^G(Y\wedge \Res_H^G X)\label{eq:frob}\\
(\Ind_H^G Y)\wedge X & \simeq \Ind_H^G(Y\wedge \Res_H^G X)\label{eq:frob2}\\
 		 	\label{ind:sphere}		\Ind_H^G S^0_H & \simeq \Sigma^\infty_+ G/H.
 		\end{align}
		If $G$ is finite, then we have a natural equivalence
		\begin{equation} 
 			\Ind_H^G Y \simeq \Coind_H^G Y \label{eq:wirthmuller}.
		\end{equation} 
 	\end{prop}
\begin{proof}

At the level of the \emph{homotopy} category, these properties are classical. 
We briefly describe how to upgrade them to $\infty$-categorical equivalences,
although knowing this is not critical for the rest of the paper.
Since we are working with presentable \emph{stable} $\infty$-categories, 
the existence of a left or right adjoint to an exact functor can
be checked at the level of the \emph{homotopy category:} the condition is that
said functor should preserve arbitrary coproducts (resp.\ products)
\cite[Prop.\ 1.4.4.1]{Lur14}. So,
$\Ind_H^G, \CoInd_H^G$ exist at the $\infty$-categorical
level for purely
abstract reasons, once one knows about their existence at the  homotopy
category level (although one can write down strict models  for these as well;
see
\cite[\S 9.2]{equivariantsymm} for the finite group case).
The property \eqref{eq:frob} comes from a natural map (from left to right) at the level of
$\infty$-categories. Checking this map is an equivalence is done at the level of homotopy categories. The remaining claims are checked in the same way. 
% Similarly for \eqref{eq:frob2}. Moreover, \eqref{ind:sphere} is a statement at the level of the homotopy category.

Finally, \eqref{eq:wirthmuller} is  a special case of the \emph{Wirthm\"uller
isomorphism} (for compact Lie groups, see \cite[II.6]{LMSM86}).
Once again, 
the map arises from universal properties, as explained in \cite{wirthmuller}.
Namely, once we know that $\Ind_H^G(S^0_H) \simeq \CoInd_H^G(S^0_H)$, then 
we get  a natural map in $\GSpec$,
\[ S^0_G \to \CoInd_H^G(S^0_H) \simeq \Ind_H^G(S^0_H).  \]
Thus, for any $Z \in \GSpec$, we have natural transformations
\begin{equation} 
Z \simeq  Z \wedge S^0_G \to Z \wedge \Ind_H^G(S^0_H) \simeq
\Ind_H^G( \Res^G_H (Z)),
\end{equation} 
where we used the projection formula \eqref{eq:frob2} in the last step.
Taking $Z = \CoInd_H^G(Y)$ for $Y \in \HSpec$, we get a natural map
\begin{equation} 
\label{eq:wmap}
\CoInd_H^G( Y) \to \Ind_H^G( \Res_H^G ( \CoInd_H^G(Y))) \to \Ind_H^G(Y),
\end{equation} 
where the last map comes from the adjunction $(\Res_H^G, \CoInd_H^G)$.
The map \eqref{eq:wmap} is the natural transformation that implements the
Wirthm\"uller isomorphism.
This is explained in \cite{wirthmuller} at the level of homotopy categories,
but it makes sense at the $\infty$-categorical level. 
\end{proof}

\subsection{Restriction as base-change}
Let $G, H$ be finite groups. 
We will now present another point of view on the restriction functor $\res^G_H\colon
\GSpec \to \HSpec$, as a sort of base change. This point of view is due (albeit
in the setting of the homotopy category) to Balmer \cite{Balmerstack} and
Balmer-Dell'Ambrogio-Sanders \cite{BAS}.

We begin with some generalities. 
Let $(\mathcal{C}, \otimes, 1_{\mathcal{C}}), (\mathcal{D}, \otimes,
1_{\mathcal{D}})$ be presentable, symmetric monoidal
 stable $\infty$-categories where the tensor structures commute with
 colimits, and let $L\colon \mathcal{C} \to \mathcal{D}$ be a
 cocontinuous, symmetric monoidal functor.
 Then its right adjoint $R$ is
 naturally \emph{lax} symmetric monoidal. 
Since $L $ (resp.\ $R$) is symmetric monoidal (resp.\ lax symmetric monoidal), we
obtain induced functors
at the level of commutative algebra objects
\begin{equation} \label{adjalg}L\colon \CAlg(\mathcal{C}) \to \CAlg(\mathcal{D}), \quad R\colon \CAlg(\mathcal{D})
\to \CAlg(\mathcal{C}). 
\end{equation}

\begin{prop}
The induced pair of functors 
\eqref{adjalg}
form an adjoint pair. (This would work for any $\infty$-operad replacing the
commutative one.)
\end{prop} 
\begin{proof}
This is the analog of \cite[Prop.\ A.5.11]{GeH} for symmetric $\infty$-operads. Their proof applies essentially without change here.
\end{proof}
% \begin{proof} 
% We briefly outline how one may extract this result from the formalism of \cite[Ch.
% 2]{Lur14}, for which we refer for definitions. Note that it is possible to
% replace the $\e{\infty}$-operad by any $\infty$-operad. 
% One has $\infty$-categories $\mathcal{C}^{\otimes}, \mathcal{D}^{\otimes}$
% mapping via coCartesian fibrations
% $p\colon \mathcal{C}^{\otimes} \to 
% N( \mathrm{Fin}_*), \, q\colon \mathcal{D}^{\otimes} \to N( \mathrm{Fin}_*)$
% to the nerve $N( \mathrm{Fin}_*)$ of the category $\mathrm{Fin}_*$ of
% pointed finite sets (\cite[\S 2.1.1-2.1.2]{Lur14}). 

% The symmetric monoidal functor $L$ induces a functor
% $L^{\otimes}\colon \mathcal{C}^{\otimes} \to \mathcal{D}^{\otimes}$ over $N(\mathrm{Fin}_*)$
% which carries $p$-coCartesian morphisms into $q$-coCartesian morphisms.
% The right adjoint induces a functor $R^{ \otimes}\colon \mathcal{D}^{\otimes} \to
% \mathcal{C}^{\otimes}$ over $N( \mathrm{Fin}_*)$ which carries inert morphisms
% in $\mathcal{D}^{\otimes}$ to inert morphisms in $\mathcal{C}^{\otimes}$, and
% we obtain that $(L^{\otimes}, R^{\otimes})$ form an adjoint pair. Thus,
% we obtain an adjoint pair  between the $\infty$-categories of
% sections, $\mathrm{Fun}_p( N( \mathrm{Fin}_*),
% \mathcal{C}^{\otimes}) \rightleftarrows \mathrm{Fun}_q( N( \mathrm{Fin}_*),
% \mathcal{D}^{\otimes})$ from $L^{\otimes}, R^{ \otimes}$. Restricting this
% to the subcategories (\cite[Def. 2.1.3.1]{Lur14}) spanned by those sections of $p, q$ that correspond to
% commutative algebra objects gives the desired adjunction \eqref{adjalg}.
% \end{proof}

We will now derive a new adjunction from the above data.

\begin{cons}
\label{Derivedadj}
 From \eqref{adjalg}, we obtain a commutative algebra structure on $R( 1_{\mathcal{D}})
 \in \mathcal{C}$
 which, thanks to the lax symmetric monoidal structure on $R$, naturally acts
 on $R(Y)$ for any $Y \in \mathcal{D}$. 
Thus, we get a functor $\overline{R}\colon \mathcal{D} \to
 \md_{\mathcal{C}}(R(1_{\mathcal{D}}))$,
 fitting into a commuting square
 \[ \xymatrix{
  & \md_{\mathcal{C}}( R( 1_{\mathcal{D}}))\ar[d]  \\
 \mathcal{D} \ar[ru]^{\overline{R}} \ar[r]^{{R}} & \mathcal{C}.
 }\]
 
 The functor $\overline{R}$ is limit-preserving, and we see that the composite
 functor
 \[ \overline{L}\colon  \md_{\mathcal{C}} (R( {1}_{\mathcal{D}})) 
 \stackrel{L}{\to} \md_{\mathcal{D}}(L R( {1}_{\mathcal{D}}))  \xrightarrow{\otimes_{\left(L R({1}_{\mathcal{D}})\right)} 1_{\mathcal{D}}} \mathcal{D}  
 \]
 is left adjoint to $\overline{R}$ by inspection.
 The first functor here (written $L$ as well) takes a $R(1_{\mathcal{D}})$-module (in
 $\mathcal{C}$) and
 applies $L$ to obtain a $LR (1_{\mathcal{D}})$-module in $\mathcal{D}$. 
The second functor is base-change along the morphism of commutative algebra
objects $LR(1_{\mathcal{D}}) \to 1_{\mathcal{D}}$ in $\mathcal{D}$.
%JN: This should really appear when the claim is made.
 % One sees that $\overline{L}$ and $\overline{R}$ are indeed adjoints by writing down the right adjoint to the composite functor
 % $\overline{L}$ above. 
Note that this composite functor $\overline{L}$ is symmetric monoidal (as the
 composite of symmetric monoidal functors), and 
$\overline{R}$ therefore acquires a lax symmetric monoidal structure.

We thus get a new adjunction 
\begin{equation} \label{newadj} (\overline{L}, \overline{R})\colon
\md_{\mathcal{C}}(R(1_{\mathcal{D}}))
\rightleftarrows \mathcal{D.}\end{equation}
%Moreover, the unit map for the induced object is given by 
%the map of $G(1_{\mathcal{D}})$-modules
%\[ F^R( 1_{\mathcal{D}}) \otimes X \to  (X),  \]
%obtained by extending the counit map $X \to GF(X)$ by
%$G(1_{\mathcal{D}})$-linearity.

\end{cons}

\begin{example} 
 For instance, we see that $\overline{L}$ carries the ``induced''
 $R(1_{\mathcal{D}})$-module $R(1_{\mathcal{D}}) \otimes X$, for $X \in
 \mathcal{C}$, to 
\begin{equation} \label{lFforinduced}
\overline{L} \left( R(1_{\mathcal{D}}) \otimes X\right)  \simeq
1_{\mathcal{D}} \otimes_{L R(1_{\mathcal{D}})}L(  R(1_{\mathcal{D}}) \otimes X)
\simeq L(X) \in \mathcal{D}.\end{equation}

\end{example}

Our first goal in this subsection is to give a simple set of criteria for when
the adjunction $(\overline{L}, \overline{R})$ is an equivalence.
\begin{definition} 
We say that the adjunction $(L, R)$ \emph{satisfies the projection formula} if,
for $X \in \mathcal{C}, Y \in \mathcal{D}$, the natural map
\begin{equation} \label{projformmap}  R(Y) \otimes X \to R( Y \otimes
L(X)) ,  \end{equation}
adjoint to the map $$L ( R( Y) \otimes X) \simeq L R(Y) \otimes L(X)
\xrightarrow{{\mathrm{counit} \otimes 1_{L(X)}} } Y
\otimes L(X),$$
is an equivalence in $\mathcal{C}$.
\end{definition} 

\begin{prop} \label{overadjunctionequiv}
Suppose we have an adjunction $(L,R) \colon \mathcal{C} \rightleftarrows
\mathcal{D}$; here as above $\mathcal{C}, \mathcal{D}$ are presentable, symmetric
monoidal stable $\infty$-categories such that the tensor structure on each
commutes with colimits in each variable, and $L$ is a symmetric monoidal
functor. 
Suppose the adjunction has the following three properties:
\begin{enumerate}
\item The adjunction  $(L, R)$ satisfies the projection formula. 
\item The right adjoint $R$ commutes with arbitrary colimits. 
\item The right adjoint $R$ is conservative.
\end{enumerate}
Then the new adjunction $(\overline{L}, \overline{R}) \colon
\md_{\mathcal{C}}(R(1_{\mathcal{D}})) \rightleftarrows \mathcal{D}$ 
of \Cref{Derivedadj}
is an inverse
equivalence of symmetric monoidal $\infty$-categories.
\end{prop} 
\begin{proof} 
Consider the collection of objects $X \in \md_{\mathcal{C}}(R(1_{\mathcal{D}}))
$ such that $X \to \overline{R} \overline{L} 
(X)$ is an equivalence. We would like to show that this collection contains 
\emph{every} object of 
$\md_{\mathcal{C}}( R(1_{\mathcal{D}}))$. 
By hypothesis 2, this class of objects forms a localizing subcategory, so it suffices to show that these maps are
equivalences
for the generators $X = R(1_{\mathcal{D}}) \otimes X'$, $X' \in \mathcal{C}$. 
In this case, we have
by \eqref{lFforinduced}, a map
\[ R(1_{\mathcal{D}}) \otimes X' \to \overline{R} \overline{L}(X) \simeq
\overline{R}(  L(X')) 
\simeq R(1_{\mathcal{D}}) \otimes X', 
\]
using the projection formula. The composite map is an equivalence, and it follows that the unit map is always
an equivalence. 

It follows from this that the left adjoint $\overline{L}$ is fully faithful. 
%JN: Cut because this argument is explained in detail later in the paragraph.
% Since
% $\overline{R}$ is conservative (as $R$ is), it follows formally that $\overline{L},
% \overline{R}$ are inverse equivalences. 
In fact, $\overline{L}$ is necessarily
a \emph{colocalization}, and if $Y \in \mathcal{D}$, then the cofiber $C$ of the
map $\overline{L} \overline{R}(Y) \to Y$ has the property that the space of maps
$\hom_{\mathcal{D}}( \overline{L}X, C) $ is contractible for any $X \in 
\md_{\mathcal{C}}(R(1_{\mathcal{D}})) 
$. Therefore, $\overline{R}C$ is contractible and hence $RC$ is contractible.
By assumption $R$ is conservative, so $C$ is contractible. In particular, the counit maps of the
adjunction $(\overline{L}, \overline{R})$ are also equivalences.
\end{proof}

We now specialize to the case of interest, that of the adjunction $(\Res^G_H,
\CoInd_H^G)\colon \GSpec \rightleftarrows \HSpec$ for a \emph{finite} group $G$ and a
subgroup $H \leq G$.
Recall that $\Res^G_H$ is a symmetric monoidal functor, so that this does fit
into the preceding discussion. We begin by identifying the relevant commutative
algebra object $R(1_{\mathcal{D}})$.

\begin{cons}
Given any finite $G$-set $T$, the function spectrum $F(T_+, S^0_G)$
inherits a natural commutative algebra structure in $\GSpec$ since $T$ is tautologically
a \emph{commutative coalgebra} in $G$-spaces. This construction sends finite coproducts of
$G$-sets to \emph{products} in $\CAlg(\GSpec)$, and it carries a point to the
unit $S^0_G$. It is also compatible with restriction to subgroups.
Note that as $G$-spectra, $F(T_+, S^0_G) \simeq T_+$.

In case $T = G/H_+$, we would like to identify $F(G/H_+, S^0_G) \in \CAlg(\GSpec)$ with
$\CoInd_H^G(S^0_H)$ (which acquires a commutative algebra structure since
$\CoInd_H^G$ is lax symmetric monoidal). To do this, recall that giving a map of commutative
algebra objects $F(G/H_+, S^0_G) \to \CoInd_H^G(S^0_H)$ in $\CAlg(\GSpec)$ amounts to
giving a map $\Res^G_H( F(G/H_{+}, S^0_G)) \to S^0_H$ in $\CAlg(\HSpec)$. 
But this map can be obtained by using the fact that $G/H$ has a natural $H$-fixed
point, which gives (as in the previous paragraph) a decomposition of
$\Res^G_H F(G/H_+, S^0_G)$ as a product of $S^0_H$ and
another commutative algebra object. The map
$\Res^G_H( F(G/H_{+}, S^0_G)) \to S^0_H$ is the projection onto the $S^0_H$ piece. 
One now checks (at the level of underlying equivariant spectra) that the
adjoint map gives an equivalence
\begin{equation} \label{GHplus}  F(G/H_+, S^0_G) \simeq \CoInd_H^G(S^0_H) \in \CAlg(
\GSpec). \end{equation}
\end{cons}

The functor $\CoInd_H^G\colon \HSpec \to \GSpec$ is lax symmetric monoidal, 
and therefore there is a natural lax symmetric monoidal lifting
\[ \xymatrix{ 
 & \md_{ \GSpec}(F(G/H_+, S^0_G)) \ar[d]  \\
 \HSpec \ar[r]^{\CoInd_H^G} \ar[ru] &  \GSpec
}\]
where 
$\md_{ \GSpec}(F(G/H_+, S^0_G)) $ denotes the symmetric monoidal $\infty$-category
of modules in $\GSpec$ over $F(G/H_+, S^0_G) \simeq \CoInd_H^G(S^0_H)$.

\begin{thm}[Cf. Balmer-Dell'Ambrogio-Sanders \cite{BAS}]
\label{restensoring}
The functor $\HSpec \to \md_{\GSpec}(F(G/H_+, S^0_G))$ is an equivalence of symmetric monoidal
$\infty$-categories.
\end{thm} 
\begin{proof} 
This is a consequence of \Cref{overadjunctionequiv}:

Firstly, we already observed that our adjunction satisfies the projection formula in \eqref{eq:frob}.

Secondly, by the Wirthm\"uller isomorphism \eqref{eq:wirthmuller}, $\CoInd_H^G \simeq \Ind_H^G$ is both a
left and a right adjoint, so it preserves all limits and colimits. 

Finally, we need to see that $\CoInd_H^G$ is conservative.
Suppose $Y \in \HSpec$ is such that $\CoInd_H^G(Y)$ is contractible. It follows that $\Res_H^G\CoInd_H^G(Y)$ is contractible. Since $G$ is finite, this contains $Y$ as a retract, so $Y$ is contractible.
% JN: This was cut because it is a longer version of the argument which mostly appears in the previous construction.
% Then the spectrum of maps $(G/K)_+ \to\CoInd_H^G(Y)$ is contractible for any
% $K \leq G$. Take $K \leq H$. Then 
% \[ \hom_{\GSpec}( G/K_+ ,\CoInd_H^G(Y)) \simeq \hom_{\HSpec}( \Res^G_H (
% G/K_+) , Y)  \]
% is therefore contractible. However, the $H$-spectrum $\res^G_H(G/K_+)$ contains
% $H/K_+$ as a direct summand since $G/K \simeq H/K \sqcup (G/K  \setminus
% H/K)$ as $H$-sets. 
% Therefore, 
% $\hom_{\HSpec}( 
% H/K_+ , Y)$ is contractible
% for any $K \leq H$, which means that $Y$ is contractible. 
\end{proof}

\newtheorem*{warning}{Warning}
\begin{warning}
Just as in \cite[Lem.\ 3.3]{BAS} the above argument still goes through for an arbitrary  compact Lie group $G$ if we assume that $H$ is a closed finite index subgroup of $G$. However, for more general $H$ the functor $\CoInd_H^G(\cdot)$ fails to be conservative. 

For example, consider the $C_2$-spectrum $X$ constructed as the 
image of the idempotent $\frac{1  - \tau}{2}$ on $(C_2)_+ \otimes\mathbb{Q}$,
where $\tau \in C_2$ is the nontrivial element.  The $C_2$-fixed point spectrum
of $(KU_{C_2}\wedge X\otimes \bQ)$ is  contractible  while
$(KU_{C_2}\wedge X\otimes \bQ)$ is non-equivariantly noncontractible. If we embed $C_2$ into $U(1)$  we obtain a coinduced $KU_{U(1)}$-module spectrum $\CoInd_{C_2}^{U(1)} (KU_{C_2}\wedge X\otimes \bQ)$ whose $U(1)$-fixed points are contractible. However, it follows from \Cref{KUtorus} that the coinduced spectrum is contractible. Hence $\CoInd_{C_2}^{U(1)}(\cdot)$ is not conservative. 
\end{warning}

\Cref{restensoring} is both philosophically and practically important to us:
it lets us identify the main concern in this paper (descent with respect to
the commutative algebra objects $F(G/H_+, S^0_G) \in \clg(\GSpec)$) as a form of descent with respect
to restriction of subgroups. 
It will also enable us to recast some of our results in
terms of restriction to subgroups.

\section{Completeness, torsion, and descent in $\GSpec$ }\label{sec:thick}

Let $G$ be a \emph{finite} group, and consider the symmetric monoidal, stable
$\infty$-category $\GSpec$ of $G$-spectra. In this section, we will consider
the phenomena of completeness, torsion, and descent (formulated abstractly in
the earlier sections) in $\GSpec$ with respect to commutative algebra objects of the form
$\left\{F(G/H_+, S^0_G)\right\}$ as $H$ ranges over a family (\Cref{family}) of
subgroups of $G$.  We will see that this theory is closely related to the Lewis-May geometric
fixed point functors. 
Next, we treat the decomposition of the $\infty$-categories of $\sF$-complete
spectra as a homotopy limit over the $\sF$-orbit category.
We then make the primary definition (\Cref{Fnildef}) of this
paper by introducing the notion of \emph{nilpotence} with respect to a family
of subgroups. 

\subsection{Families of subgroups and $\sF$-spectra}

\label{ss:familiespaper1}
We now review some further relevant terminology from equivariant 
homotopy theory. Let $G$ be a finite group.
\begin{definition} \label{family}
A \emph{family of subgroups} of $G$ is a nonempty collection $\sF$ of subgroups
of $G$ such that if $H \in \sF$ and if $H' \leq G$ is subconjugate to $H$,
then $H' \in \sF$. 
Given a family $\sF$, we will let $A_{\sF} = \prod_{H \in \sF}
F(G/H_+, S^0_G) \in \CAlg(\GSpec)$.
\end{definition} 

Important examples of families (which will arise in practice) include the
families of $p$-subgroups, abelian subgroups, elementary abelian subgroups, etc.

\begin{definition} 
\label{Fspectra}
Fix a family of subgroups $\sF$. Then: 
\begin{enumerate}
\item  
A $G$-spectrum is \emph{$\sF$-torsion} (or an \emph{$\sF$-spectrum}) if it
belongs to the smallest
localizing subcategory\footnote{In this case, this is equivalent to the smallest localizing $\otimes$-ideal.} of
$\GSpec$ containing the $\left\{F(G/H_+, S^0_G)\simeq G/H_+ \right\}_{H \in
\sF}$. In other words, the $G$-spectrum is $A_{\sF}$-torsion.
\item
A $G$-spectrum is \emph{$\sF$-complete}
 if it is $A_{\sF}$-complete. 
\item 
A $G$-spectrum is \emph{$\sF^{-1}$-local} if it is $A_{\sF}^{-1}$-local. 
\item The \emph{$\sF$-completion, $\sF$-acyclization}, and
\emph{$\sF^{-1}$-localization} functors (on $\GSpec$) respectively are the $A_{\sF}$-completion,
$A_{\sF}$-acyclization, and $A_{\sF}^{-1}$-localization functors.
\end{enumerate}
\end{definition}

\Cref{Fspectra} is certainly not new. 
Before we get to our main goal (\Cref{Fnildef} below), we write down the localization, completion, and acyclization
functors explicitly (compare \cite[II.2, II.9]{LMSM86} and \cite[Def.
IV.6.1]{MaM02}). 
In doing so, the following construction will be useful. 

\begin{cons}
We associate to $\sF$ a $G$-space $E\sF$ and a pointed $G$-space $\wt{E}\sF$
which are determined up to weak equivalence by the following properties \cite[\S V.4]{May96}:
\begin{equation}
	E\sF^K  \simeq \begin{cases} * & \mbox{ if } K\in \sF\\
	 \emptyset & \mbox{otherwise}
  \end{cases}\label{eq:EF-univ-prop}, \quad \quad
 	\wt{E}\sF^K  \simeq \begin{cases} * & \mbox{ if } K\in \sF\\
	 S^0 & \mbox{otherwise.}  \end{cases}
\end{equation}

\begin{comment}
We recall that $E \sF$ 
can be constructed as follows. Let $X $ be the $G$-space $\bigsqcup_{H \in \sF}
G/H$. Then a model of $E \sF$ is given by the geometric realization
$|X^{\bullet + 1}|$ of the simplicial object 
$\dots \triplearrows X \times X \rightrightarrows X $ (obtained as the \v{C}ech
nerve of $X \to \ast$). 
When $\sF$ is the family consisting only of the trivial subgroup, then we write
$EG$ for $E \sF$: $EG$ is a non-equivariantly contractible $G$-space with a
free $G$-action. 
\end{comment}
\end{cons}

\begin{comment}
For two families $\sF$, $\sF^\prime$ of subgroups of $G$ these universal properties and the fact that taking fixed points commutes with smash products imply the following weak equivalence:
\begin{equation} 
	\wt{E\sF}\wedge \wt{E\sF^\prime}\simeq \wt{E\sF\cup \sF^\prime}.
\end{equation}
\end{comment}

The relevance of this definition to us is given in the following proposition.
\begin{prop} \label{EFlocalization}
The $\sF^{-1}$-localization $S^0_G[A_{\sF}^{-1}] \in \GSpec$ is given by the
suspension spectrum $\Sigma^\infty\widetilde{E}\sF $.
The $\sF$-acyclization $\AC_{A_{\sF}}(S^0_G)$ is given by $\Sigma^\infty_+ E \sF$.
\end{prop} 
\begin{proof} 
It suffices to prove the second claim because there is a fiber sequence
of $G$-spectra  $\Sigma^\infty_+E\sF\longrightarrow S^0_G\longrightarrow \Sigma^\infty\widetilde{E}\sF$.
The (unpointed) $G$-space $E \sF$ admits a cell
decomposition with cells of the form $G/H$ for $H \in \sF$, so $\Sigma^\infty_+
E \sF$ is  $A_{\sF}$-torsion.
Now, it suffices to show that the map $\Sigma^\infty_+ E \sF \to S^0_G$ becomes
an equivalence after tensoring with $A_{\sF}$. Equivalently, in view of
\Cref{restensoring}, it should become
an equivalence after \emph{restricting} to any subgroup $H \in \sF$. However, $E
\sF$ is equivariantly contractible after such a restriction, so the claim
follows.
\end{proof} 

Note also that the classical simplicial model  of $E \sF$ as the geometric
realization $|X^{\bullet + 1}|$ for $X =
\bigsqcup_{H \in \sF} G/H$ reproduces the simplicial model of the $A_{\sF}$-acyclization as given in
\Cref{U_iseq}.

 Similarly, one sees easily that:
\begin{prop}\label{prop:completion}
The $\sF$-completion of a $G$-spectrum $X$ is given by the internal mapping
spectrum $F(E\sF_+, X)$.
\end{prop} 

By \Cref{ourdwg}, we have an equivalence of $\infty$-categories between
$\sF$-torsion spectra and $\sF$-complete spectra (given by $\sF$-completion). 
\subsection{Geometric fixed points}

The purpose of this subsection is to review the relationship between the
Lewis-May \emph{geometric fixed points} functor and the theory of
$A^{-1}$-localization. 
Recall first that $\GSpec$ is a presentable, symmetric monoidal, stable
$\infty$-category. As such, it receives a canonical symmetric monoidal,
colimit-preserving functor
\[ i_* \colon \Sp \to \GSpec.  \]

\begin{definition} 
The lax symmetric monoidal functor $\GSpec \to \Sp$ right adjoint to $i_*$ is called the functor of
\emph{categorical fixed points,} and will be denoted by $X \mapsto X^G$ or $i_G^* X$.
More generally, given a subgroup $H \leq G$, 
the composition $\GSpec \xrightarrow{\Res^G_H} \HSpec \xrightarrow{(\cdot)^H}\Sp$ is called the
functor of \emph{categorical $H$-fixed points,} and is denoted $X \mapsto X^H$ or $i_H^* X$.
\end{definition} 

\begin{remark} 
The same notation is used when $G$ is a compact Lie group (so that $H$ is now
required to be a closed
subgroup). 
\end{remark} 

The fixed point functors $i_H^*$, for $H \leq G$, are corepresented by $G/H_+$;  for $H
= G$ this follows because $S^0_G$ is the unit, and in general it follows by the
adjunction between induction and restriction. 
Since all $\infty$-categories in question are compactly generated, and $i_*$ preserves
compact objects (the sphere is compact in $\HSpec$ too), it follows that
the functor of categorical fixed points (for any subgroup $H \leq G$)
preserves colimits. As right adjoints, the functors $\{i^*_H\}_{H \leq G}$ of course preserve
limits. We note also the relation
\begin{equation} \label{coindfixed} ( \Coind_H^G X)^G \simeq X^H, \quad X \in
\HSpec,  \end{equation}
which follows easily by universal properties.

The functor $i_G^* = (\cdot)^G$ is only lax symmetric monoidal, and has a
nontrivial value on $G_+$ (in fact, by \eqref{coindfixed}, one sees
easily $(G_{+})^G
\simeq S^0$). To obtain a fixed point functor with the expected geometric properties, we first force the non-trivial orbits to be contractible via localization before taking categorical fixed points.
% To fix this, it is possible to localize to annihilate $G/H_+$ for $H \neq G$.

\begin{cons} 
Let $\mathcal{P}_G$ denote the family of proper subgroups of $G$.
Consider as before $A_{\mathcal{P}_G} = \prod_{H \lneq G} F(G/H_+, S^0_G) \in \CAlg( \GSpec)$ and form the
$A_{\mathcal{P}_G}^{-1}$-localization $\GSpec[{\mathcal{P}_G}^{-1}]$, which receives a symmetric monoidal functor
$\GSpec \to \GSpec[{\mathcal{P}_G}^{-1}]$ that annihilates the $G/H_+$ for $H\lneq G$. 

\end{cons} 

An important result in the theory \cite[Cor. II.9.6]{LMSM86} is that the localization 
$\GSpec[{\mathcal{P}_G}^{-1}]$ recovers the ordinary $\infty$-category of non-equivariant
spectra. 
\begin{thm}[] 
\label{geometricfixedptsthm}
The composite functor $\sp \xrightarrow{i_*} \GSpec \xrightarrow{(\cdot)[A_{\mathcal{P}_G}^{-1}]} \GSpec[{\mathcal{P}_G}^{-1}]$ is an equivalence of
symmetric monoidal $\infty$-categories, with inverse given by the fully
faithful embedding $\GSpec[{\mathcal{P}_G}^{-1}]\subset \GSpec$ followed by categorical
fixed points.
\end{thm} 

\begin{proof} 
Observe that, as a localization of $\GSpec$, the $\infty$-category
$\GSpec[{\mathcal{P}_G}^{-1}]$ is generated as a localizing subcategory by
the images of $\left\{G/H_+\right\}_{H \leq G}$. By definition, however, we have
forced the non-trivial orbits to be contractible.
In other words, $\GSpec[{\mathcal{P}_G}^{-1}]$ is generated as a localizing
subcategory by
the unit $\Sigma^\infty\widetilde{E}\mathcal{P}_G$, which is the
$A_{\mathcal{P}_G}^{-1}$-localization of the equivariant sphere by \Cref{EFlocalization}.
In particular, $\GSpec[{\mathcal{P}_G}^{-1}]$ is a symmetric monoidal, stable
$\infty$-category where the unit object is a compact generator. 

By the symmetric monoidal version of the Schwede-Shipley theorem
\cite[Prop.~7.1.2.7]{Lur14}, 
it suffices to show that the categorical fixed points of $\Sigma^\infty
\widetilde{E}
\mathcal{P}_G$  are $S^0$. Here we use that
$\Sigma^\infty\widetilde{E}\mathcal{P}_G$ is the suspension spectrum of a
space, so $i^*_G \widetilde{E}\mathcal{P}_G$ is connective and we have
equivalences of spaces
\[ \Omega^\infty i^*_G \widetilde{E}\mathcal{P}_G = \hom_{\GSpec}(
S^0_G,\Sigma^\infty\widetilde{E}\mathcal{P}_G)
\simeq \varinjlim_V\hom_{\GSpt}(S^V, \Sigma^V\widetilde{E}\mathcal{P}_G
),\]
where $V$ ranges over finite-dimensional orthogonal representations of $G$ and $S^V$ denotes the
one-point compactification of $V$. 
However, for any $V$, with fixed vectors $V_0  = V^G$, we have homotopy equivalences
of mapping spaces
\[  \hom_{\GSpt}(S^V, \Sigma^V\widetilde{E}\mathcal{P}_G)
\simeq 
\hom_{\GSpt}(S^{V_0}, \Sigma^V\widetilde{E}\mathcal{P}_G)
\simeq
\hom_{\GSpt}(S^{V_0}, \Sigma^{V_0}\widetilde{E}\mathcal{P}_G)
= \hom_{\mathcal{S}_{\ast}}(S^{V_0}, S^{V_0})
, \]
 as the pointed $G$-space $\widetilde{E}\mathcal{P}_G$ has
 contractible $H$-fixed points 
 for $H \neq G$. In particular, we find that 
 the natural map of spectra $S^0 \to i^*_G
 \widetilde{E} \mathcal{P}_G$ is an equivalence.
\end{proof} 

\begin{definition} 
 The composition $$\GSpec \xrightarrow{(\cdot)[A_{\mathcal{P}_G^{-1}}]}\GSpec[\mathcal{P}_G^{-1}] 
\subseteq \Sp_G\xrightarrow{(\cdot)^G}\Sp,$$ is called the 
\emph{geometric fixed points functor} and is denoted $\Phi^G$. 

By construction,
$\Phi^G$ is a symmetric monoidal, colimit-preserving functor $\Phi^G \colon
\GSpec \to \Sp$. 
More generally, for $H \leq G$, we define a symmetric monoidal functor $\GSpec \xrightarrow{\Res^G_H}
\HSpec \xrightarrow{\Phi^H} \sp$, which we will write as $\Phi^H$ and call the 
\emph{geometric $H$-fixed points functor.}
\end{definition}

We thus recover the following classical result. 
\begin{prop} 
A $G$-spectrum $M$ is contractible if and only if $\Phi^H M \in \sp$ is
contractible for every subgroup $H \leq G$.
\end{prop}
\begin{proof} 
Suppose $\Phi^H M $ is contractible for each $H \leq G$. By induction on
$G$, we may assume that $\res^G_H M$ is contractible for every $H\lneq G$. 
If we let $A_{\mathcal{P}_G} = \prod_{H \lneq G} F(G/H_+, S^0_G)$ as before,
then we have that $M \wedge A_{\mathcal{P}_G} \in \GSpec$
is contractible, using \Cref{restensoring}. Our assumption $\Phi^G M\cong *$ implies 
via \Cref{geometricfixedptsthm} that $M[A_{\mathcal{P}_G}^{-1}]$ is
contractible, so that $M$ must be contractible too (\Cref{checkequiv}). 
\end{proof} 

\subsection{Example: Borel-equivariant spectra} \label{sec:Borel}

Our next goal is to identify the $\infty$-categories of $\sF$-\emph{complete}
objects as a certain limit.
We begin with the most basic case, when the family consists only of the
trivial subgroup. 

\begin{definition} 
A $G$-spectrum $M \in \GSpec$ is said to be \emph{Borel-equivariant} (or
\emph{Borel-complete} or \emph{cofree}) if it is complete with
respect to the trivial family $\sF = \left\{\left\{1\right\}\right\}$, i.e., complete for the algebra object
$F(G_+, S^0_G)
\in \GSpec$. Equivalently, by \Cref{prop:completion}, $M$ is Borel-equivariant if and only if the natural
map
\[ M \to F(EG_+, M)  \]
is an equivalence in $\GSpec$. The Borel-equivariant $G$-spectra span a full
subcategory $\Gbs \subset \GSpec$.
\end{definition}

\begin{prop}\label{prop:characterizeborel}
A $G$-spectrum $M \in \GSpec$ is Borel-equivariant if and only if for every $X \in \GSpec$
which is nonequivariantly contractible,\footnote{That is, such that
$\res^G_{\left\{1\right\}}
X \in \Sp$ is contractible.} we have $\hom_{\GSpec}(X, M) \simeq
\ast$.
\end{prop} 
\begin{proof} 
Indeed, for $X \in \GSpec$ being nonequivariantly contractible is equivalent to
the condition that $G_+ \wedge X \in \GSpec$ is contractible
in view of \Cref{restensoring}. As a result, the condition that 
$\hom_{\GSpec}(X, M) \simeq
\ast$
for every nonequivariantly contractible $X$ 
is \emph{precisely}
the condition of $A_{\left\{1\right\}}=F(G_+, S^0_G)$-completeness. 
\end{proof} 

\begin{prop} 
Suppose $M \in \GSpec$ is Borel-equivariant. Then for all $H \leq G$,
$\res^G_H M \in \HSpec$ is Borel-equivariant.
\end{prop} 
\begin{proof} 

Using \Cref{prop:characterizeborel} we need to show that for every $X\in\HSpec$ with 
$\Res^H_{\{1\}} X\cong *$, the mapping space $\hom_{\HSpec}(X,\Res^G_H M)$ is contractible.

This mapping space is always equivalent to $\hom_{\GSpec}(\Ind_H^G X,M)$ and since 
$\Res^G_{\{1\}}\Ind_H^G X$ is a wedge of copies of $\Res^H_{\{1\}} X$, and hence contractible, 
the lastly displayed mapping space is indeed contractible, using the assumption that $M$
is Borel-equivariant and \Cref{prop:characterizeborel} again.
\end{proof} 

The main result is that the $\infty$-category of Borel-equivariant spectra can be described as the $\infty$-category of spectra with a $G$-action.
\begin{prop} 
\label{Borelequiv}
We have a canonical equivalence of symmetric monoidal $\infty$-categories $\Gbs \simeq \fun(BG, \sp)$.
\end{prop} 
\begin{proof} 
By \Cref{Acplimit} for $\mathcal{C}=\GSpec$ and $A=F(G_+,S^0_G)$, we know that we have an equivalence
\[ \Gbs \simeq \mathrm{Tot} \left( \md_{\GSpec}(F(G_+, S^0_G)) 
\rightrightarrows \md_{\GSpec}(F( (G \times G)_+, S^0_G)) \triplearrows \dots \right). 
\]
However, we know by \Cref{restensoring} that $\md_{\GSpec}(F(G_+, S^0_G)) \simeq \sp$. Similarly, we
obtain that $\md_{\GSpec}(F( (G^n)_+ , S^0_G))$ can be identified with 
$\md_{\sp}( F(G_+^{n-1}, S^0)) \simeq \prod_{G^{n-1}} \sp$. Unwinding the definitions,
we find that $\Gbs$ is identified with a totalization
\[ \Gbs \simeq \mathrm{Tot}\left( \sp \rightrightarrows \prod_G \sp
\triplearrows \dots \right),  \]
which recovers precisely the functor category $\fun(BG, \sp)$ for the standard simplicial
decomposition of $BG$. 
\end{proof} 
Stated more informally, a Borel-equivariant spectrum is determined by its
restriction to $\sp$ (i.e., its underlying spectrum) together with the
induced $G$-action on it.

\begin{remark}
Another way to think of
this result, in view of the equivalence (\Cref{ourdwg}) between torsion and
complete objects, is to observe that $\Gbs$ has a compact generator given by
$G_+$ itself. The endomorphism algebra of $G_+$ (in view of the universal property
of induction) is given by the group algebra of
$G$, $\Sigma^\infty_+ G \in \mathrm{Alg}(\sp)$. Similarly, the induced object
$G_+$ in $\fun(BG, \sp)$ has endomorphisms given by $\Sigma^\infty_+ G$. As a
result, both $\infty$-categories are identified with $\md_{\sp}(
\Sigma^\infty_+ G)$ by Lurie's version of the Schwede-Shipley theorem
\cite[Th. 7.1.2.1]{Lur14}.
\end{remark}

The symmetric monoidal equivalence of \Cref{Borelequiv} shows also that, for a
Borel-equivariant $G$-spectrum, the categorical fixed points can be identified
with the \emph{homotopy fixed points} of the underlying
object of $\fun(BG, \Sp)$. 
In fact, if $X$ is any $G$-spectrum, then $X$ defines an underlying spectrum
$X_u = \Res^G_{\left\{e\right\}} X$ with a $G$-action,
and 
we have a natural map
\[ X^G \to X_u^{hG} \simeq F(EG_+, X)^G.  \]

\begin{prop} 
Suppose $X$ is a $G$-spectrum with underlying spectrum with $G$-action $X_u \in  \fun(BG,
\sp)$. Then the following are equivalent: 
\begin{enumerate}
\item  
$X$ is Borel-equivariant.
\item For each subgroup $H \leq G$, the map 
$X^H \to X_u^{hH}$ is an equivalence of spectra.
\end{enumerate}
\end{prop} 
In particular, the notion of ``Borel-equivariance'' can be useful for
formulating descent questions. 
\begin{proof} 
This follows from the fact that $X$ is Borel-equivariant if and only if the
Borel-completion map
$X \to F(EG_+, X)$ is an equivalence of $G$-spectra (i.e., induces an
equivalence on $H$-fixed points for each $H \leq G$), and the $H$-fixed points of $F(EG_+, X)$
are given by $X_u^{hH}$ for $H \leq G$.
\end{proof} 

\begin{example}
Given a spectrum $M$ and a finite group $G$, we define the
\emph{Borel-equivariant $G$-spectrum} $\bor{M} \in
\GSpec$  to be 
$F(EG_+, i_* M)$.
By construction, $\bor{M}$ is the genuine $G$-spectrum that 
represents Borel-equivariant $M$-cohomology on $G$-spaces as one sees by
calculating maps in the $\infty$-category $\fun(BG, \sp)$. 
Under the correspondence of \Cref{Borelequiv}, $\bor{M}$ corresponds to the
spectrum $M$ with trivial $G$-action.
\end{example}

As another consequence, we note also that the theory of \emph{modules} over the
Borel-equivariant form of a non-equivariant ring spectrum $R$ is closely related to $\infty$-categories of
the form $\fun(BG, \md(R))$ where $\md(R)$ is the $\infty$-category of \emph{left} $R$-modules. This result connects the analysis of
``representation'' $\infty$-categories such as $\fun(BG, \md(R))$ to the
genuinely equivariant analysis we are carrying out here. 

\begin{cor} 
\label{perfborelequiv}
Let $R\in\mathrm{Alg}(\Sp)$ be an $\e{1}$-algebra. Then the functor
\[ \md_{\GSpec}( \bor{R}) \to \fun(BG, \md(R)),  \]
is fully faithful when restricted to the compact objects. 
\end{cor} 
\begin{proof} 
This follows because the compact objects in $\md_{\GSpec}(\bor{R})$ (which
form the thick subcategory generated by $\bor{R} \wedge G/H_+$ for $H
\leq G$) are
automatically Borel-complete themselves. 
\end{proof}

Borel-equivariant spectra will yield most of the examples that we apply the
$\sF$-nilpotence theory to in this paper and the next. As a result, we now
describe several important cases. 
Many deep theorems in algebraic topology state that specific equivariant
spectra are, in fact,  Borel-complete. 
Let $(\mathcal{C}, \otimes, \mathbf{1})$ be a presentable symmetric monoidal, stable
$\infty$-category whose tensor product preserves colimits in each variable. Given a finitely generated
ideal $I = (x_1, \dots, x_n) \subset \pi_0 \hom_{\mathcal{C}}(\mathbf{1}, \mathbf{1})$, we can form
the \emph{$I$-adic completion} of an object $X$ in $\mathcal{C}$ 
following techniques of \cite{DAGXII}: it is the limit of the
cofibers $X/(x_1^k, \dots, x_n^k)$ or, alternatively, the Bousfield localization of $X$ with respect to $\bigotimes_{i = 1}^n\mathbf{1}/x_i$.
For example, given an $\e{\infty}$-algebra $R$ in $\GSpec$ and a finitely
generated ideal $I \subset \pi_0 R^G$, we can form the $I$-adic completion of
$R$. 

\begin{example} 
The Atiyah-Segal completion theorem \cite{AtS69} states that the
Borel-completion of equivariant $K$-theory $KU_G \in \GSpec$ is equivalent to the
completion of $KU_G $ at the augmentation ideal $I \subset \pi_0  i_G^*KU_G
\simeq R(G)$ (the complex representation ring).
\end{example} 
\begin{example} 
The Segal conjecture, proved by Carlsson in \cite{Car84}, states that the
Borel-completion of the sphere spectrum $S^0 \in \GSpec$  is the completion at the
augmentation ideal in $\pi_0 i_G^* (S^0_G)$, which is the Burnside ring. 
\end{example} 

This point of view on completion theorems has been articulated, for
instance, by Greenlees-May in \cite{GrM92b}. 

\subsection{Example: Genuine $C_p$-spectra}
In this subsection (which may be skipped without loss of continuity), we digress to give a
decomposition of the $\infty$-category of $C_p$-spectra using the
material of the previous subsection together with the theory of fracture
squares. 
This decomposition is a well-known folklore result, but we have included it for
expository purposes.

First, 
let $G$ be  a finite group and let $\sF$ be a family of subgroups.
Given any $X \in \GSpec$, we have an arithmetic square
\[ \xymatrix{
X \ar[d] \ar[r] & \widehat{X}_{\sF} \ar[d] \\
X[\sF^{-1}]  \ar[r] &  X_{t\sF} := \left(\widehat{X}_{\sF}\right)[\sF^{-1}],
}
\]
which allows us to recover $X$ from its $\sF$-completion $\widehat{X}_{\sF}:=F(E\sF_+,X)$, 
its $\sF^{-1}$-localization $X[\sF^{-1}]:=X[A_{\sF}^{-1}]$  , and its \emph{$\sF$-Tate construction $X_{t\sF}:= \left(\widehat{X}_{\sF}\right)[\sF^{-1}]$}.
Using \Cref{Afracsquare}, we can obtain a decomposition of the
$\infty$-category $\GSpec$.

Suppose now $G = C_p$ for some prime $p$ and $\sF = \left\{ \left\{ 1 \right\}\right\}$.
In this case, we have two simplifications. First, we know that the $\sF^{-1}$-local 
objects are given by the $\infty$-category of spectra
(\Cref{geometricfixedptsthm}) and that the $\sF$-complete objects are given by
$\fun( BC_p, \sp)$ (by \Cref{Borelequiv}). 
As a result, we deduce: 

\begin{thm}
We have an equivalence of $\infty$-categories:  
\[ \mathrm{Sp}_{C_p} \simeq \fun( \Delta^1, \sp) \times_{\sp}
\fun(BC_p, \sp),  \]
where: 
\begin{enumerate}
\item The functor $\fun( \Delta^1, \sp) \to \sp$ is evaluation at the terminal
vertex $1$. 
\item The functor $\fun( BC_p, \sp) \to \sp$ is the Tate construction. 
\end{enumerate}
\end{thm} 

Stated informally: to give a $C_p$-spectrum is equivalent to giving an object $X \in
\mathrm{Fun}(BC_p, \sp)$, an object $Y \in \sp$, and a map $Y \to
X_{tC_p}$.

\subsection{Decomposition of the $\infty$-category of $\sF$-complete spectra}

Let $G$ be a finite group and let $\sF$ be a family of subgroups. 
We denote by $\mathcal{O}(G)$ the orbit category of $G$, i.e., the category of
finite transitive $G$-sets.
The purpose of this subsection is to prove a generalization of \Cref{Borelequiv}:
We identify the $\infty$-category of $\sF$-complete
objects in $\GSpec$ with a (homotopy) limit over a subcategory of the orbit
category. This gives a generalization of \Cref{Borelequiv} which will, however,
require additional effort to set up. 

First, observe that we have a functor
\[ \mathcal{O}(G)^{op} \to \clg(\sp_G), \quad G/H  \mapsto F( G/H_+, S^0_G) .  \]

\begin{definition} 
We let $\mathcal{O}_{\sF}(G) \subseteq \mathcal{O}(G)$ denote the full subcategory
spanned by the $G$-sets with isotropy in $\sF$, i.e., the $G$-sets
$\left\{G/H\right\}_{H \in \sF}$.
\end{definition}

Let $\cato$ be the $\infty$-category of symmetric monoidal $\infty$-categories
and symmetric monoidal functors. 
We now obtain a functor
\[ \mathcal{O}(G)^{op} \to \cato, \quad G/H \mapsto \md_{\GSpec}( F(G/H_+ , 
S^0_G)) \simeq \HSpec,
\]
where the last equivalence comes from \Cref{restensoring}. 
Note that $\mathcal{O}(G)^{op}$ has an initial object $G/G = \ast$, which is
mapped by the above functor to $\GSpec$. As a result, for any subcategory
$\mathcal{I} \subset \mathcal{O}(G)^{op}$, we obtain a symmetric monoidal functor
\begin{equation} \label{symmmonoidalfunctor} \GSpec \to \varprojlim_{G/H \in
\mathcal{I}} \HSpec.  \end{equation}
\newcommand{\fcompl}{{\sF-\mathrm{compl}}}

We can now state our main result, which gives a decomposition of the
$\infty$-category $(\GSpec)_{\fcompl}$ of $\sF$-complete spectra (generalizing
\Cref{Borelequiv}). 
\begin{thm} 
\label{Fcompldecomp}
The above functor \eqref{symmmonoidalfunctor}
with $\mathcal{I} = \mathcal{O}_{\sF}(G)^{op}$ factors through the
$\sF$-completion of $\GSpec$ and 
gives an equivalence of symmetric monoidal $\infty$-categories
\[ (\GSpec)_{\fcompl} \simeq \varprojlim_{G/H \in \mathcal{O}_{\sF}(G)^{op}}
\HSpec.  \]
\end{thm}

\Cref{Acplimit} already gives a decomposition of the $\infty$-category of
$\sF$-complete (i.e., $A_{\sF}$-complete) $G$-spectra; however, it is 
 indexed over $\Delta$. In order to deduce \Cref{Fcompldecomp}, we shall need
 some general preliminaries on cofinality.\footnote{We will use the convention,
 following \cite{Lur09}, that \emph{cofinality} of a functor refers to the invariance of
 \emph{co}limits.}
These arguments are not new and appear, for instance, in the proof of
\cite[Prop.~5.7]{DAGVII}, which is closely related. 
Note also that this recovers \Cref{Borelequiv} as the special case $\sF =
\left\{ \left\{ 1 \right\}\right\}$. 

Let $\mathcal{C}$ be an $\infty$-category and $X \in \mathcal{C}$. Suppose that
the products $X^n, n \geq 1$ exist in $\mathcal{C}$. In this case, we can form
a simplicial object $X^{\bullet + 1} \in \fun(\Delta^{op}, \mathcal{C})$, 
\[ X^{\bullet + 1}  = \left( \dots X \times X \times X \triplearrows X \times
X \rightrightarrows X \right).  \]
To construct this object, we adjoin a terminal object $\ast$ to $\mathcal{C}$; in this
case the above simplicial 
object is the \v{C}ech nerve of $X \to \ast$.

\begin{prop} 
\label{cofinal}
Let $\mathcal{C}$ be an $\infty$-category and let $X \in \mathcal{C}$ be an
object such that the products $X^n$ exist for $n \geq 1$. Suppose that every object $Y \in \mathcal{C}$ admits a map $Y \to X$. 
Then the functor $X^{\bullet + 1} : \Delta^{op} \to \mathcal{C}$ is
cofinal.  
\end{prop} 
\begin{proof} 
Let $F\colon \mathcal{A} \to \mathcal{B}$ be a functor of $\infty$-categories.
The $\infty$-categorical version of Quillen's Theorem A  \cite[Thm.\ 4.1.3.1]{Lur09} (due to Joyal) states that  $F$ is cofinal if and only if the fiber product $\mathcal{A}
\times_{\mathcal{B}} \mathcal{B}_{B/}$ is contractible for each $B \in
\mathcal{B}$. 
Recall that the left fibration $\mathcal{B}_{B/} \to \mathcal{B}$
classifies the corepresentable functor $f_B = \hom_{\mathcal{B}}(B, \cdot)
\colon \mathcal{B} \to \mathcal{S}$, so the
colimit of $f_B \circ F $ is given by the homotopy type of $\mathcal{A}
\times_{\mathcal{B}} \mathcal{B}_{B/}$ 
in view of the computability of colimits
 in $\mathcal{S}$ via the Grothendieck construction \cite[Cor.~3.3.4.6]{Lur09}. It follows
from this that $F$ is cofinal if and only if for every corepresentable functor
$f\colon \mathcal{B} \to \mathcal{S}$, the colimit of $f \circ F\colon \mathcal{A} \to
\mathcal{S}$ is contractible. 

Now, let $Y \in \mathcal{C}$ be arbitrary. In order to prove cofinality, we need to show that the geometric
realization $\left| \hom_{\mathcal{C}}(Y, X^{\bullet + 1})\right|$ is weakly
contractible. However, this geometric realization can be identified with 
the geometric realization
$\left| \hom_{\mathcal{C}}(Y, X)^{\bullet + 1}\right|$, which is contractible
as $\hom_{\mathcal{C}}(Y, X)$ is nonempty by assumption. 
\end{proof}

We now review some further $\infty$-categorical preliminaries on colimits.
Compare \cite[Remark 5.3.5.9]{Lur09}.  
\begin{cons}
If $\mathcal{C}'$ is an $\infty$-category, then there exists an $\infty$-category
$\mathcal{C}$ containing $\mathcal{C}'$ as a full subcategory such that
$\mathcal{C}$ admits finite coproducts and is initial with respect to this
property. 
For an $\infty$-category $\mathcal{D}$ with finite coproducts, one has an
equivalence
\[ \fun( \mathcal{C}', \mathcal{D}) \simeq \fun_{\sqcup}( \mathcal{C},
\mathcal{D}),  \]
where $\fun_{\sqcup}(\mathcal{C}, \mathcal{D})$ denotes the subcategory spanned
by those functors preserving finite coproducts.
This equivalence is given by left Kan extension.

As a result, the objects of $\mathcal{C}$ are given by formal finite coproducts
of objects in $\mathcal{C}'$. 
The $\infty$-category $\mathcal{C}$ can be
explicitly constructed as the smallest subcategory of the presheaf
$\infty$-category $\mathcal{P}(\mathcal{C}')$ containing $\mathcal{C}'$ and
closed under finite coproducts.
\label{fincoprod}
\end{cons}

\begin{lemma} 
\label{prodlimitlemma}
Suppose $\mathcal{C}$ is an $\infty$-category with finite coproducts and $f\colon
\mathcal{C} \to \mathcal{D}$ is a functor, where $\mathcal{D}$ has all colimits. 
Suppose $f$ preserves finite coproducts. 
Let $\mathcal{C}' \subset \mathcal{C}$ be a full subcategory such that
$\mathcal{C}$ is obtained by freely adjoining finite coproducts 
to $\mathcal{C}'$ as in \Cref{fincoprod}. Then the natural map in $\mathcal{D}$
\[ \varinjlim_{\mathcal{C}'} f \to \varinjlim_{\mathcal{C}} f   \]
is an equivalence.
\end{lemma} 
\begin{proof} 
It follows that $f$ is the left Kan extension of $f|_{\mathcal{C}'}$, which
forces the map on colimits to be an equivalence. 
\end{proof} 

\begin{proof}[Proof of \Cref{Fcompldecomp}]
We take $\mathcal{C}$ to be the category of all finite $G$-sets all of whose
isotropy groups lie in $\sF$, so that $\mathcal{C}$ is obtained by freely
adjoining finite coproducts to $\mathcal{O}_{\sF}(G)$. 
We now consider the functor
\[ M\colon \mathcal{C}^{op} \to \cato, \quad T \mapsto \md_{\GSpec}(F(T_+,
S^0_G)).  \]
This functor sends finite coproducts in $\mathcal{C}$ to products in $\cato$. 
Let $U \in \mathcal{C}$ be the $G$-set $\bigsqcup_{H \in \sF} G/H$.
Observe that any $G$-set in $\mathcal{C}$ admits   a map to $U$.
We have a functor
\(  \Delta^{op} \to \mathcal{C} \) given by the simplicial object $\dots
\triplearrows U
\times U \rightrightarrows  U $, which is 
is cofinal in view of
\Cref{cofinal}.
Therefore, dualizing the cofinality statement, we find that
\[  ( \GSpec)_{\fcompl} \simeq \varprojlim_{\Delta} M \circ f \simeq
\varprojlim_{\mathcal{C}} M, \]
where the first equivalence is \Cref{Acplimit} (in fact, the cosimplicial
diagram $M \circ f$ is precisely the cobar construction) and the second
equivalence follows by cofinality. 
Finally, we use \Cref{prodlimitlemma} to identify $\varprojlim_{\mathcal{C^{\mathit{op}}}} M
$ with $\varprojlim_{\mathcal{O}_{\sF}(G)^{op}} M|_{\mathcal{O}_{\sF}(G)^{op}}$.
\end{proof}

It will also be convenient to have a slight refinement of 
\Cref{Fcompldecomp} based on a further cofinality argument.
For this, we consider a collection $\mathcal{A}$ of subgroups in $\sF$ such that
every subgroup in $\sF$ is contained in  an element of 
$\mathcal{A}$. 
We assume that $\mathcal{A}$ is closed under 
conjugation and intersections. 
As before, we let $\mathcal{O}_{\mathcal{A}}(G) $ be the subcategory of the orbit category of $G$
spanned by the $G$-sets $\left\{G/H\right\}_{H \in \mathcal{A}}$. 
We have an inclusion
\[ \mathcal{O}_{\mathcal{A}}(G) \subset \mathcal{O}_{\sF}(G).  \]

\begin{prop} 
\label{furthercofinalcollection}
Let $\mathcal{A}$ be a collection of subgroups in $\sF$ such that
every subgroup in $\sF$ is contained in  an element of 
$\mathcal{A}$. 
We assume that $\mathcal{A}$ is closed under 
conjugation and intersections. Then the inclusion $\mathcal{O}_{\mathcal{A}}(G) \to
\mathcal{O}_{\sF}(G)$ is cofinal.
\end{prop} 
\begin{proof} 
This follows from \cite[Cor.\ 4.1.3.3]{Lur09b}. 
In fact, we need to show that for any $G/H \in \mathcal{O}_{\sF}(G)$, the
category $\mathcal{O}_{\mathcal{A}}(G)_{(G/H)/}$ has weakly contractible nerve. 
In fact, the category 
$\mathcal{O}_{\mathcal{A}}(G)_{(G/H)/}$ is equivalent to the
opposite of the \emph{poset} $\mathcal{P}$ of subgroups of $\mathcal{A}$ that contain 
$H$. 
To see this, observe that an object in $\mathcal{O}_{\mathcal{A}}(G)_{(G/H)/}$
is given by a map of $G$-sets $G/H \to G/K$ for some $K \in \mathcal{A}$,
which is given by multiplication by some $g \in G$. By
conjugating, we observe that this object is isomorphic to a map $G/H \to G/K'$ 
given by multiplication by $1$, so that $K' \subset H$. Thus, the objects up to
isomorphism can
be put in correspondence with $\mathcal{P}$; one checks that the morphisms can
as well. 
The hypotheses imply that $\mathcal{P}$ has a minimal element and is therefore
weakly contractible.
\end{proof} 
Combining with \Cref{Fcompldecomp}, we then obtain: 
\begin{cor} 
\label{completelimitorbit}
Suppose $\mathcal{A}$ is as above.
We then obtain an equivalence of symmetric monoidal $\infty$-categories
\[ (\GSpec)_{\fcompl} \simeq \varprojlim_{G/H \in
\mathcal{O}_{\mathcal{A}}(G)^{op}} \HSpec. \]
\end{cor}

\begin{example} 
Suppose $G$ is a $p$-group and $\sF$ is the family of proper subgroups. Then,
one can take $\mathcal{A}$ to be the collection of proper subgroups of $G$
which contain the Frattini subgroup. 
\end{example}

We note an important special case of \Cref{completelimitorbit}, which is
deduced by taking $\mathcal{A} = \left\{H\right\}$ for $H \trianglelefteq G$ normal:

\begin{cor} 
Suppose $H \trianglelefteq G$ is a normal subgroup. Then there is a natural
$G/H$-action on the symmetric monoidal $\infty$-category $\gsp{H}$ together
with a symmetric monoidal functor
\[ \GSpec \to ( \gsp{H})^{h G/H}  \]
which exhibits $( \gsp{H})^{h G/H}$ as the $\sF$-completion of $\GSpec$ for
$\sF$ the family of subgroups of $G$ that are contained in $H$.
\end{cor}

\begin{remark} 
Let $G$ be a finite group. In this case, one can give an inductive
decomposition of $\GSpec$ as a homotopy limit of $\infty$-categories of the
form $\fun(BG', \sp)$ (for various finite groups
$G'$) using \Cref{Fcompldecomp} and the arithmetic square \eqref{fracsquare}. 
\end{remark}

\subsection{$\sF$-nilpotence}

We keep the notation $A_{\sF}$ from \Cref{family}.

\begin{definition} 
\label{Fnildef}
Given a family $\sF$ of subgroups of $G$, 
we will let $\sFNil \subset \GSpec$ 
denote the subcategory of
$A_{\sF}$-nilpotent objects, or equivalently the thick $\otimes$-ideal
generated by  $\left\{(G/H)_+\right\}_{H \in \sF}$.
We will say that a $G$-spectrum $X$ is \emph{$\sF$-nilpotent} if it belongs to
$\sFNil$.  In this case, we will refer to the integer $\exp_{\sF}(X):=\exp_{A_{\sF}}(X)$
as the \emph{$\sF$-exponent} of $X$.
\end{definition}

 Clearly, $\sF$-nilpotent $G$-spectra are both $\sF$-torsion and
$\sF$-complete, i.e., if $X \in \sFNil$, then
\[ E\sF_+ \wedge X\simeq X \simeq F(E \sF_+, X).  \]
As we discuss in the sequel \cite{MNN15i}, $\sF$-nilpotence is equivalent to
$\sF$-completeness together with the very rapid convergence of the associated
homotopy limit spectral sequence based on a cellular
decomposition of $E\sF_+$.

We now discuss some of the first properties
of $\sF$-nilpotent spectra.
Combining \Cref{compositenilpotent} with \Cref{restensoring}, we find the
following criterion for $\sFNil$: 

\begin{prop} 
An object $X \in \GSpec$ belongs to $\sFNil$ if and only if there exists an
integer $N \in \mathbb{Z}_{\geq 0}$ such that whenever
\[ Y_1 \to Y_2 \to \dots \to Y_N \to X \]
are maps in $\GSpec$ whose restriction to $\HSpec$ is nullhomotopic for each $H
\in \sF$, then the composition $Y_1 \to \dots \to X$ is nullhomotopic (in
$\GSpec$). 
If this is the case, the minimal such $N$ is $\exp_{\sF}(X)$.
\end{prop}

Next, we show that $\sF$-nilpotence can be descended under restriction and
ascended under induction.
\begin{prop} 
\label{nilpandresind}
Suppose $H \leq G$ and let  $\sF$ be a family of subgroups of $G$. 
Let $\sF_H$ be the family of those
subgroups of $H$ which belong to $\sF$. 
\begin{enumerate}
\item  
If $X \in \GSpec$ is $\sF$-nilpotent, then $\Res^G_H X \in \HSpec$ is
$\sF_H$-nilpotent.
\item
If $Y \in \HSpec$ is $\sF_H$-nilpotent, then $\Coind^G_H X\simeq \Ind_H^G X \in \GSpec$ is
$\sF^\prime$-nilpotent for any family $\sF^\prime$ containing $\sF_H$. In particular $\CoInd_H^G X$ is $\sF$-nilpotent.
\end{enumerate}
\end{prop} 
\begin{proof} 
Both assertions follow by applying \Cref{Anilplaxmonoidal}.
\end{proof}

\begin{prop}
\label{fnilintersect}
Let $\sF, \sF'$ be two families of subgroups of $G$. Then a
$G$-spectrum is $\left(\sF \cap \sF'\right)$-nilpotent if and only if 
it is both $\sF$-nilpotent and $\sF'$-nilpotent. 
\end{prop}
\begin{proof} 
This follows from \Cref{nilpAB}. While it is not true that $A_{\sF} \wedge
A_{\sF'} \simeq A_{\sF \cap \sF'}$, one sees easily that 
$A_{\sF} \wedge
A_{\sF'}$ admits the structure of a module over $A_{\sF \cap \sF'}$, and that a
$G$-spectrum is nilpotent for $A_{\sF \cap \sF'}$ if and only if it is $A_{\sF}
\wedge A_{\sF'}$-nilpotent. 
\end{proof} 

We next show that all $\sF$-nilpotence questions can be reduced to the case
where the family $\sF$ is the 
family of \emph{proper} subgroups. 
\begin{prop}
\label{reducetopropersubgroups}
A $G$-spectrum
$X \in \GSpec$ 
is $\sF$-nilpotent if and only if for every subgroup $H \leq G$ with $H
\notin \sF$, the restriction $\Res^G_H X \in \HSpec$ is nilpotent with respect
to the family of proper subgroups of $H$.
\end{prop}
\begin{proof} 
The ``only if'' direction follows from \Cref{nilpandresind}. Therefore, it suffices to show
that if $\Res^G_H X \in \HSpec$ is nilpotent for the family of proper subgroups
for each $H \notin \sF$, then $X$ is $\sF$-nilpotent. 

Without loss of generality, $G \notin \sF$.
For each $H \leq G$, let $\sF_H$ denote 
the family of subgroups of $H$ which belong to $\sF$. 
Then, by induction on $|G|$, we may assume that for every $H \lneq G$, 
we have that the $H$-spectrum $\Res^G_H X$
is $\sF_H$-nilpotent. Inducing, it follows that $\Ind_H^G \Res^G_H X \simeq X
\wedge G/H_+$ is $\sF$-nilpotent for each $H \lneq G$
(\Cref{nilpandresind}). 
In particular, if $A = \prod_{H  \lneq G} F(G/H_+, S^0_G)$, we find that $X \wedge A$
is $\sF$-nilpotent. But since $X$ is $A$-nilpotent by hypothesis (as $X$ is
nilpotent for the family of proper subgroups), 
we conclude that $X$ is $\sF$-nilpotent by \Cref{tensorideal:desc}. 
\end{proof} 

Using this, we can give a criterion for when a \emph{$G$-ring spectrum}
is nilpotent for a family of subgroups.

\begin{thm} 
\label{geofixedpointscontractible}
Let $R \in \GSpec$ be a $G$-ring spectrum (up to homotopy). Then $R \in \sFNil$ if and only if for all $H \notin \sF$, 
the geometric fixed point spectrum $\Phi^H R$ is contractible.
\end{thm} 
\begin{proof} 
By \Cref{reducetopropersubgroups}, we can assume $\sF$ is the family of all proper subgroups
of $G$. In this case, we know by \Cref{ringobjectnilp} that $R \in \sFNil$ if and only if $R$ is
$\sF$-torsion, which happens if and only if the $\sF^{-1}$-localization $\Phi^G R $ is contractible. 
\end{proof}

Given a $G$-ring spectrum which is nilpotent for a family of subgroups, any module
spectrum over it has the same nilpotence property. As a result, we can obtain a
decomposition of the module $\infty$-category: 

\begin{thm} 
\label{decompFnilpmodules}
Suppose $R$ is an $\e{n}$-algebra in $\GSpec$ which is $\sF$-nilpotent. 
Let $\mathcal{A}  \subset \sF$ be a collection of subgroups of $G$ closed under
conjugation and intersection, and such that any subgroup of $\sF$ is
contained in a subgroup belonging to $\mathcal{A}$. 
Then
there is an equivalence of $\e{n-1}$-monoidal $\infty$-categories
\[ \md_{\GSpec}(R) \simeq \varprojlim_{G/H \in
\mathcal{O}_{\mathcal{A}}(G)^{op}} \md_{\HSpec}( \Res^G_H R).  \]
\end{thm} 

\begin{proof} 
For any $\infty$-operad $\mathcal{O}$, the association $\mathcal{C}
\mapsto \mathrm{Alg}_{\mathcal{O}}(\mathcal{C})$ sends homotopy limits in
symmetric monoidal $\infty$-categories
$\mathcal{C}$ to homotopy limits of $\infty$-categories.

For the convenience of the reader, we give a brief explanation of this fact. 
Recall that \cite[\S 2.1]{Lur14} to every symmetric monoidal $\infty$-category
$\mathcal{C}$ one has an $\infty$-category $\mathcal{C}^{\otimes}$ equipped
with a map $q\colon \mathcal{C}^{\otimes} \to N( \mathrm{Fin}_*)$, where
$\mathrm{Fin}_*$ is the category of pointed finite sets. In addition, the
$\infty$-operad $\mathcal{O}$ determines an $\infty$-category
$\mathcal{O}^{\otimes}$ and a functor $p\colon  \mathcal{O}^{\otimes} \to N(
\mathrm{Fin}_*)$. Given a diagram indexed over an $\infty$-category
$\mathcal{I}$ of symmetric monoidal $\infty$-categories $\mathcal{C}_i, i \in
\mathcal{I}$, we have  a symmetric monoidal $\infty$-category 
$\varprojlim_{i \in \mathcal{I}} \mathcal{C}_i$ such
that $(\varprojlim_{i \in \mathcal{I}}
\mathcal{C}_i)^{\otimes} = 
\varprojlim_{i \in \mathcal{I}} (\mathcal{C}_i^{\otimes})
$. Finally, $\mathrm{Alg}_{\mathcal{O}}(\mathcal{C})$ is a full
subcategory of $\fun_{N( \mathrm{Fin}_*)}( \mathcal{O}^{\otimes},
\mathcal{C}^{\otimes})$. This construction sends homotopy limits in
$\mathcal{C}$ to homotopy limits of $\infty$-categories (as $\mathcal{C}
\mapsto \mathcal{C}^{\otimes}$ does), and one checks that the condition that describes 
$\mathrm{Alg}_{\mathcal{O}}(\mathcal{C})$ as a full subcategory of 
$\fun_{N( \mathrm{Fin}_*)}( \mathcal{O}^{\otimes},
\mathcal{C}^{\otimes})$ is compatible with homotopy limits too.

Given a system of symmetric monoidal $\infty$-categories $\mathcal{C}_i$ indexed by 
an $\infty$-category
$\mathcal{I}$ and an algebra object $A = \left\{A_i\right\} \in \mathrm{Alg}(\varprojlim
\mathcal{C}_i)$, we obtain
a decomposition of $\infty$-categories
\[ \md_{\left(\varprojlim_{i \in \mathcal{I}} \mathcal{C}_i\right)}(A) \simeq
\varprojlim_{i \in \mathcal{I}} \md_{\mathcal{C}_i}(A_i).\]
This follows similarly using the $\infty$-operads controlling modules \cite[\S
4.2]{Lur14}.

Therefore, setting $\mathcal{C}=\Sp_\sF$, $\mathcal{C}_{G/H}=\Sp_H$, and $A_{G/H}=\Res_H^G R$ as $G/H$ ranges over $\mathcal{O}_\mathcal{A}(G)^{op}$, \Cref{completelimitorbit} gives the desired decomposition for the
$\infty$-category of $R$-modules in $\GSpec$ which are 
$\sF$-complete, i.e.,
\[ 
\md_{(\GSpec)_{\sF}}(R) \simeq \varprojlim_{G/H \in
\mathcal{O}_{\mathcal{A}}(G)^{op}} \md_{\HSpec}( \Res^G_H R).
\]
Here $(\GSpec)_{\sF}$ is the $\infty$-category of $\sF$-complete $G$-spectra. 
However, every $R$-module is automatically $\sF$-complete since
$R$ is $\sF$-nilpotent. This gives the desired claim.  
\end{proof}

\part{Unipotence for equivariant module spectra}

\section{$U(n)$-unipotence and the flag variety}
\label{unipsec}

In this section, we prove several  results
 on actions of compact Lie groups on modules over a
complex-oriented ring spectrum.
Our main results state that actions of the unitary group are determined by their homotopy fixed
points. For example, to give a $KU$-module equipped with a $U(n)$-action is
equivalent to giving a $F(BU(n)_+, KU)$-module which is complete with respect to the
augmentation ideal in $\pi_0 F(BU(n)_+, KU)$. We will use these techniques to prove
that the Borel-equivariant forms of such theories are nilpotent for the family of abelian
subgroups. 

For our nilpotence statements, the strategy of our argument, which goes back to
ideas of Quillen \cite{Qui71b} and is used prominently by Hopkins-Kuhn-Ravenel
\cite{HKR00}, is that of \emph{complex-oriented descent.}
Let $G$ be a finite group, and embed $G \leq U(n)$. The flag variety
$U(n)/T$ defines a $G$-space with abelian stabilizers and which
thus has a cell decomposition in terms of the orbits $G/A_+, A \leq G$
abelian. We will show
that in equivariant stable homotopy theory, over a base such as 
Borel-equivariant $MU$, the  flag variety actually splits up as a  sum of
copies of the unit. This is easily seen to imply the desired nilpotence statement.
 
\subsection{Unipotence}
We begin with some generalities on symmetric monoidal stable
$\infty$-categories.
Note that we do \emph{not} assume additional hypotheses such as the compactness
of the unit here. 

\begin{definition} 
Let $(\mathcal{C} , \otimes, \mathbf{1})$ be a presentable symmetric monoidal, stable
$\infty$-category where the tensor product commutes with colimits in each variable. We say that
$\mathcal{C}$ is \emph{weakly unipotent} if the unit $\mathbf{1}$ generates
$\mathcal{C}$ as a localizing subcategory. 
\end{definition}

\begin{example}\label{ex:morita}
The module $\infty$-category $\md(R)$ for an $\e{\infty}$-ring $R$ is
weakly unipotent. In
fact, by the symmetric monoidal version of Schwede-Shipley theory \cite[Prop.
7.1.2.7]{Lur14}, it is the \emph{basic} example of a weakly 
unipotent $\infty$-category: more precisely, any weakly unipotent
(symmetric monoidal, stable, etc.) $\infty$-category where the unit is {\em compact} is equivalent to $\md(R)$ for $R$
the $\e{\infty}$-ring of endomorphisms of the unit. 
\end{example} 

\begin{example} 
Consider $\fun(BG, \md(R))$ for $G$ a finite group and $R$ an
$\e{\infty}$-ring. This is generally not weakly unipotent: unless (for some
prime $p$) $G$ is a
$p$-group and  $p $ is
nilpotent in $\pi_0R$, the induced object
$F(G_+, R) \simeq R \wedge G_+$ cannot belong to the localizing subcategory generated by the unit $R$ for
purely algebraic reasons. 

In fact, if there exists a prime number $q \mid |G|$ such that $q$ is not nilpotent in $R$, 
let $G_q \leq G$ be a $q$-Sylow subgroup. Given any $M \in \fun(BG, \md(R))$
that belongs to the localizing subcategory generated by the unit, one sees
by considering long exact sequences that $\pi_*(M)[q^{-1}]$ 
must have trivial $G_q$-action (equivalently, $\sum_{g \in G_q} g$ is an
isomorphism on $\pi_*(M)[q^{-1}]$). However, this is not the case for the
induced object $R \wedge G_+$.  
\end{example} 

The compactness of the unit is crucial in \Cref{ex:morita}, and we do not know how to 
classify weakly unipotent symmetric monoidal
$\infty$-categories $\mathcal{C}$ in general. 
As a result, the following definition (\Cref{def:unipotent}) will play more of a role for us. 
Recall first that if $\mathcal{C}$ is as above, and $R =
\mathrm{End}_{\mathcal{C}}(\mathbf{1})$ is the $\e{\infty}$-ring of endomorphisms of the unit, 
then one has a basic adjunction 
\begin{equation} \label{symmmonoidaladj} (\otimes_R \mathbf{1},
\hom_{\mathcal{C}}(\mathbf{1}, \cdot)) \colon \md(R) \rightleftarrows \mathcal{C},
\end{equation}
where the left adjoint $\md(R) \to \mathcal{C}$ is determined by the condition
that it sends the unit to the unit (in fact, it canonically becomes a
symmetric monoidal functor). 
This adjunction is not an equivalence in general, but it restricts to an
equivalence between perfect $R$-modules and the thick subcategory of
$\mathcal{C}$ generated by the unit. 
Recall \cite[Def.~7.2.4.1]{Lur14} that the $\infty$-category of perfect
$R$-modules is the thick subcategory of $\md(R)$ generated by the unit. 

\begin{prop} 
\label{1generate}
$\mathcal{C}$ is weakly unipotent if and only if
$\hom_{\mathcal{C}}(\mathbf{1}, \cdot)\colon \mathcal{C} \to \md(R)$ is conservative.
\end{prop}

This follows from the following more general lemma. 
\begin{lemma} 
\label{generateifcons}
Let $\mathcal{C}$ be a presentable stable $\infty$-category. Consider a set
$\left\{X_\alpha\right\}_{\alpha \in A}$ of objects in $\mathcal{C}$. 
Then the following are equivalent: 
\begin{enumerate}
\item Given $Y \in \mathcal{C}$, $Y$ is contractible if and only if
$\hom_{\mathcal{C}}(X_\alpha, Y) \in \sp$ is contractible for each $\alpha \in A$. 
\item
The
smallest localizing subcategory containing the $\left\{X_\alpha\right\}_{\alpha
\in A}$ is all
of $\mathcal{C}$.
\end{enumerate}
\end{lemma} 
\begin{proof} 
Assume the second condition, % $\mathcal{C}$ is weakly unipotent, 
and let $Y \in \mathcal{C}$ be an
object such that 
$\hom_{\mathcal{C}}(X_\alpha, Y)$ is contractible for each $\alpha \in A$.
Consider now the collection of $X \in \mathcal{C}$ such that
$\hom_{\mathcal{C}}(X, Y)$ is contractible; it is seen to be a localizing
subcategory, and since it contains the $\left\{X_\alpha\right\}$, it contains all of
$\mathcal{C}$. In particular, $\hom_{\mathcal{C}}(Y, Y)$ is contractible, which
implies that $Y$ is contractible.

Conversely, assume the first condition, i.e. that the $\hom_{\mathcal{C}}(X_\alpha, \cdot)$  ($\alpha\in A$) are
jointly conservative.
The localizing subcategory $\mathcal{C}' \subset \mathcal{C}$ generated by the
$\{X_\alpha\}_{\alpha \in A}$ is presentable \cite[Cor.\ 1.4.4.2]{Lur14} (note that $\mathcal{C}$ is
itself presentable, so the hypotheses of that corollary are met), so that 
the inclusion $\mathcal{C}' \to \mathcal{C}$ has a right adjoint $B$ by the
adjoint functor theorem 
\cite[Cor.\ 5.5.2.9]{Lur09}. It follows that if $X \in \mathcal{C}$, then 
one has a natural fiber sequence
\[ F(X) \to B(X) \to X,  \]
where $B(X) \to X$ is the counit of the adjunction and $F(X)$ is defined to
be the fiber. One sees that for
any $Y \in \mathcal{C}'$, the spectrum $\hom_{\mathcal{C}}(Y, F(X))$ is
contractible. Taking in particular $Y = X_\alpha$ for $\alpha \in A$, we find that $F(X)$ is
contractible by hypothesis and that $B(X) \to X$ is an equivalence, so $X$
belongs to the localizing subcategory generated by the
$\left\{X_\alpha\right\}$. 
\end{proof}

\begin{definition} 
\label{def:unipotent}
$\mathcal{C}$ is 
\emph{unipotent}
if the adjunction
\eqref{symmmonoidaladj}
is a localization, i.e., if $\hom_{\mathcal{C}}( \mathbf{1}, \cdot)$ is fully
faithful.
\end{definition}

\begin{remark} 
We do not know whether there exists a symmetric monoidal,
presentable stable $\infty$-category 
$\mathcal{C}$ which is weakly unipotent but not unipotent. 
\end{remark} 

More generally, one can ask the following question. Let $\mathcal{D}$ be a
presentable stable $\infty$-category and let $X \in \mathcal{D}$ be a
generator, i.e., an object such that $\hom_{\mathcal{D}}(X, \cdot)\colon \mathcal{D}
\to \sp$ is
conservative. By \Cref{generateifcons}, this is equivalent to
supposing that the localizing subcategory generated by $X$ is all of
$\mathcal{D}$. 
In this case, one obtains an adjunction
\begin{equation} \md( \mathrm{End}_{\mathcal{D}}(X)) \rightleftarrows
\mathcal{D} , \label{stableadj} \end{equation}
where the right adjoint is conservative.

\begin{question} 
\label{GPstable}
If $X$ generates $\mathcal{D}$ as  a localizing subcategory, is \eqref{stableadj}
a localization? 
\end{question} 

The answer to the abelian analog of Question~\ref{GPstable} (in the Grothendieck case) is affirmative in
view of the Gabriel-Popescu theorem \cite{GP64}. 
However, in general the answer to Question~\ref{GPstable} can be no. 
\begin{example} 
Let $\mathcal{C} = D(\mathbb{Z}_p)$ be the derived $\infty$-category of
modules over the $p$-adic integers $\mathbb{Z}_p$. We claim that the object $X =
\mathbb{Q}_p \oplus \mathbb{F}_p$ generates $\mathcal{C}$. In fact, the cofiber
sequence
\[ \mathbb{Z}_p  \to \mathbb{Q}_p \to \mathbb{Q}_p/\mathbb{Z}_p  \]
shows easily that the localizing subcategory generated by $X$ contains
$\mathbb{Z}_p$, since the localizing subcategory generated by $\mathbb{F}_p$
contains $\mathbb{Q}_p/\mathbb{Z}_p$. However, we claim that the map
\begin{equation} \hom_{\mathcal{C}}(X, \mathbb{Z}_p) \otimes_{\hom_{\mathcal{C}}(X, X)} X
\to \mathbb{Z}_p  \label{notanequiv}
\end{equation} 
is not an equivalence, so the associated adjunction \eqref{stableadj} is not a
localization. 

Indeed, if \eqref{notanequiv} were an equivalence, then writing 
$\hom_{\mathcal{C}}(X, \mathbb{Z}_p) $ as a filtered colimit (over some
filtered $\infty$-category $\mathcal{J}$) of perfect
$\hom_{\mathcal{C}}(X, X)$-modules, one would conclude that 
$\mathbb{Z}_p = \varinjlim_{\mathcal{J}} Y_j$, where each $Y_j$ belongs to the
thick subcategory of $\mathcal{C}$ generated by $X$. Since $\mathbb{Z}_p \in
\mathcal{C}$ is compact, it follows that $\mathbb{Z}_p$ is a retract of some
$Y_j$. However, the functor $\mathcal{C} \mapsto \mathcal{C}$ given by $X
\mapsto (\widehat{X}_p)[1/p]$ annihilates $X$ (and thus anything in the
thick subcategory that $X$ generates) but does not annihilate
$\mathbb{Z}_p$, a contradiction.
\end{example}

We now state and prove the basic criterion we will use throughout to prove that
$\infty$-categories are unipotent.

\begin{prop}[Unipotence criterion] 
\label{unipcriterion}
Let $\mathcal{C}$ be a presentable, stable, symmetric monodical $\infty$-category where 
the tensor commutes with colimits in each variable, as above. Suppose $\mathcal{C}$ contains an algebra object
$A \in \mathrm{Alg}(\mathcal{C})$ with the following properties:
\begin{enumerate}
\item $A$ is compact and dualizable in $\mathcal{C}$. 
\item $\mathbb{D}A$ generates $\mathcal{C}$ as a localizing subcategory. 
\item $A$ belongs to the thick subcategory generated by the unit.
\end{enumerate}
Then $\mathcal{C}$ is unipotent. More precisely, if $R =
\mathrm{End}_{\mathcal{C}}(\mathbf{1})$, then the natural adjunction
\eqref{symmmonoidaladj} exhibits $\mathcal{C}$ as the completion of $\md(R)$
at $\hom_{\mathcal{C}}(\mathbf{1}, A) \in \mathrm{Alg}(\md(R))$.
\end{prop}
 
\begin{proof} 
Let $A_R  := \hom_{\mathcal{C}}(\mathbf{1}, A) $. Since the adjunction
\eqref{symmmonoidaladj}
establishes an equivalence between perfect $R$-modules and the thick
subcategory of $\mathcal{C}$ generated by the unit, it follows using hypothesis 3, that $A_R  \in
\md(R)$ is a perfect (equivalently, dualizable) algebra object. 
Moreover, we have $A \simeq A_R \otimes_R \mathbf{1} \in \mathcal{C}$.

We claim that the adjunction 
\eqref{symmmonoidaladj} factors through the $A_R$-completion of $\md(R)$. To
see this, it suffices to show that if $M 
\in \md(R)$ is $A_R$-acyclic, then $M \otimes_R \mathbf{1} \in \mathcal{C}$ is
contractible. But we know that $(M \otimes_R A_R) \otimes_R \mathbf{1} \in
\mathcal{C}$ is contractible, so the equivalent object $(M \otimes_R \mathbf{1}) \otimes A  \in
\mathcal{C}$ is too. Thus, $M \otimes_R \mathbf{1} \in \mathcal{C}$ is
contractible since the second assumption implies that tensoring with $A$ is
conservative on $\mathcal{C}$. 

Therefore, by the universal property of the $A_R$-completion (as a Bousfield
localization), we get a new adjunction
\begin{equation} \label{newsadjunction} \md(R)_{A_R-\mathrm{cpl}} \rightleftarrows
\mathcal{C},  \end{equation}
which we claim is an inverse equivalence. To see this, we observe that
$\mathbb{D}A_R$ is
a compact generator for 
$ \md(R)_{A_R-\mathrm{cpl}}$ by \Cref{cptgencompl}. Its image, $\mathbb{D}A$, is a compact
generator for $\mathcal{C}$ by assumption. 
However, the left adjoint of the adjunction \eqref{newsadjunction}
is fully faithful 
on the thick subcategory generated by $\mathbb{D}A_R$ (as the left adjoint in 
\eqref{symmmonoidaladj} is fully faithful on the thick subcategory generated by
the unit). 

Therefore, the left adjoint carries the compact
generator $\mathbb{D}A_R$ to a compact generator of $\mathcal{C}$, and is fully faithful on
the thick subcategory generated by $\mathbb{D}A_R$.  
It follows that the adjunction \eqref{newsadjunction} is an equivalence as
desired: both $\infty$-categories are equivalent to $\md( \mathrm{End}_R(
\mathbb{D}A_R))$. 
\end{proof}

We will also need the following criterion for unipotence. 
Although this criterion requires more hypotheses than \Cref{unipcriterion}, these additional
hypotheses will easily be verified in the case of interest. The main benefit to
the next criterion is that $A$ is not assumed to belong to the thick
subcategory generated by the unit: instead, this is deduced from the
assumptions. 
In our main application, the last hypothesis will translate into the relevance of the Eilenberg-Moore
spectral sequence.

\begin{prop}[Second unipotence criterion] 
\label{unipcriterion2}
Let $\mathcal{C}$ be as above, and let $R =
\mathrm{End}_{\mathcal{C}}(\mathbf{1}) \in \clg(\sp)$. Suppose $\mathcal{C}$ contains an algebra object
$A \in \mathrm{Alg}(\mathcal{C})$ with the following properties:
\begin{enumerate}
\item $A$ is compact and dualizable in $\mathcal{C}$. 
\item $\mathbb{D}A$ is compact and generates $\mathcal{C}$ as a localizing subcategory.
\item The $\infty$-category $\md_{\mathcal{C}}(A)$ is generated as a localizing
subcategory by the $A$-module $A$ itself, and $A$ is a compact object in 
$\md_{\mathcal{C}}(A)$. 
\item 
The natural map of $R$-module spectra
\begin{equation} \hom_{\mathcal{C}}( \mathbf{1}, A) \otimes_{R} \hom_{\mathcal{C}}( \mathbf{1}, A) \to
\hom_{\mathcal{C}}(\mathbf{1}, A \otimes A)\end{equation}
is an equivalence. 
\end{enumerate}
Then the conclusion of \Cref{unipcriterion} holds.\end{prop} 
\begin{proof} 
We claim that the natural map
\begin{equation} 
\label{1mapAnilp}
\hom_{\mathcal{C}}(\mathbf{1}, M) \otimes_R
\hom_{\mathcal{C}}(\mathbf{1}, N) \to \hom_{\mathcal{C}}( \mathbf{1}, M \otimes
N)  \end{equation} 
is an equivalence for any $M, N \in \mathcal{C}$ which are $A$-nilpotent.
It suffices to prove this for $M , N\in \md_{\mathcal{C}}(A)$ 
in view of \Cref{thicksubcatmodulesisnilpotent}.
But for $M, N \in \md_{\mathcal{C}}(A)$, 
both sides of \eqref{1mapAnilp} commute with arbitrary colimits in $M, N$ by
the assumption that $A$ is compact in $\md_{\mathcal{C}}(A)$. It thus suffices (since $\md_{\mathcal{C}}(A)$ is
generated as a localizing subcategory by $A$) to see that 
\eqref{1mapAnilp} is an equivalence for $M, N = A$, which we have assumed as
part of the hypotheses.
The natural equivalence in \eqref{1mapAnilp} implies the $R$-module
$\hom_{\mathcal{C}}(\mathbf{1}, A)$
is dualizable (i.e., perfect) in $\md(R)$ since $A$ is dualizable in $\mathcal{C}$. More generally, if $X$ is any
dualizable object in $\mathcal{C}$ which is $A$-nilpotent, then
$\hom_{\mathcal{C}}(\mathbf{1}, X) \in \md(R)$ is dualizable.

Let $\mathcal{D} \subset \mathcal{C}$ denote the thick subcategory generated by
the unit and each $A  \otimes X$ for $X \in \mathcal{C}$ a dualizable object. 
Observe that $\mathcal{D}$ is closed under duality as $\mathbb{D} A $ is a
retract of $A \otimes \mathbb{D} A$ (see \Cref{dualisamodule}).
Moreover, the natural map \eqref{1mapAnilp} is an equivalence if $M, N \in
\mathcal{D}$.

Let $\md^\omega(R)$ denote the $\infty$-category of perfect $R$-modules. 
As a result, we can restrict the adjunction $\md(R) \rightleftarrows
\mathcal{C}$ to a new adjunction
\[ \md^\omega(R) \rightleftarrows \mathcal{D}.  \]
The right adjoint in this adjunction is strictly symmetric monoidal, so by
\Cref{whenequivdual} below, we can conclude that $\mathcal{D} \simeq \md^\omega(R)$
and that $\mathcal{D}\subseteq\mathcal{C}$ is in fact the thick subcategory generated by the unit.
In particular, $A \in \mathcal{C}$ therefore belongs to the thick subcategory generated by the
unit. We can now apply 
\Cref{unipcriterion} to conclude.
\end{proof}

\begin{lemma} 
\label{whenequivdual}
Let $\mathcal{D}_1, \mathcal{D}_2$ be symmetric monoidal $\infty$-categories.
Suppose every object of $\mathcal{D}_1, \mathcal{D}_2$
is dualizable. Suppose we have a symmetric monoidal functor $F\colon \mathcal{D}_1 \to
\mathcal{D}_2$ with a \emph{strictly} symmetric monoidal right adjoint $H$. Then the adjunction $(F, H)$ is
an inverse equivalence of symmetric monoidal $\infty$-categories. 
\end{lemma} 
\begin{proof} 
We first show that $H$ is a fully faithful functor. To see this, we fix
$X, Y \in \mathcal{D}_2$ and use
\begin{align*} 
\hom_{\mathcal{D}_2}(X, Y) &  \simeq \hom_{\mathcal{D}_2}( \mathbf{1}, \mathbb{D}
X \otimes Y) \\
& \simeq \hom_{\mathcal{D}_1}( \mathbf{1}, H( \mathbb{D} X \otimes Y) ) \\
& \simeq \hom_{\mathcal{D}_1}( \mathbf{1}, H( \mathbb{D} X) \otimes H( Y)) \\
& \simeq \hom_{\mathcal{D}_1}( \mathbf{1}, \mathbb{D}H(  X) \otimes H( Y)) \\
& \simeq \hom_{\mathcal{D}_1}( H(X), H(Y)),
\end{align*}    
as desired. 
Dualizing this argument, we can also conclude that $F$ is fully faithful.
Therefore, the adjunction is an inverse equivalence. 
\end{proof}

The preceding lemma is presumably well-known to category theorists. 
We will also need a converse
of sorts to these results:
\begin{cor} 
\label{thickcompact}
Let $\mathcal{C}$ be as above with $R = \mathrm{End}_{\mathcal{C}}(\mathbf{1})$.
Suppose 
that $\mathcal{C}$ is unipotent. Then any compact object of $\mathcal{C}$ belongs to the thick
subcategory generated by the unit. In particular, if $X, Y \in \mathcal{C}$ are
two compact objects, then the natural map
\[ \hom_{\mathcal{C}}( \mathbf{1}, X)  \otimes_R \hom_{\mathcal{C}}( \mathbf{1},
Y) \to \hom_{\mathcal{C}}( \mathbf{1}, X \otimes Y)     \]
is an equivalence.
\end{cor} 
\begin{proof} 
The second assertion clearly follows 
from the first since it is true for $X=Y = \mathbf{1}$ and those  pairs $(X,Y)$ satisfying the assertion form a thick subcategory in each variable. 

Now suppose $X\in \mathcal{C}$. Then since $\mathcal{C}$ is unipotent, we know that the natural map
\[ \hom_{\mathcal{C}}( \mathbf{1}, X) \otimes_R \mathbf{1} \to X  \]
is an equivalence.

Now, by the theory of $\mathrm{Ind}$-objects in $\infty$-categories \cite[\S 5.3]{Lur09}, we can write the $R$-module 
$ \hom_{\mathcal{C}}( \mathbf{1}, X)$ as a filtered colimit of perfect
$R$-modules. That is, there exists a filtered $\infty$-category $\mathcal{I}$
and a functor $f\colon \mathcal{I} \to \md(R)$ such that:
\begin{enumerate}
\item For each $i \in \mathcal{I}$, $f(i) \in \md(R)$ is a perfect $R$-module. 
\item $\hom_{\mathcal{C}}( \mathbf{1}, X)$ is identified with
$\varinjlim_{ i \in \mathcal{I}} f(i)$.
\end{enumerate}
Therefore, we find that 
\[ X \simeq  \varinjlim_{ i \in \mathcal{I}} \left(  f(i) \otimes_R \mathbf{1}\right), \]
where each 
$f(i) \otimes_R \mathbf{1} \in \mathcal{C}$ belongs to the thick subcategory
generated by the unit.
When $X$ is compact, it follows that $X$ is a retract of $
f(i) \otimes_R \mathbf{1}$ for some $i$, proving the claim. 
\end{proof}

\subsection{The Eilenberg-Moore spectral sequence}

We now connect the abstract discussion of unipotence above to a very classical question
when $\mathcal{C} = \mathrm{Fun}(X, \md(R))$ for $X$ a
connected space and
$R$ an $\e{\infty}$-ring, so that $\mathcal{C}$ parametrizes (by definition)
local systems of $R$-modules on $X$. 

\begin{definition}[{Cf. \cite[\S 1.1]{DAGXIII}}] 
Choose a basepoint $\ast \in X$, and consider the pullback square
\begin{equation} \label{pullbacksquare} \xymatrix{
\Omega X \ar[d] \ar[r] &  \ast \ar[d] \\
\ast \ar[r] &  X,
}\end{equation}
and the induced square of $\e{\infty}$-rings
\begin{equation} \label{pushoutcalg} \xymatrix{
F(X_+, R) \ar[d] \ar[r] &  R \ar[d]  \\
R \ar[r] &  F({\Omega X }_+, R).
}\end{equation}
We say that the \emph{$R$-based Eilenberg-Moore spectral sequence
(EMSS) is relevant for
$X$} if \eqref{pushoutcalg} is a pushout of $\e{\infty}$-rings, i.e., if the 
induced morphism
\begin{equation} \label{emmap} R \otimes_{F(X_+, R)} R \to F( \Omega X_+ , R) 
\end{equation}
is an equivalence. If so, we obtain a strongly convergent $\mathrm{Tor}$-spectral sequence
\begin{equation} \label{EMSS} E^2_{p,q} = \mathrm{Tor}_{p,q
}^{\pi_*(F(X_+, R))}(\pi_*(R), \pi_*(R)) \implies \pi_*(
F(\Omega X_+, R)), \end{equation}
which we call the \emph{$R$-based Eilenberg-Moore spectral sequence (EMSS).}
\end{definition} 

If the $R$-based EMSS is relevant, the 
spectral sequence \eqref{EMSS} reduces to the classical
(cohomological) Eilenberg-Moore spectral sequence in case $R = Hk$ for $k$ a field and if $X$
has finitely generated homology in each degree.

\begin{cons}
We now give another interpretation of the $R$-based EMSS. Observe that the
pullback square
\eqref{pullbacksquare}
can be interpreted as a square in $\mathcal{S}_{/X}$, the $\infty$-category of
spaces over $X$. 
Recall that there is an equivalence
of $\infty$-categories
\[ \mathcal{S}_{/X} \simeq \mathrm{Fun}(X, \mathcal{S}) \simeq \mathrm{Fun}(X^{op}, \mathcal{S}),  \]
by the Grothendieck construction \cite[\S 2.1]{Lur09}. 

Here, since $X$ is a connected space, $\mathrm{Fun}(X, \mathcal{S})\simeq \mathrm{Fun}(B\Omega X,\mathcal{S})$ can be
identified with the $\infty$-category of spaces equipped
with an action of $\Omega X$ (where, as before, we implicitly choose a basepoint of $X$).
In particular, when one works in $\mathrm{Fun}(X, \mathcal{S})$, one has the
pullback square
\begin{equation} \label{prodsq} \xymatrix{
\Omega X \times \Omega X \ar[d] \ar[r] &  \Omega X \ar[d] \\
\Omega X \ar[r] &  \ast,
}\end{equation}
where $\Omega X$ is given the action of $\Omega X $ by 
left multiplication; this corresponds to $\ast \in \mathcal{S}_{/X}$ in view
of the fiber sequence $\Omega X \to \ast \to X$. This pullback square in $\mathrm{Fun}(X, \mathcal{S})$
corresponds via the Grothendieck construction to the cartesian square
\eqref{pullbacksquare} in $\mathcal{S}_{/X}$. 

Consider now the functor $\mathcal{S}^{op} \to \clg_{R/}$, $Y \mapsto F(Y_+, R)$. 
We apply it to 
\eqref{prodsq}.
We obtain a commutative algebra object $A \in \mathrm{Fun}(X, \md(R))$ given by
$F(\Omega X_+ , R)$ with the natural $\Omega X$-action. 
In particular, we obtain a 
square in $\clg(\mathrm{Fun}(X, \md(R)))$,
\begin{equation} \label{cocartsq1} \xymatrix{
F(\ast_+, R) \ar[d] \ar[r] &  F( \Omega X_+ , R) \ar[d] \\
F(\Omega X _+, R) \ar[r] &  F( (\Omega X \times \Omega X) _+, R). 
}\end{equation}
When we apply the lax symmetric monoidal functor 
$\hom_{\mathcal{C}}(\mathbf{1}, \cdot)\colon \fun(X, \md(R)) \to \md(F(X_+, R))$ to \eqref{cocartsq1}, 
we obtain \eqref{pushoutcalg}, in view of the correspondence between 
\eqref{prodsq} 
and \eqref{pullbacksquare}.
 \end{cons}

Suppose now $\Omega X$ has the homotopy type of a finite cell complex. Then
$F((\Omega X \times \Omega X) _+, R) \simeq A \otimes A \in \mathrm{Fun}(X,
\md(R))$, i.e., \eqref{cocartsq1} is a pushout of commutative algebra objects in $\mathcal{C}$.
We thus obtain: 
\begin{prop} 
\label{whendoesEMSSconverge}
Suppose $\Omega X $ has the homotopy type of a finite cell complex. Let $A \in
\mathcal{C} = \mathrm{Fun}(X, \md(R))$ be the commutative algebra object
$A=F(\Omega X_+,R) $. Then
the $R$-based EMSS is relevant for $X$ if and only if the natural map
of $F(X_+, R)$-modules
$$\hom_{\mathcal{C}}(\mathbf{1}, A) \otimes_{F(X_+, R)}
\hom_{\mathcal{C}}(\mathbf{1}, A) \to
\hom_{\mathcal{C}}(\mathbf{1}, A \otimes A),$$
is an equivalence.
\end{prop}

\begin{proof} 
The square in \eqref{cocartsq1} is a pushout in $\clg(\mathrm{Fun}(X, \md(R))$ since $\Omega X $ has the homotopy type of a finite cell complex. By applying the lax symmetric monoidal functor
\[\hom_{\mathcal{C}}(\mathbf{1}, \cdot) \colon \fun(X, \md(R)) \to 
\md( F(X_+, R))\] to this pushout we obtain \eqref{pushoutcalg}. Hence the $R$-based EMSS is relevant for $X$ if and only if $\hom_{\mathcal{C}}(\mathbf{1}, \cdot)$ takes this particular pushout to a pushout. 
\end{proof}

We are now ready for the main result of this subsection.

\begin{thm} 
\label{mainEMSS}
Let $G$ be a compact Lie group and $R$ an $\mathbb{E}_\infty$-ring. Then the
$R$-based EMSS is relevant for $BG$ if
and only if the symmetric monoidal $\infty$-category
$\mathcal{C} = \mathrm{Fun}(BG, \md(R))$ is unipotent.
In this case, the $F(BG_+, R)$-module $R$ is perfect, and 
$\mathrm{Fun}(BG, \md(R))$ is identified with the symmetric monoidal $\infty$-category
of $R$-complete $F(BG_+, R)$-modules.
\end{thm}

\begin{proof} 
As above, we consider the algebra object $A = F(G_+, R) \in \mathcal{C}$.
Using equivariant Atiyah duality (see \cite[Th. III.5.1]{LMSM86} for the
genuinely equivariant result), one sees that  $A$
is some suspension of the induced object $ R \wedge G_+ = \mathbb{D}A$. Since the induced object $R \wedge G_+$ is a compact generator for
$\mathcal{C}$, it follows  
that $A$ is a compact generator as well. 

Suppose $\mathcal{C}$ is unipotent. 
Then we apply \Cref{thickcompact} to conclude that $A$ belongs to the thick
subcategory generated by the unit and that 
$\hom_{\mathcal{C}}(\mathbf{1}, A) \otimes_{F(X_+, R)}
\hom_{\mathcal{C}}(\mathbf{1}, A) \to
\hom_{\mathcal{C}}(\mathbf{1}, A \otimes A)$
is an equivalence. 
It follows by \Cref{whendoesEMSSconverge} that the $R$-based EMSS is relevant for
$BG$. The remaining assertions now follow from \Cref{unipcriterion} applied
to $A$. 

Conversely, if 
$\hom_{\mathcal{C}}(\mathbf{1}, A) \otimes_{F(X_+, R)}
\hom_{\mathcal{C}}(\mathbf{1}, A) \to
\hom_{\mathcal{C}}(\mathbf{1}, A \otimes A)$
is an equivalence, we want to apply \Cref{unipcriterion2} to conclude that
$\mathcal{C}$ is unipotent. In order to do this, we need to analyze $A$-modules
in $\mathcal{C}$. 
To do this, 
consider the inclusion $\ast \to BG$ and the induced adjunction
\[ \mathcal{C} = \fun(BG, \md(R))  \rightleftarrows \md(R), \]
where the left adjoint restricts to a basepoint and the right adjoint sends $M
\in \md(R)$ to the coinduced object $F(G_+, M)$. 
In particular, the right adjoint carries the unit of $\md(R)$ to $A \in
\mathcal{C}$.
Using 
\Cref{overadjunctionequiv}, it follows that we have an equivalence
of symmetric monoidal $\infty$-categories
$\md_{\mathcal{C}}(A) \simeq \md(R)$, so that $\md_{\mathcal{C}}(A)$ has $A$ as
compact generator. Therefore, we have all the ingredients to apply
\Cref{unipcriterion2} and conclude the argument. 
\end{proof}

\begin{cor} 
\label{ascendEMSS}
Suppose the $R$-based EMSS is relevant for $BG$ for $G$ a compact Lie group.
Then it is relevant for $R'$ if $R'$ is any $\e{\infty}$-ring such that there
exists a map of $\e{1}$-rings $R \to R'$.
\end{cor} 
\begin{proof} 
We use \Cref{mainEMSS}. As the EMSS is relevant for $BG$, it follows that
$\mathcal{C} = \fun(BG, \md(R))$ is unipotent, and the coinduced object $F(G_+, R)$ belongs to the thick subcategory generated by the unit by
\Cref{thickcompact}. By base-change to $R'$, we find that 
$F(BG_+, R') \in 
\fun(BG, \md(R'))$ belongs to the thick subcategory generated by $R'$ with
trivial $G$-action. 
Now we can apply \Cref{unipcriterion} to $\mathcal{C}' = \fun(BG, \md(R'))$ to obtain that 
$\mathcal{C}'$ is unipotent too, so that the $R'$-based EMSS is relevant for $BG$. 
\end{proof}

\begin{prop} 
\label{descendEMSS}
Let $R \to R'$ be a descendable  morphism of $\e{\infty}$-rings (i.e., $R
$ is descendable as a commutative algebra in $\md(R)$, cf.\  \Cref{descent2}). 
Suppose the map $F(BG_+, R) \otimes_{R} R' \to F(BG_+, R')$ is an equivalence. 
Then the $R$-based EMSS is relevant for $BG$ if and only if the $R'$-based EMSS
is relevant.
\end{prop} 
\begin{proof} The ``only if'' implication is given by \Cref{ascendEMSS}. For the converse,
we want to show that the natural map 
$R \otimes_{F(BG_+, R)} R \to F(G_+, R)$ is an equivalence. To do so, since $R
\to R'$ is descendable, it suffices to show that the base-change to $R'$ is an
equivalence. But by hypothesis (and the fact that $G$ is a compact Lie group),
this is precisely the map $R' \otimes_{F(BG_+, R')} R' \to F(G_+, R')$.
\end{proof} 

\begin{cor} 
Let $R \to R'$ be a descendable morphism of $\e{\infty}$-rings such that $R'$
is a perfect $R$-module, and let $G$
be a compact Lie group. Then the $R$-based EMSS is relevant for $BG$ if and only
if the $R'$-based one is. 
\end{cor} 
\begin{proof} 
In fact, since $R'$ is a perfect $R$-module, the map $F(BG_+, R) \otimes_R R'
\to F(BG_+, R')$ is an equivalence, so we can apply \Cref{descendEMSS}.
\end{proof} 

The relevance of the EMSS, especially over $H \mathbb{Z}$ and $H
\mathbb{F}_p$, has been treated classically in
numerous sources, e.g., \cite{Dwy74, Dwy75a}, and is discussed for complex
$K$-theory in \cite{JO99}. 
A more recent development is the ambidexterity theory of Hopkins-Lurie
\cite{ambidexterity}. 
For example, in \cite[Th. 5.4.3]{ambidexterity}, they show (as a special case) that for $G$ a $p$-group the
$\infty$-category of $K(n)$-local modules over Morava $E$-theory with
$G$-action (at the prime
$p$) is unipotent; the analogous assertion about the EMSS is earlier work of
Bauer \cite{bauer}.

\subsection{The categorical argument}
In \Cref{mainEMSS}, we saw that the unipotence of $\infty$-categories of the
form $\fun(BG, \md(R))$, for $G$ a compact Lie group and $R$ an $\e{\infty}$-ring, is equivalent to the
relevance of the $R$-based EMSS for the space $BG$. 

The purpose of this subsection is to obtain a basic and easily checked
sufficient criterion for relevance
of the EMSS. 
For a given compact connected Lie group $G$, this criterion will always be applicable to $\e{\infty}$-rings
such that $\pi_*(R)$ is torsion-free away from a finite number of primes
(compare \Cref{exteriorcoh} below). 
Therefore, we will be able to prove that several such $\infty$-categories are
unipotent.

In the next subsection, we shall 
give a slightly different (and more geometric) variant of the following
argument. 
We have included both arguments in this paper. The present argument seems to
be more widely applicable. However, the geometric one generalizes better to the
genuinely equivariant setting. 

\begin{prop} 
\label{intconverge}
Suppose  $G$ is a compact, connected Lie group. Then the $\mathbb{Z}$-based
EMSS is relevant for $BG$
(and thus, by \Cref{ascendEMSS}, the $R$-based EMSS is relevant for $BG$ if $R$ is any discrete
$\e{\infty}$-ring). 
\end{prop} 
\begin{proof} 
Using \Cref{unipcriterion} with $\mathcal{C}=\fun(BG,\md(\mathbb{Z}))$ and
$A=F(G_+,\mathbb{Z})$,
it suffices to show that the induced object $G_+ \wedge \mathbb{Z} \in \fun(BG,
\md(\mathbb{Z}))$ belongs to the thick subcategory generated by the unit. This
follows easily by working up the (finite) Postnikov decomposition of 
$G_+ \wedge \mathbb{Z}$: each of the successive cofibers has trivial $G$-action, because 
$G$ is connected, and finitely generated homotopy. 
\end{proof}

\begin{remark} 
Using a similar argument, combined with the fact that nontrivial
representations of $p$-groups in characteristic $p$ always have nontrivial fixed points, one can argue that if $\pi_0 (G)$ is a $p$-group,
then the $\mathbb{F}_p$-based EMSS is relevant for $BG$. 
As a result, we have an equivalence of symmetric monoidal $\infty$-categories
between $\fun(BG, \md( \mathbb{F}_p))$ and complete modules over
$F(BG_+, \mathbb{F}_p)$ (cf. \Cref{mainEMSS}). 
\end{remark} 

Here is our main result:

\begin{thm}\label{thm:relevancevsunipotence}
\label{polyEMSS}
Let $R$ be an $\e{\infty}$-ring  and let $G$ be  a compact, connected
Lie group. Suppose 
$H^*(BG; \pi_0 R)\simeq H^*(BG; \mathbb{Z})\otimes_{\mathbb{Z}}\pi_0 R$ and
this is a polynomial ring over $\pi_0 R$. Suppose moreover that the
cohomological $R$-based AHSS for $BG$ degenerates (e.g., $\pi_*(R)$ is
torsion-free).
Then the $R$-based EMSS for $BG$ is relevant, so that $\mathrm{Fun}(BG, \md(R))$
is unipotent and equivalent to the symmetric monoidal $\infty$-category of
$R$-complete $F(BG_+, R)$-modules.
\end{thm} 

%\begin{remark}
%For any given compact, connected Lie group $G$, we know that there exists an integer $N$ such
%that $H^*(G; \mathbb{Z}[1/N])$ is an exterior algebra 
%and $H^*(BG; \mathbb{Z}[1/N])$ is a polynomial algebra over $\mathbb{Z}[1/N]$. We will be more specific about the minimal such
%$N$ at the end of this subsection.
%\end{remark}

\begin{proof}[Proof of \Cref{polyEMSS}] 
Without loss of generality, we may assume that $R$ is connective by
\Cref{ascendEMSS} and replacing $R$ by $\tau_{\geq 0} R$ if necessary.
Choose classes $x_1, \dots, x_r \in H^*( BG; \pi_0 R) $
which form a system of polynomial generators, so that $H^*(BG; \pi_0 R) \simeq
\pi_0 R[x_1, \dots, x_r]$. Let $k_i = |x_i|$.
Choose lifts $y_1, \dots, y_r \in \widetilde{R}^*(BG)$, which we can by the
degeneration of the AHSS. 
For each $i$, $y_i$ classifies a self-map
$ \Sigma^{-k_i}R \to R$    in the $\infty$-category $ \fun(BG, \md(R)).$ 

Consider the coinduced object $F(G_+, R) \in \clg( \fun(BG, \md(R)))$. 
One has a unit map 
\( R \to F(G_+, R).  \)
Observe also that the homotopy fixed points of $F(G_+, R)$ are given by $R$
itself. 
As a result, for each $i$, the composite map
\[ \Sigma^{-k_i} R \stackrel{y_i}{\to}  R \to F(G_+, R) \]
is nullhomotopic; this follows because $y_i$ restricts to zero in $\pi_*(R)
= \pi_* F(\ast_+, R)$. In particular, for each $i$ we obtain maps
in $\fun(BG, \md(R))$,
\[ R/y_i \to F(G_+, R).  \]
On homotopy fixed points, these classify maps
of $F(BG_+, R)$-modules
\[ F(BG_+, R)/y_i \to R  \]
that extend the map $F(BG_+, R) \to R$ given by evaluating at a point.
Using the $\e{\infty}$-structure on $F(G_+, R)$, we obtain a map
\begin{equation} \label{tensorcofib} \bigotimes_{i=1}^r R/y_i \to F(G_+, R).
\end{equation} 
%Note: this is equivalent to checking the fiber is contractible. $N$ is bounded below  then $N\otimes_R \pi_0 R$ has the same bottom group as $N$ by examing the tor-ss.
We claim that \eqref{tensorcofib} is an equivalence. 
In order to see this, it suffices (as the underlying $R$-modules of both
objects are bounded below) to base change along the map $R \to \pi_0 R$,
so that we may assume that $R$ is discrete. In this case, it suffices to see
that 
the map
\eqref{tensorcofib} induces an equivalence on homotopy fixed points 
by \Cref{intconverge}. But this is the claim that
we have an equivalence of $F(BG_+, \pi_0 R)$-modules
\[ F(BG_+, \pi_0 R) /(x_1, \dots, x_r) \simeq \pi_0 R,  \]
which we have assumed. 
As a result, it follows that the coinduced object $F(G_+, R) \in \fun(BG,
\md(R))$ belongs to the thick subcategory generated by the unit, so we may
apply \Cref{unipcriterion} and conclude unipotence. 
\end{proof}

We now obtain a basic result for unipotence for actions of the classical
compact Lie groups
(with $2$ inverted for the $SO(n)$ family).
A classic textbook reference for the calculations
of the cohomology of the relevant classifying spaces is \cite{Mis74}. 

\begin{thm} 
\label{basicunipex}
Suppose $R$ is an $\e{\infty}$-ring with $\pi_*(R)$ torsion-free. 
Then $\fun(BG, \md(R))$ is unipotent if $G $ is a product of copies of $ U(n), SU(n), Sp(n)$ for some $n$.
If $2$ is invertible in $R$, then one can also include factors of $ SO(n),
Spin(n)$. 
In particular, we have an equivalence
of symmetric monoidal $\infty$-categories
\[ \fun(BG, \md(R))  \simeq \md( F(BG_+, R))_{\mathrm{cpl}}, \]
where on the right we consider $F(BG_+, R)$-modules which are complete with respect
to  the $F(BG_+, R)$-module $R$.
\end{thm}

We also include a counterexample to show the necessity of inverting 2 in the
presence of $SO(n)$-factors.

\begin{example} 
The $\infty$-category $\fun(BSO(3), \md(KU))$ is \emph{not} unipotent. 
When one works 2-adically, $KU^*(BSO(3))$ is a power series ring on one
variable, but $KU^*(SO(3))$ has 2-torsion, and the EMSS does not converge
(compare
the remark at the end of \cite[\S 4]{JO99}). 
\end{example}

On the other hand, we shall see that $\fun(BG, \md(KU))$
is unipotent if $G$ has no torsion
in $\pi_1$. So, there are other examples of unipotence not covered by 
\Cref{basicunipex} (such as $G = \mathrm{Spin}(n), n \geq 4$).

Finally, we note that we can recover results of Greenlees-Shipley \cite{GS11}. 
\begin{example} 
Let $G$ be a connected compact Lie group. Then (by \Cref{intconverge}) for $R =
\mathbb{Q}$, the $\infty$-category $\fun(BG, \md(\mathbb{Q}))$ is unipotent. Therefore, we
find that $\fun(BG, \md(\mathbb{Q}))$ is equivalent, as a symmetric monoidal
$\infty$-category, to modules over $F(BG_+, \mathbb{Q})$
which are complete with respect to the augmentation ideal. This is closely
related to the main result of \cite{GS11}. 

Strictly speaking, Greenlees-Shipley work in the genuine equivariant setting in \cite{GS11};
however, they work with \emph{free} $G$-spectra, so that it is equivalent to
work in $\fun(BG, \md(\mathbb{Q}))$. Moreover, they use torsion instead of
complete $F(BG_+, \mathbb{Q})$-modules.
Note also that $F(BG_+, \mathbb{Q})$ is equivalent to a free
$\e{\infty}$-ring over $\mathbb{Q}$ on a finite number of generators, since
$H^*(BG; \mathbb{Q})$ is a polynomial ring. 
\end{example} 

Fix a compact connected Lie group $G$. In order to make the assumptions of
\Cref{thm:relevancevsunipotence} more explicit, we now determine the minimal 
integer $n$ with respect to divisibility such that $H^*(BG;\mathbb{Z}[1/n])$ is a polynomial algebra.
To formulate the result, recall that $G$ contains a maximal semi-simple subgroup $G_{ss}\subset G$, and 
that $G$ is homeomorphic to $G_{ss}\times T$ for a torus $T\subset G$. The group $G_{ss}$ in turn is uniquely a finite
product of simple groups, the simply-connected covers of which are simply connected, simple Lie groups; these are classified
by their Lie algebras. We refer to the finite list of Lie algebras thus associated with $G$ as the types occurring in $G$.

\begin{thm}\label{exteriorcoh}
Let $G$ be a compact connected Lie group and $n\ge 1$ an integer. Then the following are equivalent:

\begin{enumerate}
\item The $\mathbb{Z}[1/n]$-algebra $H^*(BG,\mathbb{Z}[1/n])$ is polynomial.
\item The $\mathbb{Z}[1/n]$-algebra $H^*(G,\mathbb{Z}[1/n])$ is exterior.
\item The integer $n$ is divisible by each of the following primes $p$:
\begin{itemize}
\item Each $p$ which occurs as the order of an element of $\pi_1(G)$. 
\item The prime $p=2$ if $G$ contains a factor of type $\mathrm{Spin}(N)$ for some $N\ge 7$, $G_2$, $F_4$, $E_6$, $E_7$ or $E_8$.
\item The prime $p=3$ if $G$ contains a factor of type $F_4$, $E_6$, $E_7$ or $E_8$.
\item The prime $p=5$ if $G$ contains a factor of type $E_8$.
\end{itemize}
\end{enumerate}
\end{thm}

\begin{example} The conditions in \Cref{exteriorcoh} depend on $G$ only through $G_{ss}$. The algebra $H^*(BG;\mathbb{Z})$ itself is polynomial
if and only if the semi-simple part of $G$ is simply connected and contains none of the types listed; this holds for example for $G=U(N)$.
The minimal $n$ such that $H^*(BSO(N);\mathbb{Z}[1/n])$ is polynomial is $n=2$:
The cover $\mathrm{Spin}(N)\to SO(N)$ is simple and simply connected, and
the map has degree $2$.
\end{example}

\begin{proof}[Proof of \Cref{exteriorcoh}]
We first show the equivalence of 2. and 3.
Using the obvious generalization from $\mathbb{Z}$- to $\mathbb{Z}[1/n]$-coefficients of
\cite[Proposition 1.2]{borel} for $X=G$, we see that $H^*(G;\mathbb{Z}[1/n])$ is exterior if and only if it is torsion-free,
i.e., if and only if $n$ is divisible by all primes $p$ such that $H^*(G;\mathbb{Z})$ has non-trivial $p$-torsion. As $G\simeq G_{ss}\times T$, these
are exactly the torsion primes of $H^*(G_{ss};\mathbb{Z})$. By \cite[Lemme 3.3]{borel}, passage from $G_{ss}$ to its simply connected cover
exactly picks up those $p$ which divide the (finite) order of $\pi_1(G_{ss})$.
Finally, for the semi-simple, simply connected case, the torsion primes
are listed in \cite[Th\'eor\`eme 2.5]{borel}.
To see that 1. is equivalent to 3., we observe that by \cite[Th\'eor\`eme 4.5]{borel}, condition 3. 
in \Cref{exteriorcoh} is also equivalent to $H^*(BG,\mathbb{Z}[1/n])$
being torsion free. By the natural adaptation of \cite[Th. 19.1]{Bor53} to
$\mathbb{Z}[1/n]$-coefficients, we find that $H^*(BG ;\mathbb{Z}[1/n])$ is a
polynomial algebra as desired in precisely these cases. 
\end{proof}

\subsection{The geometric argument}

Let $E$ be an $\e{\infty}$-ring which is complex-orientable as an
$\e{1}$-ring; as shown in \Cref{prop:orientations} this condition is often satisfied in practice. In this subsection, we shall describe actions of compact Lie groups $G$ on
$E$-modules where $G$ is a product of unitary groups.  
Rather than going through the EMSS as in the previous subsection, we shall use
complex-orientability instead and \Cref{unipcriterion}. 
The use of complex-orientability also appears in \cite{GS14}, and our methods
are closely related to theirs. However, our results will be
strictly contained in \Cref{basicunipex}.

\begin{prop}\label{prop:orientations}
	Suppose that $E$ is an $\e{2}$-ring and $\pi_*E$ is concentrated in even degrees. Then there is a morphism $MU\to E$ of $\e{1}$-rings. 
\end{prop}
\begin{proof}
	By passing to connective covers we can assume $E$ is connective. Using the obstruction theory of \cite{ABGHR14}, it suffices to show that the composite $f\colon BU\stackrel{J}{\to} BGL_1 S\to BGL_1 E$ of based spaces is null-homotopic.  For this purpose, we fix a cell structure on $BU$ using only even dimensional cells and inductively extend a null-homotopy over the skeleta of $BU$. Since $E$ is $\e{2}$, $BGL_1 E$ is a loop space and hence simple so we can apply elementary obstruction theory \cite[\S 18.5]{May99}.  The relevant obstructions lie in $\wt{H}^{2n+2}(BU; \pi_{2n+2} BGL_1 E)\cong \wt{H}^{2n+2}(BU;\pi_{2n+1}E)=0$ for $n\geq 0$. This builds a compatible sequence of based null-homotopies $H_n\colon BU_{2n} \wedge I_+\to BGL_1 E$ and taking colimits gives the desired null-homotopy of $f$.
\end{proof}

For simplicity, we begin with the case of $G = U(1)$.
Choose a complex orientation $x\colon \Sigma^{-2} BU(1) \to E$.
Observe that $\pi_*(F(BU(1)_+,E)) \simeq \pi_*(E)[[x]] :=
\varprojlim\pi_*(E)[x]/x^n
$ where $x \in
\pi_{-2}(F(BU(1)_+, E))$ is a class that
maps to zero under the map $F(BU(1)_+, E) \to E$ given by evaluation at a point. 

We will now give a geometric proof of unipotence in the case of $U(1)$-actions.

\begin{thm} 
\label{circleunipotent} 
The $\infty$-category $\fun(BU(1), \md(E))$ is unipotent. 
The functor of homotopy fixed points is fully faithful and embeds 
$\fun( BU(1), \md(E)) $ as the subcategory of $x$-complete objects in 
$\md( F(BU(1)_+, E) )$.
\end{thm} 

\begin{proof} 
Let $V$ be the standard one-dimensional complex representation of $U(1)$. 
Consider the Euler cofiber sequence
of pointed spaces with an $U(1)$-action
\begin{equation} \label{standard} 
S(V)_+ \to S^0 \to S^V.
\end{equation}

Here $S(V)$ denotes the unit sphere in $V$ and $S^V$ denotes the one-point
compactification of $V$. 
Note that $S(V)_+$ is \emph{induced} from the trivial group: it is just $U(1)_+$
with the action by translation.  
After smashing with $E$, we get a cofiber sequence in $\fun(BU(1), \md(E))$
given by 
\begin{equation} \label{keycof} E \wedge U(1)_+ \to E \to E \wedge S^V.  \end{equation}

Now, we use the $\e{1}$-complex orientation of $E$ to give the equivalence
\begin{equation} \label{untwist} E \wedge
S^V \simeq \Sigma^2 E  \in \fun(BU(1), \md(E)) ,\end{equation}
in view of the theory of orientations 
of \cite{ABGHR14}. This argument is crucial. 
To see \eqref{untwist}, 
it suffices
to take $E = MU$, and in this case one knows that $MU \wedge \Sigma^{-2}S^V \in
\fun(BU(1), \md(MU))$ factors through $\fun( BU(1), B GL_1(MU))$ via the composition
\[ BU(1) \to BU  \stackrel{J}{\to} BGL_1(S^0) \to BGL_1(MU),  \]
for $J$ the complex $J$-homomorphism. 
The composition is nullhomotopic, which implies that the local system
of $MU$-modules
$\Sigma^{-2} S^V \wedge MU$
over $BU(1)$ is trivial; this is the claim of
\eqref{untwist}. 

Finally, in view of \eqref{keycof} and
\eqref{untwist}, we find that the induced object $E \wedge
U(1)_+$ belongs to the thick subcategory generated by the unit. 
Dualizing, we find that the coinduced object $A = F(U(1)_+, E)$ belongs to the
thick subcategory generated by the unit.
Now, applying \Cref{unipcriterion}, we conclude that $\fun(BU(1), \md(E))$ is
unipotent and is equivalent to the $\infty$-category of modules over $F(BU(1)_+, E)$
which are complete with respect to the $F(BU(1)_+,E)$-module $E$.
Since $ E \simeq F(BU(1)_+, E)/x$, this completes the proof. 
\end{proof}

We now give the analog for any of the unitary or special unitary groups. 
\begin{thm} 
If $G = U(n), SU(n)$, then the $\infty$-category 
$\fun(BG, \md(E))$ is unipotent and equivalent (via homotopy fixed points) to
the $\infty$-category of $E$-complete $E^{BG}$-modules.
\end{thm} 
\begin{proof}  
We consider the case $G = U(n)$.
We will use the criterion of \Cref{unipcriterion}. It suffices to show that the
induced object $E \wedge U(n)_+$ belongs to the thick subcategory generated by the unit.

Here we can work by induction on
$n$. Suppose the induced object $E\wedge U(n-1)_+ \in \fun(BU(n-1),
\md(E))$ belongs to the thick subcategory generated by the unit. Let $V$ be the standard representation of $U(n)$ on $\mathbb{C}^n$. Now we have a cofiber sequence
as in \eqref{standard}, in pointed $U(n)$-spaces. It reads
\[ U(n)/U(n-1)_+ \to S^0 \to S^V . \]
Smashing with $E$, we get a cofiber sequence
\begin{equation} \label{keycof2} E \wedge  U(n)/U(n-1)_+  \to E  \to
\Sigma^{2n}E ,
\end{equation}
where we used the same ``untwisting'' argument as in \Cref{circleunipotent} to
identify $E \wedge S^V$ with $\Sigma^{2n} E$.

Now, by the inductive hypothesis, the induced object $E\wedge U(n-1)_+ \in \fun( BU(n-1), \md(E))$
belongs to the thick subcategory generated by the unit. Inducing
upwards to $U(n)$, it follows that the induced object $E \wedge U(n)_+ \in \fun( BU(n), \md(E))$
belongs to the thick subcategory generated by  $ E \wedge U(n)/U(n-1)_+$. 
However, \eqref{keycof2} shows that $ E \wedge U(n)/U(n-1)_+$ belongs to the
thick subcategory generated by the  unit in $\fun(BU(n), \md(E))$. Therefore, by transitivity, the induced 
object $E \wedge U(n)_+$ belongs to the thick subcategory generated by the
unit, and we can apply \Cref{unipcriterion} to conclude the proof. 
\end{proof}

\subsection{The flag variety}

In this subsection, we include the principal applications of our general
categorical machinery to
nilpotence results. 

Let $T \subset U(n)$ be a maximal torus, and let $F \simeq U(n)/T$ be the \emph{flag variety} of $\mathbb{C}^n$. 
Observe that $F$ has an action of $U(n)$, as a topological space. 
Therefore,
$F_+ \in \fun( BU(n), \sp)$.
Our goal is to show 
that this action can actually be trivialized over a complex-oriented base. 
These ideas go back to \cite{Qui71b, HKR00}.

\begin{prop} 
\label{trivflagvar}
Let $E$ be an $\e{\infty}$-ring which admits an $\e{1}$-complex orientation. Then 
we have an equivalence 
$E \wedge F_+ \simeq \bigoplus_{i=1}^{n!} \Sigma^{k(i)} E$
of objects in $\fun(BU(n), \md(E))$, where the $k(i)$ are even integers. 
\end{prop} 
\begin{proof} 
It suffices to prove this with $F_+$ replaced by its Spanier-Whitehead dual 
$\mathbb{D}F_+$. Since $F \simeq U(n)/T$, the
Spanier-Whitehead dual $\mathbb{D}F_+$ is the coinduction of the unit from
$\fun( B T, \sp)$ to $\fun(BU(n), \sp)$. 
Let $\mathcal{C}  = \fun(BU(n), \md(E))$.
Now, 
\[ \hom_{\mathcal{C}}( \mathbf{1}, \mathbb{D} F_+ \wedge E) \simeq F(BT_+,E) \in
\md( F(BU(n)_+, E)) \]
and this is a free $F(BU(n)_+,E)$-module of rank $n!$ 
as $E^*(BT)$ is a free $E^{*}(BU(n))$-module of rank $n!$ with generators in
even degrees by the general
theory of complex-oriented ring spectra.
By unipotence of $\fun(BU(n), \md(E))$ (\Cref{basicunipex}), this is enough to 
prove the claim. 
\end{proof}

Although we included the general unipotence results for their own interest, 
\Cref{trivflagvar} can be seen directly. Consider the $\infty$-category
$\md^\omega( \underline{E})$, where $\underline{E} \in \GSpec$ is the
Borel-equivariant form of $E$. It suffices to prove that $\underline{E} \wedge
\mathbb{D}F_+$ is a free module over $\underline{E}$; this amounts to checking 
that the homotopy groups of $\underline{E} \wedge \mathbb{D}F_+$ are free over
the homotopy groups of $\underline{E}$, for each subgroup $H \subset G$ 
(cf.\ Recollection~\ref{rec:free} below). This, however, is the formula for the
complex-oriented cohomology of a flag bundle (cf.\ \cite[Prop. 2.4]{HKR00}). We
leave the details to the reader.  

\begin{thm} 
\label{complexorBorequivabelian}
Let $E$ be an $\e{\infty}$-ring which admits an $\e{1}$-complex orientation. 
Let $G$ be any compact Lie group. Then the thick subcategory of  $\fun(BG,
\md(E))$ generated by the $E \wedge G/A_+$, as $A \leq G$ ranges over the abelian
subgroups, contains the unit $E  \in \fun(BG, \md(E))$. 
\end{thm} 
\begin{proof} 
Embed $G \leq U(n)$ for some $n$ and consider the flag variety $F$. As an
$E$-module with $U(n)$-action, we saw in
\Cref{trivflagvar} 
that the unit is a retract of $E \wedge F_+$ in $\fun( BU(n),
\md(E))$. 
Therefore, if we restrict to $G$ and consider $E \wedge F_+$ as an
object in $\fun(BG, \md(E))$, it contains the unit as a retract. But $F$, as a space with $G$-action,
has abelian stabilizers and thus 
admits a finite cell decomposition with cells of the form $G/A \times D^n$ by
the equivariant triangulation theorem \cite{Il83}. In particular, $E \wedge F_+ \in
\fun(BG, \md(E))$ belongs to the \emph{stable} subcategory generated by the $G/A_+$
as $A \leq G$ ranges over the abelian subgroups. This proves the result.  
\end{proof} 

Here again there is a variant for the orthogonal groups. 
Let $T \subset SO(n)$ be a maximal torus, and let $F' = SO(n)/T$ be the real flag
variety. 
One has: 
\begin{thm} 
Let $E$ be an $\e{\infty}$-ring such that $2$ is invertible in $\pi_0(E)$ and
$\pi_*(E)$ is torsion-free. Then, as an object in $\fun( BSO(n), \md(E))$, the
flag variety $F'_+ \wedge E$ is equivalent to a direct sum of copies of shifts
of the unit.
\end{thm} 
\begin{proof} 
%JN: This isn't used in the proof.
%We can assume without loss of generality that $E$ is connective by \Cref{ascendEMSS}.
We know by \Cref{basicunipex} that $\fun(BSO(n), \md(E))$ is equivalent to the
$\infty$-category of complete modules over $F(BSO(n)_+,E)$. The hypotheses imply
that the AHSS for $E^*(BSO(n))$ degenerates (as the differentials are
torsion valued) and we have
that 
\[ E^*(BSO(n)) \simeq \pi_*(E)[u_1, \dots, u_{m}], \quad E^*(BT) \simeq
\pi_*(E)[t_1, \dots, t_m], \]
where $m$ is the rank of $SO(n)$. Moreover, $E^*(BT)$ is a free module over
$E^*(BSO(n))$, so that the same 
reasoning as in \Cref{trivflagvar} can be applied. 
\end{proof} 

%JN: As far as I can tell this result is contained in the theorem that precedes the above the theorem. I moved the label as well
% \begin{cor}
% \label{complexorBorequivabelian}
% Let $G$ be a finite group. 
% If $E$ is an $\e{\infty}$-ring with $\pi_*(E)$ concentrated in even degrees, then the Borel-equivariant
% $G$-spectrum $\underline{E} \in \GSpec$ is nilpotent
% for the family of abelian subgroups. 
% \end{cor} 
% \begin{proof} 
% Recall that 
% the $\infty$-category of dualizable modules in $\GSpec$
% is a full symmetric monoidal subcategory of $\fun(BG, \md(E))$. 
% As a result, the above statements easily imply that $\underline{E} \in \GSpec$
% belongs to the thick $\otimes$-ideal generated by $\left\{G/A_+\right\}$ as $A$
% ranges over the abelian subgroups of $G$. 
% \end{proof} 

\section{Equivariant complex $K$-theory}

In this section, we will study the $\infty$-category of modules over $U(n)$-equivariant $K$-theory
$KU_{U(n)}$ in the $\infty$-category of genuine $U(n)$-spectra. Our main
result will show that the
symmetric monoidal $\infty$-category $\md_{\mathrm{Sp}_{U(n)}}(KU_{U(n)})$ is 
equivalent to the $\infty$-category of modules in spectra over its categorical
fixed points. More generally, we will be able to replace $U(n)$ with any
compact Lie group $G$ with $\pi_1(G)$ torsion-free for this. 
This  is a \emph{unipotence} result for modules over equivariant $K$-theory. 
Note that the unit is compact in the genuine equivariant setting, so the
completeness and convergence issues of the previous section do not arise.

This result 
(which was known to Greenlees-Shipley for $G$ a torus) 
gives a new point of view on the classical question, considered 
by Hodgkin, McLeod, and Snaith, of K\"unneth spectral sequences in equivariant
complex $K$-theory.
In the following section, we will also treat the case of equivariant real $K$-theory
using Galois descent.
For the purposes of nilpotence, it gives (when combined with the equivariant
$K$-theory of the flag variety) an ``explicit'' proof of nilpotence with
respect to the family of abelian groups (as in the previous section). 
In the sequel \cite{MNN15i} to this paper, we shall in fact see that in this result, abelian subgroups can be
replaced by the family of \emph{cyclic} subgroups. 
However, the reduction to the abelian case is in some sense the most important
(and the one that generalizes).

\begin{remark} 
We will generally write $KO_G, KU_G \in \GSpec$ for the $G$-spectra
representing $G$-equivariant $K$-theory. When the group is clear, we will
sometimes simply write $KO, KU$ (especially when we want to describe the
equivariant $K$-theory of a space). 
\end{remark}

\subsection{The case of a torus}

We begin with the (simpler) case of a torus.
We will need to use the existence of $\e{\infty}$-structures on
equivariant real and complex $K$-theory. 
These $\e{\infty}$-structures are established in work of Joachim
\cite{Jo04}, and appear in Schwede's theory of global spectra
\cite{Schwedeglobal}.
These results are also a consequence of forthcoming work of Lurie on elliptic
cohomology announced in \cite{Lur09b}.
% In addition, $G$-equivariant $KU_G \in \clg( \GSpec)$ has a $C_2$-action (given
% by complex conjugation) such that there is an equivariant map
% \[ KO_G \to KU_G,  \]
% where $KO_G$ is given trivial $C_2$-action. We will \emph{not} use the
% identification $KO_G \simeq (KU_G)^{hC_2}$; in fact, we will prove this below.

We begin with the following special case. Most of the ideas (if not the statement)
appear in \cite{GS14}, and the result was known to Greenlees-Shipley. 

\begin{thm} 
\label{KUtorus}
Let $T$ be a torus. Then the symmetric monoidal $\infty$-category $\md_{\gsp{T}}(KU_T)$ of modules (in
$T$-equivariant spectra $\gsp{T}$) over $KU_T$ is equivalent to $\md( i_T^*
KU_T)$.
\end{thm} 
\begin{proof} 
By the Thom isomorphism, $KU_T$ is what Greenlees-Shipley \cite{GS14} call
\emph{complex-stable}: that is, given a representation sphere $S^V$ (for a
\emph{complex} representation $V$ of the torus), we have an equivalence of
$KU_T$-modules $S^V
\wedge KU_T \simeq KU_T$. By \cite[Lem.~4.4]{GS14}, we are done. 
\end{proof} 

For the reader's convenience, we recall the method of argument in the case that $T = U(1)$. 
Recall that the $\infty$-category of $U(1)$-spectra is generated as a localizing
subcategory by the $U(1)/H_+$ as $H \leq U(1)$ ranges over the closed
subgroups. The only possibilities are $H = U(1)$ (in which case $U(1)/H_+$ is the
unit) or $H = \mu_n$ for some $n$, the group of $n$th roots of unity. In this
case, one considers the one-dimensional complex representation 
$V_n$ of $U(1)$ given by the character $z \mapsto z^n$. 
The unit sphere $S(V_n)_+$ gives precisely $(U(1)/\mu_n)_+$. 
The cofiber sequence $S(V_n)_+ \to S^0\to S^{V_n}$ now shows that the
$\infty$-category of $U(1)$-spectra is generated as a localizing subcategory by the
representation spheres $S^V$ for $V$ a complex representation of $U(1)$. When
one works over $KU_T $, though, the complex stability 
enables us to include only the unit. 

\subsection{The general case}

The purpose of this section is to give the proof of our main unipotence result
for equivariant complex $K$-theory.
\begin{thm} 
\label{KUunip}
Suppose $G$ is a compact, connected Lie group such that $\pi_1(G)$ is torsion-free.
Then
$\md_{\GSpec}(KU_G)$ is canonically equivalent, as a symmetric monoidal
$\infty$-category, to the $\infty$-category of module spectra over the
categorical fixed points $i_G^*KU_G$.
\end{thm}

 To appreciate the possible simplification this result brings about for studying 
 $\md_{\GSpec}(KU_G)$, we briefly remark on the structure of the (non-equivariant)
 $\mathbb{E}_\infty$-ring $A:=i_G^*KU_G$: it is even periodic with $\pi_1(A)=0$ and
 $\pi_0(A) = R(G)$. Landweber exactness shows that as a multiplicative cohomology theory,
 one has $A^*(-)=KU^*(-)\otimes_{\mathbb{Z}} R(G)$. It seems an interesting question to ask if
 the $\mathbb{E}_\infty$-ring $A$ can be built from $KU$ in a similarly transparent fashion. 

\Cref{KUunip} is equivalent, by Morita theory, to the assertion
that
$\md_{\GSpec}(KU_G)$ is generated, as a localizing subcategory, by the unit. 
When $G$ is a product of copies of $U(1)$ (i.e., a torus), we have already seen 
the proof of this (\Cref{KUtorus}). 
The general case proceeds by restriction to a maximal torus.

The key ingredient for the general case is given by:

\begin{lemma} 
\label{restorus}
Let $G$ be a compact connected Lie group with $\pi_1(G)$ torsion-free.
Let $T \subset G$ be a maximal torus and 
let $F = G/T $ be the flag variety of $G$.
Let $X$ be a finite $G$-cell complex. Then 
$KU_G^*(F) \simeq KU_T^*(\ast)$ is a free $R(G)[\beta_2^{\pm 1}] = KU_G^*(\ast)$-module, and 
the canonical map
\begin{equation}\label{kunnth} KU_G^*(X) \otimes_{KU_G^*(\ast)} KU_G^*(F) \to
KU_G^*(X \times
F)\cong KU_T^*(X)  \end{equation}
is an isomorphism.
\end{lemma} 

\begin{proof}
\Cref{restorus} follows by combining work of Hodgkin, Snaith, and McLeod. By 
\cite[Th. 3.6]{Sna72}, the construction in \cite{Hod75} of a K\"unneth spectral
sequence is relevant (i.e., converges to the desired limit) for $\pi_1(G)$ torsion-free if the natural map $R(G) \otimes_{R(T)} R(G) [\beta_2^{\pm
1}] \to KU_T^*(G/T)$ is an isomorphism. The main result of \cite{Mc78} shows
that this is in fact the case if $\pi_1(G)$ is torsion-free, so there is a K\"unneth spectral sequence with an edge map of the form appearing in \eqref{kunnth}. 
This implies that \eqref{kunnth} is an isomorphism once we know that the
representation ring $R(T)$ is free over $R(G)$; this is a theorem of Pittie
\cite[Thm.\ 1]{Pi70}. \end{proof}

The map \eqref{kunnth} can be rewritten as follows. As before, we let $i_G^*$
denote categorical fixed points $i_G^*\colon \gsp{G} \to \mathrm{Sp}$. The
equivariant $K$-theory of a finite $G$-cell complex $X$ is obtained as 
\[ KU_G^*(X) = \pi_{-*} i_G^* ( \mathbb{D}X_+ \wedge KU_G),  \]
where $X_+$ denotes the suspension spectrum of $X$ in $\gsp{G}$ and
$\mathbb{D}$ denotes Spanier-Whitehead duality.
We will need this in the following form. 

\begin{lemma} \label{keylemma}
Let $G$ be a compact connected Lie group with $\pi_1(G)$ torsion-free and let $T\subset G$ and $F=G/T$
be as above. Then for any $M \in \md_{\GSpec}(KU_G)$, the map 
\begin{equation} \label{Tmap} R(T) \otimes_{R(G)} \pi_* i_G^* M \to  \pi_*
i_{T}^* M  \simeq \pi_* i_G^*
\left( M  \wedge \mathbb{D}F_+ \right)\end{equation}
is an isomorphism.
\end{lemma} 

The last map in \eqref{Tmap} is an isomorphism for tautological reasons: $M \wedge \mathbb{D}
F_+$ is the coinduction of the restriction of $M$ to $\gsp{T}$ in view of the projection formula. 

\begin{proof} 
We observe that there is a natural map, since 
the right-hand side is linear over $R(T)$. It is a natural
transformation of homology theories in $\md_{\GSpec}(KU_{G })$, and
\Cref{restorus} implies that it is an isomorphism if $M$ is the
Spanier-Whitehead dual of the suspension
spectrum of a finite $G $-cell complex.
This implies that it is true in general, since the duals of suspension spectra of
finite $G $-cell complexes generate $\gsp{G }$ under colimits.
\end{proof}

\begin{proof}[Proof of \Cref{KUunip}]
Let $G$ be as hypothesized. If $G$ is  a torus, we are already
done. Let $M \in \md_{\GSpec}( KU_G)$. 
Suppose that $i_G^* M = 0$; we want to show that $M$ is itself contractible. If
we can prove this, then we will have proved \Cref{KUunip} because 
as $i_G^*(-)\simeq\mathrm{Hom}_{\md_{\GSpec}(KU_G)}(KU_G,-)$ we will then 
know that the compact unit $\mathbf{1}$ generates $\md_{\GSpec}(KU_G)$ as a localizing
subcategory.

Choose a maximal torus $T \subset G$.
By Lemma~\ref{keylemma}, we find that $\pi_* (i_{T}^* M )= 0$. 
In view of \Cref{KUtorus}, 
this implies that the \emph{restriction} of $M$ to $T$ (i.e., as an object of
$\md_{\gsp{T}}(KU_T)$) is contractible. 
Since restriction is symmetric
monoidal, it follows that for any $G$-space $X$, the restriction of $M \wedge
X_+$ to $\gsp{T}$ is contractible; in particular, for any such $X$, 
\[ i_{T}^* ( M \wedge X_+) = 0.  \]
But by Lemma~\ref{keylemma} again, this implies that
\[  i_G^* ( M \wedge X_+) = 0, \]
and since we had this for any $X$, we find that $M = 0$ itself, as the
Spanier-Whitehead duals of the finite $G$-cell complexes $X$ generate $\gsp{G}$
as a localizing subcategory. 
\end{proof}

\begin{remark} 
Our analysis relied on deep work of Hodgkin, Snaith, and McLeod. In the case when
$G$ is a product of unitary groups (which is the essential case for
nilpotence results), the results needed are much more
elementary. Namely, instead of \Cref{keylemma}, one can use: 
\begin{lemma}[{\cite[Prop.~3.9]{Seg68a}}]
\label{seglemma}
Let $X$ be a finite $G$-cell complex and let $V$ be any complex
$G$-representation. Let $\mathbb{P}(V)$ be the projectivization of $V$
considered as a $G$-space. Then 
$KU_G^*(\mathbb{P}(V))$ is a free $R(G)[\beta_2^{\pm 1}] = KU_G^*(\ast)$-module, and 
the map
\begin{equation}\label{kunnth2} KU_G^*(X) \otimes_{KU_G^*(\ast)} KU_G^*(
\mathbb{P}(V)) \to KU_G^*(X \times
\mathbb{P}(V))  \end{equation}
is an isomorphism.
\end{lemma} 

One can carry out the above strategy of proof, using \Cref{seglemma} to replace
a copy of $U(n)$ by a copy of $U(n-1) \times U(1)$. 
\end{remark}

Although our analysis relied on classical work on the equivariant K\"unneth
spectral sequence, the main result (\Cref{KUunip}) gives a new interpretation 
of this spectral sequence. Namely, let $G$ be a compact, connected Lie group
with $\pi_1(G)$ torsion-free. If $X$ is a finite $G$-cell complex, then one
has $KU_G^*(X) \simeq \pi_* i_G^* ( KU_G \wedge \mathbb{D}X_+)$. 
In particular, to every such $X$, we can associate the $KU_G$-module $ KU_G \wedge
\mathbb{D}X_+$ and its categorical fixed points $ M_X
:= i_G^*( KU_G \wedge \mathbb{D}X_+)
\in \md( i_G^* KU_G)$. The homotopy groups of the spectrum $M_X$ give
precisely the equivariant $K$-theory of $X$. 

Since we have seen a symmetric monoidal equivalence
$\md_{\GSpec}(KU_G) \simeq \md(i_G^* KU_G)$ (where the latter takes place in the world of
\emph{nonequivariant} spectra), 
it follows that the association $X \mapsto M_X \in \md( i_G^* KU_G)$ is symmetric
monoidal. The K\"unneth spectral sequence can thus be recovered as the
classical $\mathrm{Tor}$-spectral sequence for modules over the
(nonequivariant) $\e{\infty}$-ring spectrum  $i_G^* KU_G$.

Let $G$ be a general compact connected Lie group. In this generality, we do not know how to
describe the 
$\infty$-category $\md_{\GSpec}(KU_G)$. However, we note that: 

\begin{prop} 
\label{flagvarcompactgenerator}
If $G$ is connected, 
the flag variety
$KU_G \wedge (G/T)_+$ is a compact generator of $\md_{\GSpec}( KU_G)$. 
\end{prop} 
The argument presented here shows that the above result is a consequence of
Atiyah's ``holomorphic transfer'' \cite{Ati68a}. The identification of the
holomorphic transfer and a spectrum-level transfer as a consequence of index
theory is discussed in \cite[\S 4.3]{Cos87} for the unitary group.  
See also \cite{Nis78} for a discussion of these transfer maps. 
We have spelled out the details for the convenience of the reader. 
\begin{proof}
The key step is to show that the unit $KU_G$ is a retract of $KU_G \wedge
(G/T)_+$. 
Let $\tau$ denote the tangent bundle of the flag variety $F = G/T$ and let $F^{\tau}$
denote its Thom space, the latter considered as a pointed $G$-space. Since $G/T$ is a complex manifold, we have a Thom
isomorphism
\[ G/T_+ \wedge KU_G \simeq F^{\tau} \wedge KU_G \in \md_{\GSpec}(
KU_G).  \]
It suffices now to show that the unit is a retract of $F^{\tau} \wedge
KU_G$.

To see this, embed the
flag variety $F \subset W$ for $W$ a real $G$-representation.  
As a result, we have an embedding $\tau \subset TW \simeq W
\otimes_{\mathbb{R}} \mathbb{C}$ and a consequent Pontryagin-Thom collapse map
\[ S^{W \otimes_{\mathbb{R}}\mathbb{C}} \to F^{\tau + \nu} , \]
for $\nu$ the normal bundle of $\tau \subset TW$. 
After smashing with $K$-theory, we obtain a map $KU_G \to F^{\nu + \tau} \wedge
KU_G \simeq F^{\tau}  \wedge KU_G$ by the Thom isomorphism. 
We will show that there exists a map $F^{\tau} \to KU_G$ such that the induced
composite $KU_G \to F^{\tau} \wedge KU_G \to KU_G$ is  an equivalence. 
In other words, we will produce a class in $\widetilde{KU}_G^0(F^{\tau})$ whose
pullback to $\widetilde{KU}_G^0( S^{W \otimes_{\mathbb{R}} \mathbb{C}}) \simeq
R(G)$ is a
unit. 

Indeed, the pull-back map 
$\widetilde{KU}_G^0(F^{\tau}) \to \widetilde{KU}_G^0( S^{W \otimes_{\mathbb{R}}
\mathbb{C}})$ is given by the \emph{analytic index} by the Atiyah-Singer index
theorem \cite{ASi}. 
As a result, one has to produce a $G$-equivariant elliptic differential
operator (or complex) on $F$
whose index in $R(G)$ is one-dimensional. We can take the Dolbeaut complex of the complex
manifold $F$. By a special case of the Borel-Weil-Bott theorem, the coherent cohomology $H^*(F, \mathcal{O})$ is
one-dimensional and concentrated in degree zero (with trivial $G$-action). 
It follows that the associated class in $\widetilde{KU}_G^0(F^{\tau})$ (which
is the Thom class of the complex tangent bundle) has the desired property, and
we get the splitting. 

As a result, for any $KU_G$-module $M$, the natural map $\pi_* i_G^* M \to
\pi_* i_T^* M$ is canonically a (split) injection. We now leave it to the
reader to show, imitating the proof of \Cref{KUunip}, that if $M$ is such that $\pi_* i_T^* M = 0$, then $M$ itself is
contractible: in other words, the flag variety is a compact generator. 
\end{proof}

The above argument with the holomorphic transfer underscores the importance of
the flag variety in proving the above statements: in fact, 
the argument in \cite{Sna72} regarding the K\"unneth spectral sequence goes
through the holomorphic transfer too. 

\begin{remark} 
We can translate the proof of the main result of \cite{Sna72} into our
language, too. 
Suppose $\pi_1(G)$ is torsion-free, so that $R(T) \otimes_{R(G)}
R(T)[\beta_2^{\pm 1}] \simeq
KU_G^*(  G/T \times G/T)$ and $R(T)$ is free over $R(G)$ (by \cite{Mc78} and
\cite{Pi70}). 
We apply \Cref{whenequivdual}
now to the thick subcategory of $\md_{\GSpec}(KU_G)$ generated by the unit and the flag
variety (which is self-dual by the Wirthm\"uller isomorphism). It follows
from \Cref{whenequivdual} that
this thick subcategory is generated by the unit. However, 
\Cref{flagvarcompactgenerator} implies that this thick subcategory consists
precisely of the compact objects, so $\md_{\GSpec}(KU_G)$ is unipotent as
desired. 
\end{remark} 

\subsection{The Borel-completion}

We can Borel-complete \Cref{KUunip} to obtain a strengthening of our ``Koszul
duality'' result \Cref{basicunipex} in the case of (nonequivariant) $K$-theory.
We find: 
\begin{thm} 
\label{BEtorfree}
Let $G$ be a compact, connected Lie group with $\pi_1(G)$ torsion-free. Then
the $\infty$-category $\fun(BG, \md(KU))$ is unipotent and equivalent as a
symmetric monoidal $\infty$-category to
$KU$-complete modules over $F(BG_+, KU)$.
\end{thm} 
\begin{proof} 
Let $\hat{R}(G), \hat{R}(T)$ be the completions of 
the representation rings $R(G), R(T)$ at the respective augmentation ideals
$I_G \subset R(G), I_T \subset R(T)$. 
By the Atiyah-Segal completion theorem \cite{AtS69}, these give precisely
$\pi_0$ of $F(BG_{+}, KU) $ and $F(BT_+, KU)$. 
Note that the map $R(G) \to R(T)$ exhibits $R(T)$ as a finite module over
$R(G)$ by \cite[Prop.~3.2]{Seg68b}, and that the 
$I_T$-adic topology on $R(T)$ is equivalent to the $I_G$-adic one \cite[Cor.
3.9]{Seg68b}. 
In any event, we find that 
\begin{equation} \label{complbasechange}  \hat{R}(G) \otimes_{R(G)} R(T) \to
\hat{R}(T)  \end{equation}
is an isomorphism.

By \Cref{unipcriterion}, it suffices to show 
that the induced object $KU \wedge G_+\in \fun(BG,\md(KU))$ belongs to the thick subcategory
generated by the unit. 
Inducing upwards from $\fun(BT, \md(KU))$ (where we already know the result by
\Cref{polyEMSS}), we see that it belongs to the thick subcategory generated by $KU \wedge
(G/T)_+$, so it suffices to show that the flag variety $G/T_+$ belongs to the 
thick subcategory generated by the unit. 
Note first that the flag variety is self-dual in $\fun(BG, \md(KU))$ in view
of the Wirthm\"uller isomorphism and complex orientability. 
As a result, \Cref{whenequivdual} shows that 
it suffices to prove that the natural map
\begin{equation} \label{mapFF}(KU \wedge F_+)^{hG} \otimes_{F(BG_+, KU)} ( KU \wedge F_+)^{hG} \to
(KU \wedge F_+ \wedge F_+)^{hG}
\end{equation}
is an equivalence.
Indeed, we can then apply \Cref{whenequivdual} for $\mathcal{C} =
\md^\omega(F(BG_+, KU))$ the $\infty$-category of perfect $F(BG_+, KU)$-modules, and
$\mathcal{D}$ the thick subcategory of $\fun(BG, \md(KU))$ generated by the unit
and $KU \wedge F_+$; the result implies that $\mathcal{C} = \mathcal{D}$.

However, in view of the discussion in \eqref{complbasechange}, it is a consequence of the Atiyah-Segal completion theorem and
\Cref{restorus} that \eqref{mapFF} is an equivalence. 
\end{proof}

\subsection{Applications to nilpotence}

As before, we can obtain:
\begin{cor} 
\label{flagvarequiv}
Let $G$ be a compact, connected Lie group with torsion-free $\pi_1$ and let $T
\subset G$ be a maximal torus. Let $F = G/T$.
Then we have an equivalence
in $\md_{\GSpec}(KU_G)$,
\[ KU_G \wedge  \mathbb{D}F_+ \simeq \bigoplus_m KU_G \]
for some integer $m$.
\end{cor} 
\begin{proof} 
This follows from \Cref{KUunip} and the fact (due to \cite{Pi70}) that the $KU_{G}^*(F) \simeq
R(T)[t^{\pm 1}]$ is a
free module over $R(G)[t^{\pm 1}]$.
\end{proof}

Applying \Cref{BEtorfree}, one
obtains: 
\begin{cor} 
Let $G$ be a compact connected Lie group with torsion-free $\pi_1$ and let $F =
G/T$ be the flag variety as before. Then, as an object of $\fun(BG, \md(KU))$,
$F_+ \wedge KU$ is a direct sum of copies of the unit. 
\end{cor} 

Finally, as before we can obtain the nilpotence statement.
Again, the full strength of unipotence is not really necessary for this
argument: the freeness of the flag variety is equivalent to the projective
bundle formula in equivariant $K$-theory. 

\begin{cor} 
If $G$ is a finite group, then $KU_G \in \GSpec$ is nilpotent for the family of
abelian subgroups. 
\end{cor} 

\begin{proof} 
Embed $G \leq U(n)$ and consider the action of $G$ on $F = U(n)/T$ for $T
\subset U(n)$ a maximal torus. 
We have $\mathbb{D}F_+ \wedge KU_{U(n)} \simeq \bigoplus_1^{n!} KU_{U(n)}$ 
by \Cref{flagvarequiv}. Restricting down to $G$, we have $\Res^{U(n)}_G
KU_{U(n)} \simeq KU_G$ and we get
\[ \mathbb{D}F_+ \wedge KU_G \simeq \bigoplus_{1}^{n!} KU_G.  \]
Choosing a triangulation of $F$ with abelian stabilizers in $G$, we can
conclude the proof as before. 
\end{proof}

\section{Equivariant real $K$-theory}

Let $G$ be a compact Lie group.
In this section, we will analyze $G$-equivariant \emph{real} $K$-theory. 
Our main goal is to extend the results in the previous section
to $KO_G$-modules in $\GSpec$, as well as to develop a Galois descent picture for
equivariant $K$-theory. In particular, we will obtain an $\sF$-nilpotence result
for $KO_G$ for $G$ finite (for $\sF$ the family of abelian subgroups). 

Our main tool, which we will start with, 
is an 
equivariant version of the equivalence $KO \wedge
\Sigma^{-2} \mathbb{CP}^2 \simeq KU$.  This will enable us to ``descend''
(via thick subcategory arguments) many results for $KU_G$ to $KO_G$.
As a result, we will prove a similar unipotence result for $KO_G$-modules. 
When combined with techniques from \cite{MM15} for $G = U(n)$, we will be able
to prove that the \emph{equivariant} complexification map $KO_G \to KU_G$
is a faithful $C_2$-Galois extension of $\e{\infty}$-rings in $\gsp{U(n)}$ (which
we will then deduce for any compact Lie group $G$). 
The Galois picture was first developed nonequivariantly by Rognes \cite{Rog08}
and has numerous applications.

\subsection{Complexification in equivariant $K$-theory}

The key ingredient in the proof below of the equivariant version of Wood's theorem
(\Cref{equivwood}) is an analysis of the complexification map 
\[ \widetilde{KO}_G^*(\mathbb{CP}^2) \to \widetilde{KU}_G^*(\mathbb{CP}^2).  \]
This \emph{mostly} reduces to a purely non-equivariant calculation, since
$\mathbb{CP}^2$ is regarded here as a space with trivial $G$-action. 
First of all, by \cite[Prop.~2.2]{Seg68a}, we have a natural isomorphism
$\widetilde{KU}_G^*(\mathbb{CP}^2) \simeq R(G) \otimes
\widetilde{KU}^*(\mathbb{CP}^2)$. However, the picture is somewhat more
complicated for equivariant $KO$. 
In this subsection, we discuss the equivariant real and complex $K$-theory of
spaces with trivial $G$-action and give a complete analysis of the (equivariant)
complexification map.

\begin{definition} 
\label{RCH}
Given an irreducible representation $V$ of $G$
over $\mathbb{C}$, recall that there are three possibilities: 
\begin{enumerate}
\item The representation $V$ is not self dual as a $G$-representation over $\mathbb{C}$. 
In this case, the real representation $V|_{\mathbb{R}}$ underlying $V$ is irreducible, and
$\mathrm{End}_{G, \mathbb{R}}(V|_{\mathbb{R}}) \simeq \mathbb{C}$.
\item The underlying real representation $V|_{\mathbb{R}}$ is not irreducible as a real representation. Thus, $V$
contains an $\mathbb{R}$-subspace $V_{\mathbb{R}} \subset V$ which is
$G$-stable. One has $V_{\mathbb{R}} \otimes_{\mathbb{R}} \mathbb{C} \simeq V$.
As a complex representation, we have $V \simeq V^*$.
Moreover, $\mathrm{End}_{G}(V_{\mathbb{R}}) \simeq \mathbb{R}$.
\item The representation $V$ is self-dual as a $G$-representation over $\mathbb{C}$, but $V|_{\mathbb{R}}$ is
irreducible. In this case, $\mathrm{End}_{G, \mathbb{R}}(V|_{\mathbb{R}}) \simeq \mathbb{H}$. 
\end{enumerate}

Given an irreducible representation $W$ of $G$ over $\mathbb{R}$, there are
three corresponding possibilities: 
\begin{enumerate}
\item There is an isomorphism $\mathrm{End}_G(W) \simeq \mathbb{C}$, and $W$ arises as the restriction
of an irreducible complex representation $V$ of type 1.  
If it arises as the restriction of $V$, it equivalently arises as the
restriction of $V^*$. In this case, $W$ is called \emph{complex.}
\item There is an isomorphism $W \simeq V_{\mathbb{R}}$ for an irreducible $V$ of type 2. 
In this case, $\mathrm{End}_G(W) = \mathbb{R}$. In this case, $W$ is called
\emph{real.}
\item The representation $W $ is the restriction of some $V$ of type 3; in this case
$\mathrm{End}_G(W) = \mathbb{H}$ and $W$ is called \emph{quaternionic.}
\end{enumerate}
\end{definition} 

\newcommand{\ROR}{RO^{\mathbb{R}}}
\newcommand{\ROC}{RO^{\mathbb{C}}}
\newcommand{\ROH}{RO^{\mathbb{H}}}

\begin{definition} 
Given a compact Lie group, we let $RO(G)$ denote the Grothendieck group of real
finite-dimensional $G$-representations, as usual. 
Thus, $RO(G)$ is a free abelian group on the isomorphism classes of irreducible
$G$-representations over $\mathbb{R}$.

We let $\ROR(G) \subset RO(G)$ denote the subgroup spanned by the classes of
those representations which are real (in the sense of \Cref{RCH}),
$\ROC(G) \subset RO(G)$ the subgroup spanned by the classes of those
representations which are complex, and $\ROH(G)$ the subgroup spanned by those
that are quaternionic.
We thus obtain a decomposition $RO(G) = \ROR(G) \oplus \ROC(G) \oplus
\ROH(G)$.
\end{definition}

We now need the following result, in which $KSp$ denotes symplectic or quaternionic $K$-theory.
\begin{prop}[{\cite[p.133--134]{Seg68a}}] 
\label{equivKtrivial}
Let $X$ be a finite CW complex 
given trivial $G$-action. Then we have natural isomorphisms

\begin{gather}
KU_G^*(X) \simeq R(G) \otimes KU^*(X), \\
KO^*_G(X) \simeq \ROR(G) \otimes KO^*(X) \oplus 
\ROC(G) \otimes KU^*(X) \oplus \ROH(X) \otimes KSp^*(X).
\end{gather}
\end{prop} 

We note that the first isomorphism is $C_2$-equivariant for the complex
conjugation on all sides: in particular, including the $C_2$-action on $R(G)$.
The second decomposition arises as follows in degree zero (by suspending and using
periodicity, one gets the general case). Given a complex
$G$-representation $V$ and a complex vector bundle $\mathcal{W}$ on $X$, we form the
$G$-equivariant \emph{real} vector bundle $V \otimes_{\mathbb{C}} \mathcal{W}$. 
The other two summands are interpreted similarly.

We will need to describe the complexification map $KO^*_G(X) \to KU^*_G(X)$ with
respect to the above decompositions.
Without loss of generality (up to replacing $X$ by a suspension), we take $\ast = 0$.
\begin{enumerate}
\item On $RO^{\mathbb{R}}(G) \otimes KO^0(X)$, the complexification map  
behaves as follows: given a real $G$-representation $V$ and a real vector
bundle $\mathcal{W} \in KO^0(X)$, the class $[V] \otimes [\mathcal{W}]$ maps to
$[V_{\mathbb{C}}] \otimes [\mathcal{W}_{\mathbb{C}}]$.
\item On $RO^{\mathbb{C}}(G) \otimes KU^0(X)$, the complexification map behaves
as follows: given a complex $G$-representation $V$ and a complex vector
bundle $\mathcal{W} \in KU^0(X)$, the class $[V] \otimes [\mathcal{W}]$ maps to
$[V] \otimes [\mathcal{W}] + [V^*] \otimes [\mathcal{W}^*]$.
This follows from unwinding the definitions: one has to complexify the
$G$-equivariant real vector bundle (which is the restriction of a
$G$-equivariant complex vector bundle) $V \otimes_{\mathbb{C}} \mathcal{W}$. 
\item  On $RO^{\mathbb{H}}(G) \otimes KSp^0(X)$, the complexification map is
the most complicated. 
Given a quaternionic representation $V$ (where we interpret the $\mathbb{H}$
action on the \emph{right}) and a quaternionic vector bundle
$\mathcal{W}$ on $X$, the associated equivariant real vector bundle 
is $V \otimes_{\mathbb{H}} \mathcal{W}$. 

The complexification is therefore the equivariant complex vector bundle
$V_{\mathbb{C}} \otimes_{\mathbb{H} \otimes_{\mathbb{R}} \mathbb{C}}
\mathcal{W}_{\mathbb{C}}$.
Here $V_{\mathbb{C}}$ has a right action of $\mathbb{H}
\otimes_{\mathbb{R}}\mathbb{C} \simeq M_2(\mathbb{C})$ and
$\mathcal{W}_{\mathcal{C}}$ has a left action of $M_2(\mathbb{C})$.
In general, we recall that the category of left (resp. right) 
$M_2(\mathbb{C})$-modules is equivalent to the category of $\mathbb{C}$-vector
spaces, and the $M_2(\mathbb{C})$-linear tensor product between a right and
left $M_2(\mathbb{C})$-module corresponds to the $\mathbb{C}$-linear tensor product
between vector spaces. 

In particular, we can describe the equivariant $\mathbb{C}$-vector bundle 
$V_{\mathbb{C}} \otimes_{\mathbb{H} \otimes_{\mathbb{R}} \mathbb{C}}
\mathcal{W}_{\mathbb{C}}$
as follows. The $M_2(\mathbb{C})$-module with $G$-action $V_{\mathbb{C}}$
corresponds to a complex representation $V'$ of $G$, and the $M_2(\mathbb{C})$-bundle
$\mathcal{W}_{\mathbb{C}}$ over $X$ corresponds 
to a $\mathbb{C}$-vector bundle $\mathcal{W}'$ over $X$. 
It is easy to see that $V'$ satisfies $V' \simeq (V')^*$ and the underlying
real representation $V$ (which is still irreducible) is the underlying real
representation of $V'$. 
In any event, the complexification carries $[V] \otimes [\mathcal{W}] \mapsto
[V'] \otimes [\mathcal{W}']$.
\end{enumerate}

In order to make this useful, we will need to describe more explicitly the map $\phi\colon KSp^0(X) \to
KU^0(X)$ (which sends an $\mathbb{H}$-bundle $\mathcal{W}$ to the complex
vector bundle associated to the $M_2(\mathbb{C})$-bundle
$\mathcal{W}_{\mathbb{C}}$ via the Morita equivalence between
$M_2(\mathbb{C})$ and $\mathbb{C}$). In particular, we will need to know that 
the following diagram is commutative
\begin{equation} \label{complexificationdiag} \xymatrix{
KSp^0(X) \ar[d]^{\simeq}  \ar[r]^\phi & KU^0(X) \ar[d]^{\simeq} \\
KO^4(X)  \ar[r] & KU^4(X). 
}\end{equation}
Here, the vertical arrows are Bott periodicity 
and the bottom horizontal 
map is the usual complexification from real to complex $K$-theory. 

To see the commutativity of this diagram, we use the fact that  the
natural transformation $KSp^0 \to KU^0$ comes from a map of $KO$-module spectra
$KSp \simeq \Sigma^{4} KO \to KU$. 
Since this map, for $X = \ast$, carries the class of the $\mathbb{H}$-module $\mathbb{H}$ to the class of the
$\mathbb{C}$-module $\mathbb{C}^2$, it induces multiplication by 2 in $\pi_0$. 
Therefore, one sees that the induced map of $KO$-module spectra $\Sigma^4 KO
\to KU$ is the complexification map $\Sigma^4 KO \to \Sigma^4 KU$ followed by
Bott periodicity $\Sigma^4 KU \simeq KU$. 
\subsection{The equivariant Wood theorem}\label{sec:wood}

We recall first: 

\begin{thm}[Wood] One has an equivalence of $KO$-module spectra $KU \simeq KO
\wedge \Sigma^{-2} \mathbb{CP}^2$. \label{woodthm}
\end{thm} 

A proof of this result (as well as an analog for $TMF$) can be found in \cite{htmf}. 
In \cite{htmf}, the strategy is to take the $C_2$-action on $KU$ given by
complex conjugation and \emph{define} $KO \simeq KU^{hC_2}$. Wood's theorem is
proved by showing that $KU \wedge \Sigma^{-2} \mathbb{CP}^2$, as a spectrum
with $C_2$-action, is the coinduced object $F(C_{2+}, KU)$.

The main goal of this subsection is to prove an equivariant analog of
\Cref{woodthm}: 

\begin{thm} 
\label{equivwood}
Let $G$ be a compact Lie group. 
One has an equivalence of $KO_G$-modules in $\GSpec$, $KO_G \wedge \Sigma^{-2}
\mathbb{CP}^2 \simeq KU_G$.
\end{thm} 

We note that the $\mathbb{CP}^2$ that enters here is the ordinary one: that is,
it is treated as a pointed space with trivial $G$-action.
As in \Cref{woodthm}, the $C_2$-action on $KU_G$ will play an important role in
this analysis.
However, unlike in the setting of \Cref{woodthm}, we do \emph{not} want to assume 
an equivalence of the form $KO_G \simeq KU_G^{hC_2}$ in $\GSpec$; we will
instead prove this as a corollary. Our strategy instead is to build on the
known result (\Cref{woodthm}) and analyze directly the map 
\begin{equation}
\label{KOGKUGCP2}
KO_G \wedge \mathbb{D}(\mathbb{CP}^2) \to   KU_G \wedge
\mathbb{D}(\mathbb{CP}^2)  
\end{equation} 
in homotopy. 
It will be convenient to work with the (equivalent, up to a shift) Spanier-Whitehead duals,
since this amounts to understanding the map 
$\widetilde{KO}_G^*(\mathbb{CP}^2) \to \widetilde{KU}_G^*(\mathbb{CP}^2)$.
Our goal is to show that this map is injective with image the $C_2$-invariants
in 
$\widetilde{KU}_G^*(\mathbb{CP}^2)$.

To begin with the proof of \Cref{equivwood}, we
give a description 
of the equivariant real $K$-theory of $\mathbb{CP}^2$.
Our technical tool is the following:

\begin{prop} 
\label{KUKOGCP2map}
Let $G$ be a compact Lie group. Then for any $i$, $\widetilde{KO}^i_G(\mathbb{CP}^2)
\to \widetilde{KU}_G^i(\mathbb{CP}^2)$ is injective with image the
$C_2$-invariants in 
$ \widetilde{KU}_G^i(\mathbb{CP}^2)$.

\end{prop} 
\begin{proof} 

Consider the space $\mathbb{CP}^2$. 
We denote by $\psi$ the complex conjugation action.
Then we will need to use the following facts from \emph{nonequivariant}
$K$-theory:
\begin{enumerate}
\item $\widetilde{KO}^i(\mathbb{CP}^2) = \mathbb{Z}$  for $i $ even
and vanishes for $i$ odd. For $i$ even, we let $z_i \in
\widetilde{KO}^i(\mathbb{CP}^2)$ be a generator.
\item For $i$ even, $\widetilde{KU}^i(\mathbb{CP}^2) \simeq \mathbb{Z}^2$, generated by
classes $x_i, y_i$ with $\psi x_i = y_i$. For $i$ odd, $\widetilde{KU}^i(\mathbb{CP}^2)
= 0$.
\item  Under complexification, 
the map $\widetilde{KO}^i(\mathbb{CP}^2) \to \widetilde{KU}^i(\mathbb{CP}^2)$
is injective with image precisely the $C_2$-invariants.
That is, (up to multiplying by a sign) complexification maps $z_i \mapsto x_i + y_i$.
\end{enumerate}

Observe first that by 2., we know that  $\widetilde{KU}^*_G(\mathbb{CP}^2) \simeq R(G) \otimes
\widetilde{KU}^*(\mathbb{CP}^2)$ is a coinduced 
$C_2$-representation. As a $C_2$-representation, $R(G)$ is a direct sum of
copies of $\mathbb{Z}$ (one for each irreducible $G$-representation over
$\mathbb{C}$ of
type 1 or 3 in the sense of \Cref{RCH}) and $\mathbb{Z}[C_2]$ (one for each
irreducible representation of type 2). 

Now, we analyze the complexification map 
\[ c\colon  \widetilde{KO}_G^i(\mathbb{CP}^2) \to \widetilde{KU}_G^i(\mathbb{CP}^2),  \]
which is the effect of \eqref{KOGKUGCP2} in homotopy. 
For $i$ odd, both sides vanish, so we may consider only the case $i$ even. 
Using \Cref{equivKtrivial} applied to the suspensions of $\mathbb{CP}^2$, we obtain a decomposition of the left-hand-side
into three pieces that we analyze separately, using the discussion in the
previous subsection. 

\begin{enumerate}
\item On $\ROR(G) \otimes \widetilde{KO}^i(\mathbb{CP}^2)$,  
$c$ is an injection with image given by $R^{\mathrm{real}}(G) \otimes
\mathbb{Z}\left\{x_i + y_i\right\} \subset \widetilde{KU}_G^0( \mathbb{CP}^2)$.	
Here $R^{\mathrm{real}}(G)$ denotes the free abelian group on the complex
irreducible representations of type 2 in \Cref{RCH}. 
The map carries the class of $[V] \otimes z_i$ to $[V_{\mathbb{C}}] \otimes (x_i +
y_i)$.
\item On $\ROC(G) \otimes \widetilde{KU}^i( \mathbb{CP}^2)$, $c$
behaves as follows. If $[V] \in \ROC(G)$ is the class of an irreducible 
 representation of $G$ obtained as the restriction of a complex irreducible $W$, then $c$
acts on the classes $[V] \otimes x_i , [V] \otimes y_i$ as:
 \[ c( [V]\otimes x_i) = [W] \otimes x_i + [W^*] \otimes y_i, \quad 
c( [V] \otimes y_i) = [W] \otimes y_i + [W^*] \otimes x_i.
\]
\item 
Consider finally $RO^{\mathbb{H}}(G) \otimes \widetilde{KSp}^0(\mathbb{CP}^2)$.
Let $V$ be a $\mathbb{C}$-representation whose restriction to $\mathbb{R}$ is
an irreducible quaternionic representation, denoted the same.
Given a generator $w_i \in \widetilde{KSp}^i(\mathbb{CP}^2)$, the associated
class in $\widetilde{KU}^i(\mathbb{CP}^2)$ (obtained by complexification
together with Morita equivalence) is $x_i + y_i$, thanks to
\eqref{complexificationdiag}.
Therefore, we  have
\[ c( [V] \otimes w_i) = [V] \otimes (x_i + y_i).
\]
\end{enumerate}
From this, and the description of $R(G) $ as a $C_2$-representation, the proposition follows. 

\end{proof}

The proof of \Cref{equivwood} will require a little bookkeeping, and we begin with
some recollections on equivariant homotopy groups.

\newtheorem{rec}[thm]{Recollection}

\begin{rec}
\label{rec:free}
Let $G$ be a compact Lie group.
Let $A \in \mathrm{Alg}(\GSpec)$ be an associative algebra in $\GSpec$. 
For each $H \leq G$, we define $\pi_*^H(A) = \pi_* \hom_{\GSpec}(G/H_+,
A) = \pi_* i_H^*A$. Each $\pi_*^H(A)$ is a ring, and as
$H$ varies these rings are equipped with restriction homomorphisms
\[ \res^{H'}_H \colon \pi_*^{H'} A \to \pi_*^H A. \]
Given a
module $M \in \md_{\GSpec}(A)$, we can define the homotopy groups
$\{\pi_*^H(M)\}_{H \leq G}$, which come with restriction homomorphisms of
their own and form a module over
$\left\{\pi_*^H(A)\right\}$.

Note that the maps of $A$-modules $A \to M$ are classified by
the elements of $\pi_0^G(M)$. 
Suppose for instance that there exist elements $x_1, \dots, x_n \in \pi_0^G(M)$
such that for each $H \leq G$, the $\left\{\Res^G_{H} x_i\right\} \subset
\pi_0^H(M)$ form a basis for the $\pi_*^H(A)$-module $\pi_*^H(M)$. In this
case, we get maps
$x_i \colon A \to M$ which yield an equivalence of $A$-modules
\( A^n \simeq M.  \)
\end{rec}

\begin{proof}[Proof of \Cref{equivwood}]
Recall that $\widetilde{KU}_G^0(\mathbb{CP}^2) \simeq
R(G) \otimes_{\mathbb{Z} }\mathbb{Z}\left\{x,y\right\}$ for classes $x,y$ which are interchanged under
complex conjugation. 
We have  an equivalence of $KU_G$-modules 
\begin{equation} \label{CP2free} KU_G \vee KU_G \simeq KU_G \wedge
\mathbb{D}(\mathbb{CP}^2) \end{equation}
classified by the elements $x,y$: the elements $x,y$ produce a map, and it
is an equivalence because the restrictions of $x,y$ to $
\widetilde{KU}_H^*(\mathbb{CP}^2)$ form an $R(H)$-basis for any $H \leq G$. 
As a result, we can consider a map
\[ f \colon KU_G \wedge \mathbb{D}( \mathbb{CP}^2) \to KU_G,  \]
which, on homotopy, sends $x \mapsto 1$ and $y \mapsto 0$. 
For each subgroup $H \leq G$, one sees that the induced map 
\[ \left(\pi_*^H( KU_G \wedge \mathbb{D}( \mathbb{CP}^2))\right)^{C_2} 
\to \pi_*^H( KU_G \wedge \mathbb{D} (\mathbb{CP}^2)) \simeq
\widetilde{KU}^*_H(\mathbb{CP}^2) \to \pi_*^H(KU_G)
\]
is an isomorphism.
It follows easily that the composition
\[ KO_G \wedge \mathbb{D}( \mathbb{CP}^2) \to KU_G \wedge \mathbb{D}
(\mathbb{CP}^2) \simeq KU_G \vee KU_G \stackrel{f}{\to} KU_G  \]
is an equivalence, by comparing with \Cref{KUKOGCP2map}. 
\end{proof}

As a result, we can also obtain the homotopy fixed point relation between real
and complex $K$-theory, equivariantly.

\begin{cor} 
\label{KOKUhfp}
The natural map $KO_G \to KU_G^{h C_2}$ in $\GSpec$ is an equivalence. 
\end{cor} 
\begin{proof} 
It suffices to show that the natural map $KO_G \wedge \mathbb{D} (\mathbb{CP}^2)
\to 
(KU_G \wedge \mathbb{D}( \mathbb{CP}^2))^{hC_2}$ is an equivalence, because
the thick subcategory that $\Sigma^\infty \mathbb{CP}^2$ generates is all of
finite spectra by the nilpotence of $\eta$.
This in turn can be checked on $\pi_*^H$ for each subgroup $H \leq G$. 
Now the map 
\(  \pi_*^H(KO_G\wedge \mathbb{D} (\mathbb{CP}^2))
\to 
\pi_*^H(KU_G \wedge \mathbb{D}( \mathbb{CP}^2) )\)
is injective and has image the $C_2$-invariants in the target, by
\Cref{KUKOGCP2map}. However, we have
\[ \pi_*^H(  (KU_G \wedge \mathbb{D} (\mathbb{CP}^2))^{hC_2}) \simeq 
\pi_*^H( KU_G \wedge \mathbb{D} (\mathbb{CP}^2))^{C_2}
\]
because the homotopy fixed point spectral sequence degenerates: the
$C_2$-representation is induced.
\end{proof}

\subsection{Unipotence and nilpotence results}

Using \Cref{equivwood}, we will now prove an analog of \Cref{KUunip} for $KO$.

\begin{thm} 
\label{KOpi1torsionfree}
Suppose $G$ is a compact, connected Lie group such that $\pi_1(G)$ is torsion-free.
Then
$\md_{\GSpec}(KO_G)$ is canonically equivalent, as a symmetric monoidal
$\infty$-category, to the $\infty$-category of module spectra over the
categorical fixed points $i_G^*KO_G$.
\end{thm} 
\begin{proof} 
Suppose $M$ is a $KO_G$-module (in $\GSpec$) whose categorical fixed points are
trivial, i.e., $i_G^*M$ is contractible. We need to show that $M$ is
contractible; by \Cref{generateifcons}, this will suffice for the theorem. 
To see this, we observe that if $i_G^*(M)$ is contractible, then $i_G^*( M \wedge
\mathbb{D}(\mathbb{CP}^2))$ is contractible as well. However, $M \wedge
\mathbb{D}(\mathbb{CP}^2) $ is a
$KU_G$-module by \Cref{equivwood}, so by \Cref{KUunip}, it follows that $M \wedge \mathbb{D}(\mathbb{CP}^2)$ is contractible. Now, the thick subcategory that $\mathbb{D}(\mathbb{CP}^2)$ generates in finite spectra contains the sphere $S^0$ by the
nilpotence of $\eta$, so that $M$ is contractible itself.
\end{proof}

Using similar logic, one easily obtains:
\begin{prop} \label{FnilKOKU} Let $G$ be a finite group and $\sF$ a family of subgroups of $G$. Then $KO_G \in \sFNil$
if and only if $KU_G \in \sFNil$.
In particular, $KO_G$ is nilpotent for the family of abelian subgroups. 
\end{prop}

\begin{remark} 
In the sequel \cite{MNN15i} to this paper, we will give another approach to the
$\sF$-nilpotence of $KO_G$ using the \emph{spin orientation}. 
We will actually show that $KO_G$ is nilpotent for the family of \emph{cyclic}
subgroups.
\end{remark}

\subsection{The Galois picture}

Let $G$ be a compact Lie group.
In the theory of structured ring spectra, it is known by work of Rognes
\cite{Rog08} that the complexification map $KO \to KU$ 
is a faithful $C_2$-Galois extension: that is, it behaves like a $C_2$-torsor
in ordinary algebraic geometry. As a consequence, it is for instance possible
to carry out a form of \emph{Galois descent} along $KO \to KU$. 
In this final subsection, we prove that an analogous picture holds equivariantly.
We refer to \cite{galois} for preliminaries on Galois theory in a symmetric
monoidal, stable $\infty$-category.

\begin{thm} 
The natural map $KO_G \to KU_G$, together with the $C_2$-action on $KU_G$,
exhibits $KU_G$ as a faithful $C_2$-Galois extension of $KO_G$ in $\gsp{G}$. 
\end{thm} 
\begin{proof} 
% We have already seen the first assertion (\Cref{KOKUhfp}), so it suffices to show that
% $KO_G \to KU_G$ is a faithful $C_2$-Galois extension. 
Choose an embedding $G \leq U(n)$. In this case, one obtains a symmetric
monoidal, cocontinuous functor $\res^{U(n)}_G\colon \gsp{U(n)} \to \gsp{G}$ that carries
$KO_{U(n)}, KU_{U(n)} $ to $KO_G, KU_G$. As a result, it suffices to show that 
$KO_{U(n)}\to KU_{U(n)}$ is a faithful $C_2$-Galois extension in
$\gsp{U(n)}$.\footnote{Recall that faithful Galois extensions are preserved by
symmetric monoidal left adjoints.}

In this case, the equivalence
of \Cref{KOpi1torsionfree} shows that it suffices to prove that if $A =
i_{U(n)}^*
KU_{U(n)}$, then the natural map
\begin{equation} \label{Ahc2} A^{hC_2} \to A,  \end{equation}
exhibits $A$ as a faithful $C_2$-Galois extension of $A$ (in the category of non-equivariant spectra).We will prove this using the affineness machinery of \cite{MM15}; one can also argue directly using \Cref{woodthm}. 

Observe that $A$ is an even periodic $\e{\infty}$-ring, with $\pi_0(A) \simeq
R(U(n))$, with a $C_2$-action. 
It follows that we can associate to $A$ a \emph{formal group}  over $R(U(n))$,
given by $\mathrm{Spf} A^0( \mathbb{CP}^\infty) \simeq \mathrm{Spf}
R(U(n))[[x]]$. One sees that the
associated formal group law over $\spec R(U(n))$ is isomorphic to
$\widehat{\mathbb{G}}_m$ (i.e., $A^*(\mathbb{CP}^\infty) =
KU_{U(n)}^*(\mathbb{CP}^\infty) = R(U(n)) \hat{\otimes}
KU^*(\mathbb{CP}^\infty)$, etc.). 
One concludes that the unique map
of schemes
\[ \spec R(U(n)) \to \spec \mathbb{Z}   \]
is such that the 
formal group $\mathrm{Spf} A^0( \mathbb{CP}^\infty)$ over $\spec
R(U(n))$ is pulled back from $\widehat{\mathbb{G}}_m$ over $\spec \mathbb{Z}$. 

Now, $A$ has a $C_2$-action. Thus, $\spec R(U(n))$ has a $C_2$-action from complex conjugation, and the
formal group over $\mathrm{Spf} A^0( \mathbb{CP}^\infty)$ has one too.
In the language of \cite{MM15}, we obtain a diagram
\[ \spec R(U(n))/ C_2 \to  \spec \mathbb{Z} / C_2 \to M_{FG},   \]
for $M_{FG}$ the moduli stack of formal groups. Observe now that the first map 
$\spec R(U(n))/ C_2 \to  \spec \mathbb{Z} / C_2$ is affine (as the
$C_2$-quotient of the affine map $\spec R(U(n)) \to \spec \mathbb{Z}$) and the
map $( \spec \mathbb{Z})/C_2 \to M_{FG}$ is affine. Therefore, the composition 
$\spec R(U(n))/ C_2 \to  \spec \mathbb{Z} / C_2 \to M_{FG}$ is affine. By
\cite[Th. 5.8]{MM15} (see also \cite[\S 2.5]{MM15}), we obtain that 
the map \eqref{Ahc2} exhibits $A$ as a faithful $C_2$-Galois extension of
$A^{hC_2}$. 
\end{proof}

\begin{example} 
We briefly calculate the homotopy fixed point spectral sequence (HFPSS) for $\pi_*^G KO_G
\simeq \pi_*^G (KU_G)^{hC_2}$ as a modification of the (classical)
computation when $G = 1$. 
First of all, we know that $\pi_*^G KU_G \simeq R(G)[\beta_2^{\pm 1}]$ where
$|\beta_2| = 2$.  
The $C_2$-action on $R(G)$ is such that
\[ R(G) =  \bigoplus_V \mathbb{Z} \oplus \bigoplus_{V'} \mathbb{Z} \oplus
\bigoplus_{W} \mathbb{Z}[C_2]. \]
Here $V$ ranges over isomorphism classes of irreducible
$\mathbb{C}$-representations of type 2 (in the sense of \Cref{RCH}), $V'$ ranges over 
isomorphism classes of irreducible $\mathbb{C}$-representations of type 3.
Finally, 
$W$ ranges over isomorphism classes of complex representations of
type 1, up to the action $W \mapsto W^*$. Moreover, the $C_2$-action on the
Bott element is by the sign representation. 

It follows easily that, at the $E_2$-page the HFPSS for $\pi_*^G(KO_G)$ is a
direct sum of copies of the HFPSS for $\pi_*(KO)$, one for each $V$ and $V'$,
together with a sum of copies, one for each $W$, of $\mathbb{Z}[\beta_{2}^{\pm 1}]$ concentrated
on the 0-line. Observe that the last component is necessarily
given by permanent cycles because these classes come from $KO_G^*(\ast)$. 

We now analyze the remaining classes. Recall first that 
the $E_2$-page for $\pi_*(KO)$ is given by $\mathbb{Z}[\beta^{\pm 2},
\eta]/(2\eta)$ where $\beta$ has bidegree $(s,t)=(0,2)$ and $\eta$ has bidegree $(s, t) = (1, 2)$. 
It follows that the $E_2$-page for $\pi_*^G(KO_G)$, when we ignore the
contributions from $W$'s, is given by 
the free module over $\mathbb{Z}[\beta^{\pm 2},\eta]/(2\eta)\cdot [\mathbf{1}]$
\[ E_2^{*,*} = \mathbb{Z}[\beta^{\pm 2},
\eta]/(2\eta)\left\{ [V], [V']\right\} . \]
We conclude that $d_2=0$ holds for degree reasons.
In the spectral sequence for $\pi_*(KO)$, it is well-known that one has the
differential $d_3(\beta^2) = \eta^3$. This differential must happen here, too,
i.e. we have $d_3(\beta^2\cdot [\mathbf{1}])=\eta^3\cdot [\mathbf{1}]$.
The classes $[V]$ survive to $KO_G^0(\ast) = RO(G)$ and are
therefore permanent cycles and by multiplicativity one has $d_3( \beta^2\cdot [V])
= \eta^3\cdot [V]$. However, the classes $[V']$ do not survive to $KO_G^0(\ast)$ and
necessarily support differentials. Since $\beta^2\cdot [V']$ survives to
$KO_G^{-4}(\ast)$ (thanks to \eqref{complexificationdiag}), we find that
$\beta^2\cdot [V']$ is a permanent cycle. Using multiplicativity again, we get $d_3(
[V']) =  \eta^3 \beta^{-2}\cdot [V']$. This determines the entire spectral sequence
as being the direct sum of shifts by $0$ and $4$ of the $\pi_*(KO)$-HFPSS as
well as the degenerate components coming from complex representations. 
\end{example} 

\appendix

\bibliographystyle{alpha}
\bibliography{bib/biblio}

\end{document}